\documentclass[11pt,dvipsnames]{amsart}

\usepackage{amsmath,amssymb,amsthm,tikz, nicefrac, mathrsfs,upgreek} 
\RequirePackage{color}
\RequirePackage[colorlinks, urlcolor=my-blue,linkcolor=my-red,citecolor=my-green]{hyperref}
\definecolor{my-blue}{rgb}{0.0,0.0,0.6}
\definecolor{my-red}{rgb}{0.5,0.0,0.0}
\definecolor{my-green}{rgb}{0.0,0.5,0.0}
\definecolor{nicos-red}{rgb}{0.75,0.0,0.0}
\definecolor{light-gray}{gray}{0.6}
\definecolor{really-light-gray}{gray}{0.8}
\definecolor{sussexg}{rgb}{0.0,0.5,0.5}
\definecolor{sussexp}{rgb}{0.5,0.0,0.5}

\usepackage{dsfont} %  To avoid compliation problems with the indicator function

\usepackage{nicefrac}
\newtheorem{theorem}{\color{darkblue}{\sc Theorem}}[section]
\newtheorem{lemma}[theorem]{\color{darkblue} \sc Lemma}
\newtheorem{proposition}[theorem]{\color{darkblue} \sc Proposition}
\newtheorem{corollary}[theorem]{\color{darkblue} \sc Corollary}
\newtheorem{example}{\color{darkblue} \sc Example}
\newtheorem{conjecture}[theorem]{\color{darkblue} \sc Conjecture}
\newtheorem{question}[theorem]{\color{darkblue} \sc Question}
\newtheorem{assumption}[theorem]{\color{darkblue} \bf Assumption}
\newtheorem{definition}[theorem]{\color{darkblue} \it Definition}
\numberwithin{equation}{section}
\theoremstyle{remark}
\newtheorem{remark}[theorem]{\color{darkblue} Remark}

\newcommand{\be}{\begin{equation}}
\newcommand{\ee}{\end{equation}}

\providecommand{\abs}[1]{\vert#1\vert}

\newcommand{\fl}[1]{\lfloor{#1}\rfloor} 
\newcommand{\ce}[1]{\lceil{#1}\rceil}

\def\sD{\mathscr{D}}

\newcommand{\eup}{\text{\normalfont e}}

\def\bN{\mathbb{N}}
\def\bP{\mathbb{P}}

\def\bZ{\mathbb{Z}}
\def\cZ{\mathcal{Z}}
\def\bT{\mathbb{T}}

\def\cL{\mathcal{L}}

\def\cJ{J} % current
\def\om{\omega}
\def\e{\varepsilon}

 \def\Z{\bZ}

\def\N{\bN}
\def\P{\bP}

\def \sM{\mathscr M}

\def\sS{\mathcal{S}}

\usepackage{stackengine}

\newcommand{\Er}{E}

\newcommand{\dif}{\textup{d}}
\newcommand{\GW}{\textup{GW}}

\def\P{\bP} %% environment measure 

\definecolor{partcolor1}{rgb}{0.0,0.5,0.0}
\definecolor{partcolor2}{rgb}{0.0,0.5,0.0}

\definecolor{darkgreen}{rgb}{0.0,0.5,0.0}
\definecolor{darkblue}{rgb}{0.5,0.1,0.5}
\definecolor{nicosred}{rgb}{0.65,0.1,0.1}
\definecolor{light-gray}{gray}{0.7}
\allowdisplaybreaks[1]

\RequirePackage{datetime} %comment out when arxiving or submitting
\allowdisplaybreaks
\begin{document}
\usdate
\title[TASEP on trees]
{The TASEP on Galton--Watson trees}
\author{Nina Gantert}
\address{Nina Gantert, Technical University of Munich, Germany}
\email{nina.gantert@tum.de}
\author{Nicos Georgiou}
\address{Nicos Georgiou, University of Sussex, UK.}
\email{N.Georgiou@sussex.ac.uk}
\author{Dominik Schmid}
\address{Dominik Schmid, Technical University of Munich, Germany}
\email{dominik.schmid@tum.de}
\date{\today}
\begin{abstract}
We study the totally asymmetric simple exclusion process (TASEP) on trees where particles are generated at the root. Particles can only jump away from the root, and they jump from $x$ to $y$ at rate $r_{x,y}$ provided $y$ is empty. Starting from the all empty initial condition, we show that the distribution of the configuration at time $t$ converges to an equilibrium. We study the current and give conditions on the transition rates such that the current is of linear order or such that there is zero current, i.e.\ the particles block each other.
 A key step, which is of independent interest, is to bound the first generation at which the particle trajectories of the first $n$ particles decouple.
\end{abstract}
\maketitle
\vspace*{-0.35cm}  
\section{Introduction} 

The one-dimensional totally asymmetric simple exclusion process (TASEP) is among the most studied particle systems. It is a classical model which describes particle movements or traffic jams, studied by scientists from statistical mechanics, probability and combinatorics over several decades.  The model is simple but shows a variety of phase transitions and phenomena such as the formation of shocks \cite{FF:ShockFluctuations, L:Book2}.  It can be briefly described as follows. A set of indistinguishable particles are individually placed on distinct integer sites. Each site is endowed with a Poisson clock, independently of all others, which rings at rate 1. Should a particle occupy a given site, the particle attempts to jump one unit to the right when the site clock rings, and the jump is performed if and only if the target site is unoccupied, otherwise it is suppressed. This last condition is the exclusion rule.  One-dimensional TASEP is only a particular example of an exclusion process, with a degenerate jump kernel on $\Z\times \Z$ given by $p(x,x+1) = 1$ for all $x \in \Z$. When different jump kernels are considered, exclusion processes can be defined on any graph, including higher dimensional lattices or trees and they have also been studied extensively; see for example \cite{L:Book2}. 

In this article we define the TASEP on (directed) rooted trees. This way the particle system retains the total asymmetry of its one-dimensional analogue, while having more space to explore. Figure \ref{fig:snap} shows a snapshot of the evolution.  In our setup, particles jump only in the direction pointing away from the root under the exclusion rule and choose their target site according to some jump kernel, that puts mass only on the children of their current location. In addition, we create particles at the root at a constant rate through a reservoir. Our underlying tree may be random as long as it doesn't have leaves so particles cannot be eternally trapped. We will restrict our attention to TASEP on supercritical Galton--Watson trees without leaves, including the special case of regular trees. Moreover, we will assume that the tree is initially empty.  

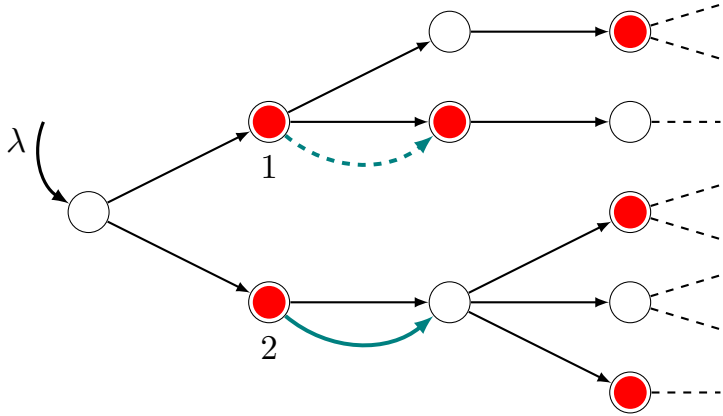
\begin{figure}[h]\label{fig:snap}
\begin{tikzpicture}[scale=1.2,  >=latex]

   \node[shape=circle,scale=1.5,draw] (A) at (0,0){} ;
 	\node[shape=circle,scale=1.5,draw] (B1) at (2,1){} ;
 	\node[shape=circle,scale=1.5,draw] (B2) at (2,-1){} ;
	\node[shape=circle,scale=1.5,draw] (C1) at (4,2) {};
	\node[shape=circle,scale=1.5,draw] (C2) at (4,1) {};
	\node[shape=circle,scale=1.5,draw] (C3) at (4,-1) {};
	
	\node[shape=circle,scale=1.5,draw] (D1) at (6,2) {};
	\node[shape=circle,scale=1.5,draw] (D2) at (6,1) {};
	\node[shape=circle,scale=1.5,draw] (D3) at (6,0) {};	
	\node[shape=circle,scale=1.5,draw] (D4) at (6,-1) {};
	\node[shape=circle,scale=1.5,draw] (D5) at (6,-2) {};		

			\draw[thick, ->] (A) to (B1);	
			\draw[thick,->] (A) to (B2);	
			\draw[thick,->] (B1) to (C1);		
			\draw[thick,->] (B1) to (C2);			
			\draw[thick,->] (B2) to (C3);			
			\draw[thick,->] (C1) to (D1);	
			\draw[thick,->] (C2) to (D2);			
			\draw[thick,->] (C3) to (D3);				
			\draw[thick,->] (C3) to (D4);	
			\draw[thick,->] (C3) to (D5);	
						
			\draw[thick,dashed] (D1) to (7,2.3);		
			\draw[thick,dashed] (D1) to (7,1.7);	

			\draw[thick,dashed] (D2) to (7,1);	
				
			\draw[thick,dashed] (D3) to (7,0.3);	
			\draw[thick,dashed] (D3) to (7,-0.3);	
			
			\draw[thick,dashed] (D4) to (7,-0.7);		
			\draw[thick,dashed] (D4) to (7,-1.3);	
			
			\draw[thick,dashed] (D5) to (7,-2);

 \node[shape=circle, scale=1.2, fill=red] (A1) at (B1){} ;
 \node[shape=circle,scale=1.2, fill = red] (A2) at (C2){} ;
\node[shape=circle,scale=1.2, fill = red] (A3) at (D1){} ;
\node[shape=circle,scale=1.2,  fill = red] (A4) at (D3){} ;
\node[shape=circle,scale=1.2,  fill = red] (A5) at (D5){} ;
\node[shape=circle,scale=1.2, fill= red] (A6) at (B2){} ;

\draw[thick,->,line width=1.2pt] (-0.5,1) to [bend left,out=-40, in=220] (A);	
\node[scale=1.2] (name) at (-0.8,0.8){$\lambda$};
\draw [->,line width=1.5pt, dashed, sussexg] (B1) to [bend right,out=-40,in=220] (C2);
\draw [->,line width=1.5pt, sussexg] (B2) to [bend right,out=-40,in=220] (C3);

\node[scale=1.2] (name) at (2,0.5){$1$};
\node[scale=1.2] (name) at (2,-1.5){$2$};
\end{tikzpicture}
\caption{Snapshot of a Tree-TASEP evolution. Particles enter at the root at rate $\lambda$ and then move down the tree, i.e. their distance from the root can only grow. They attempt a jump when the Poisson clock of an edge in front of them rings and the target site will be the child associated to the edge. The jump is suppressed if the target site is occupied (e.g.\ look at particle 1 attempting to jump at the occupied child) otherwise the jump is performed (e.g.\ particle 2).}
\end{figure}

Ideas to investigate the TASEP on trees can already be found in the physics literature as a natural way to describe transport on irregular structures, like blood, air or water circulations system; see \cite{BM:ASEPtrees, MWE:TASEPregTree, SFR:TASEPnetworks}. Exclusion processes on trees, but with no forbidden directions, were studied when the particles perform symmetric simple random walks; see \cite{CCGS:SpeedTree,GS:SpeedGalton}. 

One-dimensional TASEP provided an early connection between interacting particle systems and last passage percolation (LPP) on the two dimensional lattice, in an i.i.d.~ exponential environment. Viewing the particle system as queues in series, one can utilize Burke's theorem to find a family of invariant LPP models;  see \cite{balazs2006cube}. These models can be exploited to obtain, for example,  sharp variance bounds for last passage times. Burke-type theorems usually imply that the model in question is an integrable example of the KPZ universality class; see \cite{C:KPZReview} for an overview and articles \cite{Bar-Cor-15-, ciech2019order, Cor-Sep-She-14-, OCo-Ort-14b, Sep-12-corr} for other lattice examples having Burke's property. In particular, the exponential corner growth model and the one-dimensional TASEP, which are linked  through specific initial conditions and a height function representation, provably exhibit the correct scalings and Tracy-Widom weak limits associated with the KPZ class \cite{Joh-00}. Recently, it was shown that for a large class of initial conditions, TASEP converges to the KPZ fixed point \cite{matetski2016kpz}.

Coupling the TASEP to a growth model can be done via the current (or aggregated current) of the particle system. The current states how many particles pass through a certain site (or generation) by a given time. Our interests are two-fold. On one hand, we fix a time window and we want to know the current across a given generation by that time. The dual question is to fix a generation window and see how many particles occupy sites in there, by a given time. We study both of these questions. 

Finally, we investigate the law of the process in a finite region for large times to derive properties of the limiting equilibrium measures. 
An important observation is that once two particles are on distinct branches of the tree, they do not effect the transitions of each other. We make use of this observation by locating where the particle trajectories disentangle and the particles start to move independently. Quantifying the location of disentanglement is a key step in our analysis. The proof utilizes combinatorial, geometric and probabilistic arguments. 

In the next subsection we give a formal introduction to the TASEP on trees and present our results on the disentanglement, the current and the large time behaviour of the particles. 
Our main results are Theorem~\ref{thm:DisentanglementGWT}, Theorem \ref{thm:ConvergenceFlow}, Theorem~\ref{thm:LinearCurrentL}, 
Theorem \ref{thm:CurrentfixedT}, Lemma \ref{lem:InvarianceFlow}, and Lemma \ref{lem:DominationSuperflow}.
%In Subsection \ref{ssec:parallel} we note further parallels between TASEP on the integers and on trees.  

\subsection{Model and results}  \label{sec:ModelResults}
\subsubsection{TASEP on trees.}
We will work with Galton--Watson trees; see \cite[Chapter 4]{LP:ProbOnTrees} for a general introduction. Let $\bT=(V,\Er,o)$ be an infinite, locally finite, rooted tree with directed edges pointing away from the root $o$, and let $\mathcal{T}$ be the set of all such trees. 
\begin{definition}\label{def:Supercritical}
Let $\mu$ be a distribution on $\N_0=\N \cup \{0\}$ and set $p_\ell:=\mu(\ell)$ for all $\ell \in \N_0$. A \textbf{Galton--Watson tree} with offspring distribution $\mu$ is a tree in $\mathcal{T}$ sampled as follows. We start with the root $o$ and draw a number of children according to $\mu$. Then for each child, we again draw a number of children according to $\mu$ independently, and iterate.
All edges in the tree are directed edges from parents to their respective children.
\end{definition}  For the remainder of the paper, we assume that all Galton--Watson trees are \textbf{supercritical} and \textbf{without leaves}, i.e. \begin{equation}\label{def:expectedOffspings}
\mathfrak{m}:=\sum_{\ell\geq 0}\ell p_{\ell} \in (1,\infty) \qquad \text{and} \qquad p_0=0 \, .
\end{equation}
Note that the Galton--Watson branching process with offspring distribution $\mu$ induces a probability measure $\GW$ on $\mathcal{T}$; see \cite[Chapter 4]{LP:ProbOnTrees}. This includes the special case of regular trees when $\mu$ is a Dirac measure. \\

Next, we fix a tree $\bT=(V,\Er,o) \in \mathcal{T}$ drawn according to $\GW$. On this tree $\bT$, the \textbf{totally asymmetric simple exclusion process (TASEP)} $(\eta_t)_{t \geq 0}$ with a reservoir of intensity $\lambda > 0$ and transition rates $(r_{x,y})_{(x,y) \in \Er}$ is given as follows.
A particle at site $x$ tries to move to $y$ at rate $r_{x,y}$ provided that $(x,y) \in \Er$. However, this move is performed if and only if the target is a vacant site. Moreover, we place a particle at the root at rate $\lambda$ whenever the root is empty. We will choose the transition rates $(r_{x,y})_{(x,y) \in \Er}$ such that $(\eta_t)_{t \geq 0}$
is a Feller process; see \cite{L:interacting-particle} for an introduction. More precisely, $(\eta_t)_{t \geq 0}$ will be the Feller process on the state space $\{ 0,1\}^V$ with generator
\begin{equation}
\cL f(\eta) =  \lambda(1-\eta(o))[f(\eta^{o}) - f(\eta)]  + \sum_{(x,y) \in \Er}  r_{x,y}(1 - \eta(y))\eta(x)[f(\eta^{x,y}) - f(\eta)] 
\end{equation} for all cylinder functions $f$. Here, we use the standard notation
\be \label{eq:etaxy}
\eta^{x,y} (z) = \begin{cases} 
 \eta (z) & \textrm{ for } z \neq x,y\, ,\\
 \eta(x) &  \textrm{ for } z = y\, ,\\
 \eta(y) &  \textrm{ for } z = x\, ,
 \end{cases}
 \quad
 \text{and}
 \quad
 \eta^{x} (z) = \begin{cases} 
 \eta (z) & \textrm{ for } z \neq x\, ,\\
1-\eta(z) &  \textrm{ for } z = x\, ,
 \end{cases}
\ee
to denote swapping and flipping of values in a configuration $\eta \in \{0,1\}^V$ at sites $x,y \in V$.  \\

The following statement gives a sufficient criterion on the transition rates such that the totally asymmetric simple exclusion process on $\bT$ is indeed a Feller process.
\begin{proposition}[c.f.\ Proposition A.1 in \cite{GS:SpeedGalton}]
Assume that for $\GW$-almost every tree in $\mathcal{T}$, the transition rates $(r_{x,y})$ are uniformly bounded from above. Then for $\GW$-almost every tree $\bT$, the TASEP on $\bT$ is a Feller process.
\end{proposition}
For a tree $\bT \in \mathcal{T}$, let $P_{\bT}$ denote the law of the TASEP on $\bT$. Furthermore, we set $$\P= \GW \times P_{\bT}$$ to be the semi-direct product where we first choose a tree $\bT\in \mathcal{T}$ according to $\GW$ and then perform the TASEP on $\bT$. 
For $x\in V$, let $\abs{x}$ denote the shortest path distance to the root.
We set
 \begin{equation}\label{def:generation}
\cZ_{\ell}:=\{ x \in V \colon \abs{x} = \ell\}
\end{equation}
and we will refer to $\cZ_{\ell}$ as the
 $\ell^{\text{th}}$ \textbf{generation} of the tree, for $\ell \in \N_0$. \\
 
 Throughout this article, we will consider the $d$-regular tree for some $d\geq 3$ with common rates per generation as an example. In this case, the offspring distribution $\mu$ is the Dirac measure on $d-1$ and we  let the rates $(r_{x,y})_{(x,y)\in \Er}$ satisfy
 \begin{equation}\label{eq:homo}
 r_{x,y} = r_{x^{\prime},y^{\prime}}
 \end{equation} for all $(x,y),(x^{\prime},y^{\prime}) \in \Er$ with $|x|=|x^{\prime}|$; see Figure \ref{fig:RegularTreeTASEP} for $d = 3$ and $r_{x,y}= 2^{-|x|-1}$. 
We say the rates are \textbf{homogeneous} on the $d$-regular tree if 
\begin{equation} \label{eq:RunningExample}
r_{x,y}=(d-1)^{-|x|-1}.
 \end{equation}
% A visualization is given in Figure \ref{fig:RegularTreeTASEP}.

\begin{figure}[t]
\centering
\begin{tikzpicture}[scale=0.9]

  \node[shape=circle,scale=1.5,draw, line width=2pt] (A) at (0,-1){} ;
 	\node[shape=circle,scale=1.5,draw] (B) at (2,-1){} ;
 	\node[shape=circle,scale=1.5,draw] (-B) at (-2,-1){} ;
	
 	\node[shape=circle,scale=1.5,draw] (D) at (3.3,0){} ;
	\node[shape=circle,scale=1.5,draw] (-D) at (-3.3,0) {};
	\node[shape=circle,scale=1.5,draw] (E) at (3.3,-2){} ;
	\node[shape=circle,scale=1.5,draw] (-E) at (-3.3,-2) {};

		\node[scale=1.2] (test1) at (2.75,-1.2){$\frac{1}{4}$};
		\node[scale=1.2] (test2) at (2.55,-0.15){$\frac{1}{4}$};		
		
		\node[scale=1.2] (test3) at (-2.75,-1.2){$\frac{1}{4}$};
		\node[scale=1.2] (test4) at (-2.55,-0.15){$\frac{1}{4}$};		
		
		\node[scale=1.2] (test5) at (0.9,-0.58){$\frac{1}{2}$};
		\node[scale=1.2] (test6) at (-0.9,-0.58){$\frac{1}{2}$};

			\draw[thick,->] (A) to (-B);	
			\draw[thick,->] (A) to (B);	
		
			\draw[thick,->] (B) to (D);		
			\draw[thick,->] (-B) to (-D);		
			\draw[thick,->] (B) to (E);		
			\draw[thick,->] (-B) to (-E);

			\draw[thick,dashed] (E) to (4.3,-2);		
			\draw[thick,dashed] (-E) to (-4.3,-2);		
			\draw[thick,dashed] (E) to (3.3,-3);		
			\draw[thick,dashed] (-E) to (-3.3,-3);		
			
			\draw[thick,dashed] (D) to (4.3,0);		
			\draw[thick,dashed] (-D) to (-4.3,0);		
			\draw[thick,dashed] (D) to (3.3,1);		
			\draw[thick,dashed] (-D) to (-3.3,1);

			\node[shape=circle,scale=1.2,fill,red] (AC) at (E){} ;
			\node[shape=circle,scale=1.2,fill,red] (AC) at (-B){} ;

\draw[thick,->,line width=1.2pt] (0.5,0.5) to [bend right,in=220,out=-40] (A);	

			\node[scale=1.2] (name) at (0.2,-0.1){$\lambda$};

\end{tikzpicture}
\caption[The $3$-regular tree with a flow rule.]{\label{fig:RegularTreeTASEP}The $3$-regular (or binary)  tree satisfying a flow rule with $r_{j}^{\min} =  r_{j}^{\max}=2^{-j-1}$.}	\end{figure}
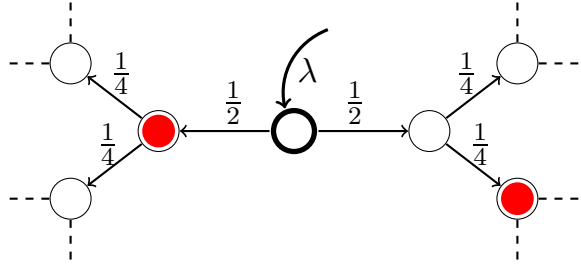

\subsubsection{Conditions on rates}  In the following, let the rates be bounded uniformly from above for $\GW$-almost every Galton--Watson tree, and let the tree be initially empty. We start with an upper bound on the first generation at which the first $n$ particles are located in different branches of the tree, and hence behave like independent random walks. Throughout this section, we will impose the following  two conditions on the transition rates.  Our first assumption on $(r_{x,y})$ is a non-degeneracy condition, which ensures that the particle system can in principle explore the whole tree.

\begin{assumption}[Uniform Ellipticity (UE)]\label{def:UniformElliptic}  
The transition rates on $\bT$ are \textbf{uniformly elliptic}, i.e.\ there exists an $\varepsilon \in (0, 1]$ such that
\begin{equation*}
\inf \left\{ \frac{r_{x,y}}{r_{x,z}} \colon (x,y),(x,z) \in \Er \right\}\geq \varepsilon \, .
\end{equation*}
\end{assumption}
Note that \hyperref[def:UniformElliptic]{\normalfont(UE)}  guarantees that the first $n$ particles will eventually move on different subtrees of $\bT$ and behave as independent random walks after a certain generation; see Proposition \ref{lem:Early-sep}. 
%In order to give an  upper bound on this generation, we require an additional assumption on the rates.
%\begin{remark}
%Condition \hyperref[def:UniformElliptic]{\normalfont(UE)} implies that the jump probability $p(x,y)$ from $x$ to a child $y$ satisfies the bounds
%\[\frac{\e}{d_x} \le \frac{ r_{x,y}}{d_x \max_{z \in \cC(x)} r_{x,z}} \le p(x,y) = \frac{r_{x,y}}{\sum_{z \in \cC(x)} r_{x,z}} \le \frac{ r_{x,y}}{d_x \min_{z \in \cC(x)} r_{x,z}} \le \frac{1}{d_x \e}.\]
%So \hyperref[def:UniformElliptic]{\normalfont(UE)} actually controls the transitions from above and below, together with the local degree that may be unbounded.
%\end{remark} 
 To state our next assumption, we define
 \begin{align}\label{def:minandmax}
	r^{\min}_\ell &:= \min\{r_{x,y} \colon x \in \cZ_\ell,  y \in \cZ_{\ell+1}, (x,y)\in \Er\} \nonumber \\
	r^{\max}_\ell &:= \max\{r_{x,y} \colon x \in \cZ_\ell,  y \in \cZ_{\ell+1}, (x,y)\in \Er\} 
\end{align} to be the minimal and maximal transition rates in generation $\ell$ for all $\ell \in \N_0$. The following
assumption guarantees that the rates are not decaying too fast, which may cause certain branches of the tree to  become blocked for the particles.
\begin{assumption}[Exponential decay (ED)]\label{def:ExponentialScaling} 
The transition rates \textbf{decay at most 
exponentially} fast, i.e.\ 
there exist constants $c_{\textup{low}},\kappa >0 $ such that for all $\ell \geq 0$
\begin{equation*}
 r_{\ell}^{\min} \geq \kappa \exp(-c_{\textup{low}} \ell) \, .
\end{equation*} 
\end{assumption}

At this point we want already to keep four quintessential examples in mind, that we will use to highlight the results and to show different regimes of behaviour. They are all on the $d$-regular tree which can be viewed as a Galton-Watson tree with $\mu \sim \delta_{d-1}$ and the rates are equal across their generation, making the tree endowed with the rates a spherically symmetric object. 

\begin{example}[Uniform ellipticity and Exponential decay in four examples] Consider the following four archetypes of rates on the $d$-regular tree for some $d \geq 3$:
	\begin{enumerate}
		\item[(C)] ({\bf Constant rates})  $r_{x,y} = 1$.  Assumption  \hyperref[def:UniformElliptic]{\normalfont(UE)} is satisfied with $\e = 1$ and  \hyperref[def:ExponentialScaling]{\normalfont(ED)} is satisfied with $\kappa = 1$ and any $c_{\textup{low}} >0$.
		\item[(E)] ({\bf Exponentially decaying, homogeneous, rates}) $r_{x,y} = r_{|x|} = (d-1)^{-|x|-1}$. Assumption  \hyperref[def:UniformElliptic]{\normalfont(UE)} is satisfied with $\e = 1$ and  \hyperref[def:ExponentialScaling]{\normalfont(ED)} is satisfied with $\kappa = (d-1)^{-1}$ and $c_{\textup{low}} = \log(d-1)$. 
		\item[(S)] ({\bf Slow rates}) $r_{x,y}  = (d-1)^{-|x|-1}g(|x|)$ where $g(s) \to 0 $ as $s \to \infty$ at most exponentially fast.  Assumption  \hyperref[def:UniformElliptic]{\normalfont(UE)} is satisfied with $\e = 1$ and  \hyperref[def:ExponentialScaling]{\normalfont(ED)} is satisfied for $\kappa>0$ and $c_{\textup{low}}>0$ sufficiently small.
		\item[(P)] ({\bf Polynomially decaying rates of power $p$}) $r_{x,y} = (|x|+1)^{-p}$ where $p>0$.  Assumption  \hyperref[def:UniformElliptic]{\normalfont(UE)} is satisfied with $\e = 1$ and  \hyperref[def:ExponentialScaling]{\normalfont(ED)} is satisfied with $\kappa = 1$ and $c_{\textup{low}} =  p \max_{|x|} \frac{\log(|x| +1)}{|x|} = p\log2$.
	\end{enumerate}
\end{example}

Equipped with these two assumptions, we will now introduce some notation to state our main results. In the following, we let
\begin{equation}\label{def:MinimalDegreeConditionedOffspring}
d_{\min}:=\min\{i \colon p_i >0\} \qquad \tilde{\mathfrak m}: = \left(\sum_{k=2}^\infty p_k\right)^{-1}\sum_{k=2}^\infty k p_k
\end{equation}
be the minimal number of offspring and the mean number of offspring when conditioning on having at least two offspring, respectively. Let 
\begin{equation}\label{def:cmu}
c_{o} := \begin{cases} (5+\log_2\tilde{\mathfrak m})(\log(1+p_1)-\log({2p_1}))^{-1} &\text{ if } d_{\min}=1 \, , \\
1/\log d_{\min} &\text{ if } d_{\min}>1 \, ,
\end{cases}
\end{equation}
and define the integer function
\begin{equation} \label{def:FurthestGeneration}
\sD_n := \inf\left\{ m \in \N \colon  r^{\max}_\ell  \leq n^{-(2 +c_{\textup{low}}c_{o})}\log^{-3}n \text{ for all } \ell\geq m \right\} 
\end{equation} for all $n\in \N$, where we use the convention $\inf\{\emptyset\}=\infty$.  In words, $(\sD_n)_{n \in \N}$ denotes a sequence of generations along which all rates decay at least polynomially fast. The order of the underlying polynomial depends on the structure of the tree.  In particular, for exponentially fast decaying rates, $\sD_n$ will be of order $\log n$. We are now ready to quantify the generation where decoupling of the first $n$ particles is guaranteed.

\begin{theorem}[The disentanglement theorem]\label{thm:DisentanglementGWT}
Consider the TASEP on a Galton--Watson tree and assume that the transition rates satisfy assumptions \hyperref[def:UniformElliptic]{\normalfont(UE)} and \hyperref[def:ExponentialScaling]{\normalfont(ED)}. Recall $\e \in (0,1]$ from \hyperref[def:UniformElliptic]{\normalfont(UE)}. Let $\delta >0$ be arbitrary, but fixed, and define $\sM_n$ for all $n\in \N$ as follows.
\begin{enumerate}
\item\label{eq:LimSupFinite} When  $\displaystyle \limsup_{n \to \infty} \frac{\sD_n}{\log n} < \infty $ holds, set
\be \label{def:DecouplingPoint}
\sM_n := 
\begin{cases}\vspace{0.3cm}
(c_{o}+1)\sD_n +  c_o (2+\delta) \log_{1+\e} n,  & \textrm{if } d_{\min} =1 \, ,\\
\frac{d_{\min}}{d_{\min}-1}\sD_n + (2+\delta) \log_{1+\e} n, & \textrm{if } d_{\min} > 1 \, .
\end{cases}
\ee
\item\label{eq:LimSupFinite2} When  $\displaystyle \liminf_{n \to \infty} \frac{\sD_n}{\log n} = +\infty $ holds, set
\be \label{def:DecouplingPoint2}
\sM_n := \left(  c_{o}\mathds{1}_{\{ d_{\min} =1\}} + \frac{1}{d_{\min}-1}\mathds{1}_{\{ d_{\min} >1\}} +  (1+\delta)\right) \min\{\sD_n,n\}\, .
\ee
%\item When $\sD_n \geq n$, (including $\sD_n = \infty$) holds for all $n$ large enough, set
%\be %\label{eq:DecouplingPoint}
%\mathscr M_n := 
%\begin{cases}\vspace{0.3cm}
%(c_{o}+1)n +  c_o (3+\delta) \log_{1+\e} n , & \text{if } d_{\min} =1\\
%\frac{d_{\min}}{d_{\min}-1}n + (3+\delta) \log_{1+\e} n, & \text{if } d_{\min} > 1.
%\end{cases}
%\ee
\end{enumerate}
Then $\P$-almost surely, the trajectories of the first $n$ particles  decouple after generation $\sM_n$ for $n$ large enough, i.e.\ the first $n$ particles visit distinct sites at level $\sM_n$.
\end{theorem} 
\begin{remark} If in Theorem \ref{thm:DisentanglementGWT} neither \eqref{eq:LimSupFinite} or \eqref{eq:LimSupFinite2}  is satisfied, one could either pass to subsequences which satisfy \eqref{eq:LimSupFinite} or \eqref{eq:LimSupFinite2}, or instead apply Proposition \ref{lem:Early-sep} from Section \ref{sec:Disentangelment} which will give a coarse bound of order $n$ on the generation $\mathcal{M}_n$.
\end{remark}

\begin{example}[Disentanglement generations] \label{ex:intro} For the four examples on the $d$-regular tree, $\mu(d-1) = 1$, $d \ge 3$ we have $d_{\min} = d-1, c_o = 1/\log(d-1)$ and for $\delta>0$ arbitrarily small 
	\begin{enumerate}
		\item[(C)] ({\bf Constant rates})  $r_{x,y} = 1$.  Here $\sD_n=+\infty$ and so by \eqref{def:DecouplingPoint2}
			\[ \sM_n =\Big( \frac{1}{d-2} +  1+\delta\Big) n \,. \]
		\item[(E)] ({\bf Exponentially decaying, homogeneous, rates}) $\sD_n$ is of logarithmic order, so 
		\begin{equation}\label{eq:sMnExample}
			\sM_n = \left(\frac{3(d-1)}{(d-2)\log(d-1)} + \frac{2+\delta}{\log 2}\right)\log n \, .
		\end{equation}		
		\item[(S)] ({\bf Slow rates}) $r_{x,y} = (d-1)^{-|x|-1}g(|x|)$ where $g(s) \to 0 $ as $s \to \infty$ at most exponentially fast.  Because of $g$, we cannot solve for $\sD_n$ explicitly, but $\sD_n$ is still of logarithmic order as above.  If for example  $g(s) = s^{-p}$ for some $p>0$,  we may bound for $n$ large enough
		\[
		\sM_n \le \left(\frac{3(d-1)}{(d-2)(\log(d-1))} + \frac{2+\delta}{\log 2}\right)\log n \, . 
		\]
		\item[(P)] ({\bf Polynomially decaying rates of power $p$}) $r_{x,y} = (|x|+1)^{-p}$ where $p>0$.  $\sD_n$ is polynomial in order, so for $c_{\textup{low}}>0$ arbitrarily small
		\begin{equation*}
			\sM_n =  \frac{d-1+\delta}{d-2}\min( \left(n^{2+c_oc_{\textup{low}}}\log^{3} n \right)^{\frac{1}{p}} ,n)\, .
		\end{equation*} 
	\end{enumerate}
\end{example}

\begin{remark} Using the pigeonhole principle, we see that for $\GW$-almost every tree and any family of rates, the first generation of decoupling of $n$ particles will be at least of order $\log n$. When the rates decay exponentially fast, the disentanglement theorem ensures that order $\log n$ generations are sufficient to decouple $n$ particles. In particular, in this case the bounds in Theorem \ref{thm:DisentanglementGWT} are sharp up to constant factors.
\end{remark}
\begin{remark} A similar result on the disentanglement of the particles holds when we replace the reservoir by any dynamic that generates almost surely a number of particles which grows linearly in time. This may for example be a TASEP on a half-line attached to the root and started from a Bernoulli-$\rho$-product measure for some $\rho \in (0,1)$.
\end{remark}

%For the $d$-regular tree, Theorem \ref{thm:DisentanglementGWT} yields the following bounds on the point where all trajectories of particles have decoupled.
%\begin{example}\label{cor:RunningExample} On the $d$-regular tree for $d\geq 3$ with homogeneous rates from \eqref{eq:RunningExample}, for every $\delta>0$, there exists $\P$-almost surely an $N$ such that, for all $n \geq N$, the first $n$ particles are decoupled at 
%\begin{equation}\label{eq:sMnExample}
%\sM_n = \left(\frac{3(d-1)}{(d-2)\log(d-1)} + \frac{2+\delta}{\log 2}\right)\log n \, .
%\end{equation} If for some $p>0$, the rates satisfy 
%\begin{equation}\label{eq:PolynomialRatesExampleReg}
%r_x = |x|^{-p}
%\end{equation} for all $x\in V(\bT)$, then for every $\delta>0$, there exists $\P$-almost surely an $N$ such that, for all $n \geq N$, the first $n$ particles are decoupled at
%\begin{equation}
%\sM_n =  \frac{d-1+\delta}{d-2}\min( \left(n^{2}\log^{3} n \right)^{\frac{1}{p}} ,n)\, .
%\end{equation} 
%\end{example}
\subsubsection{Currents}  
Using Theorem \ref{thm:DisentanglementGWT}, we now study the current for the TASEP on Galton--Watson trees. For any pair of sites $x,y \in V$, we say that $y$ is below $x$ (and write $x \leq y$) if there exists a directed path in $\bT$ connecting $x$ to $y$. Let the starting configuration $\eta_0$ be either the empty configuration --- as we will mostly assume in the following --- or contain finitely many particles. Then we define the \textbf{current} $(\cJ_x(t))_{t \geq 0}$ across $ x \in V$ by
\begin{equation}\label{def:CurrentPointwise}
\cJ_x(t) := \sum_{y \colon x\leq y} \eta_t(y) - \sum_{y \colon x\leq y}  \eta_0(y) = \sum_{y \colon x\leq y}  (\eta_t(y)-\eta_0(y))
\end{equation} for all $t \geq 0$. Similarly, we define the \textbf{aggregated current} $(\cJ_\ell(t))_{t \geq 0}$ at generation $\ell$ by
\begin{equation}\label{def:CurrentAggregated}
\cJ_\ell(t) :=  \sum_{x \in \cZ_\ell} \cJ_x(t)
\end{equation} for all $\ell \in \N_0$ and $t \geq 0$.
The current (aggregated current) denotes the number of particles that have reached site $x$ (generation $\ell$) by time $t$. Our goal is to prove almost sure bounds for the aggregated current. This can be achieved in two different ways. On the one hand, we consider a given generation $\ell$ and study the time until $n$ particles have passed through $\ell$. On the other hand, we fix a time horizon $T$ and study how many particles have passed through a given generation until time $T$.  For $x \in V$, we denote by
\begin{equation}\label{r_xdef}
r_x:=\sum_{y \colon (x,y) \in \Er}r_{x,y}
\end{equation} the sum of outgoing rates at site $x$. For $m \geq \ell \geq 0$, we define
\begin{equation}\label{def:R_xdInverse}
R^{\min}_{\ell,m} := \sum_{i=\ell}^{m} \big(\min_{x \in \cZ_i}  r_{x} \big)^{-1} ,  \qquad  R^{\max}_{\ell,m} := \sum_{i=\ell}^{m} \big(\max_{x \in \cZ_i}  r_{x} \big)^{-1} 
\end{equation} and set $R^{\min}_{\ell}:=R^{\min}_{\ell,\ell}$ as well as $R^{\max}_{\ell}:=R^{\max}_{\ell,\ell}$. 
Intuitively, $R^{\min}_{\ell,m}$ and $R^{\max}_{\ell,m}$ are the expected waiting times to pass from generation $\ell$ to $m+1$ when choosing the slowest, respectively the fastest, rate in every generation. 
In the following, we state our results on the current in Theorems \ref{thm:LinearCurrentLSPECIAL} and \ref{thm:CurrentfixedTSPECIAL} only for exponentially decaying rates, i.e.\ if there exists some $c_{\textup{up}}>0$ such that 
\begin{equation}\label{def:ExpDecNew}
R^{\max}_{\ell} \geq \exp(c_{\textup{up}}\ell)
\end{equation}  holds for all $\ell \in \N$. 

We provide more general statements in Section \ref{sec:CurrentTheorems} and \ref{sec:RegularTreeTASEP} from which Theorems~\ref{thm:LinearCurrentLSPECIAL} and~\ref{thm:CurrentfixedTSPECIAL} as well as Examples~\ref{ex:Example3} and~\ref{pro:RunningExample} follow. Fix now some integer sequence $(\ell_n)_{n \in \N}$ with $\ell_n \geq \sM_n$ for all $n \in \N$, where $\sM_n$ is taken from Theorem \ref{thm:DisentanglementGWT}. For every $n\in \N$, we define a time window $[t_{\textup{low}},t_{\textup up}]$ in which we study the current through the $\ell_n^{\textup{th}}$ level of the tree, and where we see a number of particles proportional to $n$ passing~through~$\mathcal{Z}_{\ell_n}$.
\begin{theorem}[Time window for positive current under \eqref{def:ExpDecNew}]
\label{thm:LinearCurrentLSPECIAL} 
Suppose that  \hyperref[def:UniformElliptic]{\normalfont(UE)} and \hyperref[def:ExponentialScaling]{\normalfont(ED)} hold, and that the rates satisfy \eqref{def:ExpDecNew}.
 Then for any $\delta \in(0,1)$, there exist some $c>0$ such that for all choices of $t_{\textup{low}}=t_{\textup{low}}(n)$ and $t_{\textup{up}}=t_{\textup{up}}(n)$ with
\begin{equation}\label{eq:StatementLinearSPECIAL}
t_{\textup{up}} \geq c (n R^{\min}_{\sM_n}+ R^{\min}_{\ell_n})  \, ,  \qquad t_{\textup{low}} \leq \exp\Big(\frac{1}{2}c_{\textup{up}} \ell_n\Big) \,  ,
\end{equation} for $n \in \N$, we see that  $\P$-almost surely
\begin{equation}\label{eq:StatementCurrentLinearSPECIAL}
\lim_{n\to \infty} \cJ_{\ell_n}(t_{\textup{low}}) =0\, , \qquad  \liminf_{n\to \infty}\frac{1}{n}  \cJ_{\ell_n}(t_{\textup{up}}) \ge 1- \delta   \, .
\end{equation}
\end{theorem} 
\begin{remark} Note that under assumption \eqref{def:ExpDecNew}, we have the bound 
\begin{equation}
c (n R^{\min}_{\sM_n}+ R^{\min}_{\ell_n}) \leq  n^{\tilde{c}c_{\textup{low}}}+ n \exp(c_{\textup{low}}\ell_n) 
\end{equation} for some $\tilde{c}>0$. This upper bound can be used as a simple potential value for $t_{\textup{up}}$ in \eqref{eq:StatementLinearSPECIAL}.
\end{remark}

\begin{example}[Exponentially decaying, homogeneous rates (E)] \label{ex:Example3}On the $d$-regular tree for $d\geq 3$ with homogeneous rates from \eqref{eq:RunningExample}, let $\ell_n=\sM_n=c_d\log n$ with $\sM_n$ from \eqref{eq:sMnExample} and  $c_d>0$. Then there exists a constant  $c>0$ such that \eqref{eq:StatementCurrentLinearSPECIAL} holds when we set
\begin{equation}
t_{\textup{up}} = \frac{c}{d-1}n^{ \log(d-1)c_d+1} \, ,  \qquad t_{\textup{low}} = n^{ \frac{1}{2}\log(d-1)c_d} \, .
\end{equation} 
\end{example}
Note that the precision of the bounds on the current strongly depend on the transition rates and the structure of the tree. 
See Examples \ref{pro:PolyLinearCurrentFixedSPECIAL}-\ref{last} for polynomially decaying rates in Section \ref{sec:RegularTreeTASEP}. 
Theorem \ref{thm:LinearCurrentLSPECIAL} will be deduced from the more general Theorem~\ref{thm:LinearCurrentL}, while Example~\ref{ex:Example3} follows directly from Theorem~\ref{thm:LinearCurrentLSPECIAL} and Example \ref{ex:intro}.

%after which you can find example \ref{ex:TwindowConstant} for the time window on the $d$-regular tree with constant rates.  

Next, we let $t$ be a fixed time horizon and define an interval $[L_{\textup{low}}, L_{\textup{up}} ]$ of generations. Recall the generation $\sM_{n}$ from Theorem~\ref{thm:DisentanglementGWT} for the first $n$ particles and define 
\be \label{eq:goodC}
n_{t} := \sup\bigg\{ n \in \N_0: (n+\sM_n)(\min_{|x| \leq \sM_n}r_x)^{-1} \le  t\bigg\} \,  .
\ee
Note that for exponentially decaying rates, the quantity $n_t$ will be a polynomial in $t$. For large times $t$, the next theorem gives a window of generations where we expect to see the first $n_t$  particles which entered the tree.
\begin{theorem}[Generation window for positive current under \eqref{def:ExpDecNew}]
\label{thm:CurrentfixedTSPECIAL} Suppose that  \hyperref[def:UniformElliptic]{\normalfont(UE)} and \hyperref[def:ExponentialScaling]{\normalfont(ED)} holds, and the rates satisfy \eqref{def:ExpDecNew}.  Then there exists a constant $c>0$ such that for all  $L_{\textup{low}}=L_{\textup{low}}(t)$ and $L_{\textup{up}}=L_{\textup{up}}(t)$  with
\begin{equation}
L_{\textup{low}} \leq c \log t \, , \qquad L_{\textup{up}} \geq \frac{2}{c_{\textup{up}}} \log t \, ,
\end{equation} for $t\geq 0$, we see that  $\P$-almost surely
\begin{equation}
\lim_{t\to \infty} \cJ_{L_{\textup{up}}}(t) =0  \,  , \qquad  \liminf_{t\to \infty}\frac{1}{n_{t}} \cJ_{L_{\textup{low}}}(t)  > 0 \, .
\end{equation}
\end{theorem} 
We will obtain Theorem \ref{thm:CurrentfixedTSPECIAL} as a special case of Theorem \ref{thm:CurrentfixedT} in Section \ref{sec:CurrentTheorems}.
For the $d$-regular tree with homogeneous rates, we have the following current estimate in a sharp window of generations; see Section \ref{sec:RegularTASEPexpo} for the proof.
\begin{example}[Exponentially decaying, homogeneous rates (E)]\label{pro:RunningExample}  On the $d$-regular tree for $d\geq 3$ with  rates from \eqref{eq:RunningExample}, we have that $\P$-almost surely, for every $\alpha=\alpha(t)$ going to $0$ with $t \rightarrow\infty$
\begin{equation*}
\lim_{t \rightarrow \infty}  J_{L_{\textup{up}}}(t) =0 \,  , \qquad  \lim_{t \rightarrow \infty} t^{-\alpha} J_{L_{\textup{low}}}(t) = \infty \, ,
\end{equation*}  where we can choose $L_{\textup{low}}$ and $L_{\textup{up}}$ to be of the form
\begin{equation*}
L_{\textup{up}} = \frac{1}{\log(d-1)}\log t + o(\log t) \,  , \qquad   L_{\textup{low}} = \frac{1}{\log(d-1)}\log t - o(\log t)      \,  .
\end{equation*} 
\end{example}

\subsubsection{Large time behaviour}  We study the law of the TASEP on trees for large times.  
Again, we let the TASEP start from the all empty initial configuration according to 
$\nu_0$, where for $\rho \in [0,1]$, $\nu_{\rho}$ denotes the Bernoulli-$\rho$-product measure on $\{0, 1\}^V$. 
In contrast to the previous results, the geometry of the tree does not play an important role. However, we need assumptions on the transition rates. We assume that the rates are bounded uniformly from above. For $x \in V(T)$ recall $r_x$ from \eqref{r_xdef}. Let the \textbf{net flow} $q(x)$ through $x$ be 
\begin{equation}\label{def:LocalFlow}
q(x) := 
\begin{cases}  \vspace{0.3cm}
r_x - r_{\bar x, x} & x\neq o\, \, \\
r_o & x = o\, ,
\end{cases}
\end{equation}
where $\bar x$ is the unique parent of $x$. The rates satisfy a \textbf{superflow rule} if $q(x)\geq 0$ holds for all $x \in V(T) \setminus \{ o\}$. In particular, when $q(x)=0$ for all $x \neq o$ we say that the rates satisfy a \textbf{flow rule} with a flow of \textbf{strength} $q(o)$; see Figure \ref{fig:RegularTreeTASEP} for an example of flow with homogeneous rates on the binary tree.
% In words, at every site in $\bT$, the sum of the rates to jump away from $x$ is at least the amount of the rate of incoming jumps. 
%that a \textbf{global superflow rule} holds if
%\begin{equation}\label{def:GlobalSuperflow}
%\lim_{m \rightarrow \infty} r^{\min}_m \abs{\mathcal{Z}_m} = \infty
%\end{equation}
%and 
The rates satisfy a \textbf{subflow rule} if
\begin{equation}\label{def:GlobalSubflow}
\lim_{m \rightarrow \infty}\sum_{x \in \mathcal{Z}_m} r_{x} = 0\, .
\end{equation}% Intuitively, 
%a global superflow rule guarantees that the total amount of  particles movement increases along generations while
%a subflow rule ensures that the total particle movement across generations will converge to zero. 
%For $\rho \in (0,1)$, let $\nu_{\rho}$ be the Bernoulli-$\rho$-product measure on $\{0, 1\}^V$. Again, we let the TASEP start from the all empty initial distribution $\nu_0$.

\begin{example}[Flow conditions] For the four examples on the $d$-regular tree, $\mu(d-1) = 1$, $d \ge 3$ we have 	\begin{enumerate}
		\item[(C)] ({\bf Constant rates})  $r_{x,y} = 1$.  This is a superflow rule but not a flow rule.
		\item[(E)] ({\bf Exponentially decaying, homogeneous, rates})  $r_{x,y} = (d-1)^{-|x|-1}$. The rates here satisfy a flow rule.
		\item[(S)] ({\bf Slow rates}) $r_{x,y} = (d-1)^{-|x|-1}g(|x|)$ where $g(s) \to 0 $ as $s \to \infty$.  The condition on $g$ gives that the rates satisfy a subflow rule. 
		\item[(P)] ({\bf Polynomially decaying rates of power $p$}) $r_{x,y}  = (|x|+1)^{-p}$ where $p>0$. The rates satisfy a superflow rule if $\log 2 \le p^{-1}\log(d-1)$.
		\end{enumerate}
\end{example}

We endow the probability measures on $\{ 0,1 \}^{V}$ with the topology of weak convergence. 
%We use $\nu_1$ for the all-occupied measure on the tree, i.e. $\nu_1(\eta_x = 1) = 1$. 

\begin{theorem}[Fan and shock behaviour]
\label{thm:ConvergenceFlow} Let $(S_t)_{t \geq 0}$ be the semi-group of the TASEP $(\eta_t)_{t \geq 0}$ on a fixed tree $\bT \in \mathcal{T}$, where particles are generated at the root at rate $\lambda > 0$. We assume that  $\bT$ is infinite and without leaves. For all choices of $(r_{x,y})$, there exists a stationary measure $\pi_\lambda$ of $(\eta_t)_{t \geq 0}$ with
\begin{equation}\label{eq:ConvergenceFromAllEmptyTheorem}
\lim_{t \rightarrow \infty}\nu_{0}S_t = \pi_\lambda\, .
\end{equation}
Then 
under a superflow rule, $\pi_\lambda \neq \nu_1$ and the current $(\cJ_{o}(t))_{t \geq 0}$ through the root is $P_{\bT}$-almost surely linear in $t$,
\[
\varliminf_{t\to \infty}  \frac{\cJ_{o}(t)}{t} \ge \lambda \pi_{\lambda}(\eta(o) = 0).
\]
Moreover, if in addition $\lambda < r_o$ as well as
\begin{equation}\label{eq:SuperflowGrowthConditionSingle}
\lim_{n \rightarrow \infty} \abs{\mathcal{Z}_n} \min_{x \in \mathcal{Z}_n} r_{x,y} = \infty \, ,
\end{equation} the system exhibits a fan behaviour, i.e.\ we have 
\begin{equation}\label{eq:VanishingDensitySuperflowTheorem}
 \lim_{n \rightarrow \infty}\frac{1}{\abs{\mathcal{Z}_n}} \sum_{x \in \mathcal{Z}_n} \pi_{\lambda}(\eta(x)=1)   = 0 \, .
\end{equation}
Under a subflow rule, the system exhibits a shock behaviour, i.e.\ $\pi_{\lambda}=\nu_1$ and 
 \begin{equation}
 \lim_{t \rightarrow \infty}\frac{\cJ_{o}(t)}{t} = 0\, \qquad P_{\bT}- \text{almost surely}.
 \end{equation} 
\end{theorem}

We direct the reader to Section \ref{sec:LargeTimes} for the proof of Theorem \ref{thm:ConvergenceFlow}, a more detailed discussion and further results regarding the TASEP on trees in equilibrium.

\subsection{One-dimensional TASEP and parallels with TASEP on trees}
\label{ssec:parallel}

%The totally asymmetric simple exclusion process has been extensively studied in the past decades from a variety of different perspectives. see  
%One of the simplest and most studied case is TASEP on the infinite integer line. Particles attempt to move independently one unit to the right at rate 1, but the jump is suppressed if the target location is occupied. When the jump rates are all equal we say that the TASEP is homogeneous.

A great strength of various particle systems are their explicit hydrodynamic limits, as macroscopic and microscopic behaviour are connected; see for example \cite{ferrari2018tasep} for a beautiful survey on TASEP, and references therein. Hydrodynamic limits for the homogeneous TASEP, in the sense of a rigorous connection to the Burgers equation, were originally established by Rost in the rarefaction fan case \cite{Ros-81}. This result was extended in various ways in \cite{Sep-97-aap-1, S:CouplingInterfaces, Sep-98-mprf-1, Sep-99-aop}.
The particle density was shown to satisfy a scalar conservation law with an explicit flux function that turns out to be the convex dual of the limiting level curve of the last passage limiting shape. The density is the almost sure derivative of the aggregated current process, which, when appropriately scaled, satisfies a Hamilton-Jacobi-Bellman equation.   

A key endeavour is to understand the equilibrium measures. For the homogeneous one-dimensional TASEP, the extremal invariant measures are Bernoulli-product measures and Dirac measures on blocking states; see   \cite{BLM:CharacterizationStationary,J:Extremal,L:Coupling}. 
In Lemma \ref{lem:InvarianceFlow} we verify for the TASEP on trees  that a flow rule implies the existence of  non-trivial invariant product measures.

When the jump rates are deterministic, but not equal, we have a spatially inhomogeneous TASEP. In this case, even in dimension one, less is known and usually the results have some conditions on the rates. For example, consider the one-dimensional particle system where we alter the rate of the Poisson process of the origin only; any particle at $0$ jumps to site 1 at rate $r <1$, while all other jump rates remain $1$. 
The question is whether this local microscopic change affects the macroscopic current for all values of $r <1$. This is known as the `slow bond problem', introduced in \cite{Ja-92, Ja-94} on a finite segment with open boundaries. On $\Z$, progress was made in \cite{Sep-01-jsp} where a coupling with the corner growth model in LPP showed that the current is affected for $r$ less than $\sim 0.425$, and a hydrodynamic limit  for the particle density was proven. A positive answer to the question appeared in  \cite{BSV:SlowBond}.  

When the inhomogeneities are not local but macroscopic, several articles show hydrodynamic limits for TASEP  (or study the equivalent inhomogeneous LPP model) with increasing degree of complexity for admissible deterministic rates coming from a macroscopic speed function \cite{Cal-15, CG:LPP2018, GKS:TASEPdisc,RT:LPPInhom}. The commonality between them is that the rates need to behave in a nice way so that the current of TASEP at position $\fl{nx}$ at time $\fl{nt}$ remains linear in $n$. In this  article we use a coupling with the corner growth model to bound the current; see Section \ref{sec:LPP}. We only make the Assumptions \hyperref[def:UniformElliptic]{\normalfont(UE)} and \hyperref[def:ExponentialScaling]{\normalfont(ED)} for  admissible rates. As such, evidenced by Theorems \ref{thm:LinearCurrentL} and \ref{thm:CurrentfixedT}, we have different regimes for the order of current up to a given time, where the current is not necessarily linear in time. Our results such as the sharp order of magnitude for the time window assume more on the decay of the jump rates across the tree; see Section \ref{sec:RegularTreeTASEP} for explicit calculations on the $d$-regular tree. 
%Under additional assumptions on the decay of the jump rates, the results can be sharpened for the $d$-regular trees to obtain the order of magnitude for the time window, see Section \ref{sec:RegularTree}. \\

Depending on initial particle configurations, the macroscopic evolution of the particle density in one-dimensional TASEP may exhibit a shock or a rarefaction fan, as one can see from the limiting partial differential equation. In a simple two-phase example, even starting from macroscopically constant initial conditions, one can see the simultaneous development of shocks and fans, depending on the common value of the density \cite{GKS:TASEPdisc}.  In this article, we can still describe shock or fan behaviour of the limiting particle distribution; see Theorem~\ref{thm:ConvergenceFlow}, even without a hydrodynamic limit. In particular, we show that this behaviour in fact occurs in the limit, starting from the all empty initial condition. 
A tool we are using is to approximate the TASEP by a finite system with open boundaries; see Section  \ref{sec:LargeTimes}. Stationary measures for the one-dimensional TASEP with reservoirs and deaths of particles were studied using elaborated tools like the Matrix product ansatz; see \cite{BE:Nonequilibrium}, or combinatorial representations, like staircase tableaux and Catalan paths  \cite{CW:TableauxCombinatorics,M:TASEPCombinatorics,SW:CombinatoricSemiopen}.

\subsection{Outline of the paper}

In the remainder of the paper, we give proofs for the results presented in Section \ref{sec:ModelResults}. We start in Section \ref{sec:Disentangelment} with the proof of the disentanglement theorem. The proof combines combinatorial arguments, geometric properties of Galton--Watson trees and large deviation estimates on the particle movements. 
In Section \ref{sec:Couplings}, we introduce three couplings with respect to the TASEP on trees which will be helpful in the proofs of the remaining theorems. This includes the canonical coupling for different initial configurations, a coupling to independent random walks and a comparison to a slowed down TASEP on the tree which can be studied using inhomogeneous last passage percolation. 
These tools are then applied in Section \ref{sec:CurrentTheorems} to prove Theorems \ref{thm:LinearCurrentL} and  \ref{thm:CurrentfixedT} on the current, which in return give Theorems \ref{thm:LinearCurrentLSPECIAL} and \ref{thm:CurrentfixedTSPECIAL}. 
We show in Section \ref{sec:RegularTreeTASEP} that the current bounds can be sharpened for specific rates on the $d$-regular tree. 
In Section \ref{sec:LargeTimes}, we turn our focus to the large time behaviour of the TASEP and prove Theorem \ref{thm:ConvergenceFlow}. This uses ideas from \cite{L:ErgodicI} as well as the canonical coupling. 
We conclude with an outlook on open problems. 
%\newpage

%%%%%%%%%%%%%%%%%%%%%%%%%%%%%%%%%%%%%%%%%%%%%%%%%%%%%%

\section{The disentanglement theorem}\label{sec:Disentangelment}

The proof of Theorem \ref{thm:DisentanglementGWT} will be divided into four parts. First, we give an a priori  argument on the level where the particles disentangle, requiring assumption \hyperref[def:UniformElliptic]{\normalfont(UE)}. We then study geometric properties of Galton--Watson trees. Afterwards, we estimate the time of $n$ particles to enter the tree. This will require only assumption  \hyperref[def:ExponentialScaling]{\normalfont(ED)}. In a last step, the ideas are combined in order to prove Theorem \ref{thm:DisentanglementGWT}.

\subsection{An a priori bound on the disentanglement}
In this section, we give an a priori bound on the disentanglement of the trajectories within the exclusion process.  This bound relies on a purely combinatorial argument, where we count the number of times a particle performing TASEP has a chance to disentangle from a particle ahead. Recall that we start from the configuration where all sites are empty.
 For a given infinite, locally finite rooted tree $\bT$ and $x,y \in V(\bT)$, recall that we denote by  $[x,y]$ the set of vertices in the shortest path in $\bT$ connecting $x$ and $y$. We set 
\begin{equation}\label{def:NumberOfTwoOffsprings}
 F(o,x) := \left|\left\{ z \in [o,x] \setminus \{ x \} \colon \deg(z) \geq 3  \right\} \right|
\end{equation} 
to be the number of vertices in $[o,x] \setminus \{ x \}$ with degree at least $3$. 
For any fixed  tree $(\bT,o) \in \mathcal T$, let $d_{\bT}$ be the smallest possible number of offspring a site can have. Note that when $\bT$ is a Galton--Watson tree,  $d_\bT = d_{\min}$ holds $\GW$-almost surely for $d_{\min}$ from \eqref{def:MinimalDegreeConditionedOffspring}. For all $i, m \in \N$, let $z_i^m \in \cZ_m$ denote the unique site at generation $m$ which is visited by the $i^{\text{th}}$ particle which enters the tree.
\begin{proposition}
\label{lem:Early-sep}
For $(\bT,o)\in \mathcal{T}$, consider the TASEP on $\bT$ where $n$ particles are generated at the root according to an arbitrary rule. Assume that  \hyperref[def:UniformElliptic]{\normalfont(UE)} holds for some $\varepsilon>0$. Then 
\begin{equation}\label{eq:DisentanglementStatementLemma}
P_{\bT}\left( z_i^m \neq z_j^m  \text{ for all } i,j \in \{1,\dots,n\} , \ i\neq j\right) \geq 1- n^2 \left(\frac{1}{\varepsilon+1}\right)^{F_n(m)},
\end{equation}
where for all $m,n \in \N$, we set
\begin{equation}\label{def:DisentangleGeneration}
F_n(m) := 
\begin{cases}  \vspace{0.3cm}
\min\left\{ F(o,x)  \colon x \in \cZ_m\right\} - n & \text{ if \,\,} d_{\bT} =1\, , \\
\displaystyle m - \left\lceil  n (d_{\bT}-1)^{-1}\right\rceil & \text{ if \,\,} d_{\bT} \geq 2 \,   .
\end{cases}
\end{equation} 
\end{proposition}

We will use Proposition \ref{lem:Early-sep} to control the probability that two particles 
have the same exit point at $\cZ_m$, in a summable way, provided that $F_n(m) \geq c\log(n)$ for some $c=c(\varepsilon)>0$. Note that this bound can in general be quite rough as for example on the regular tree with rates as in \eqref{eq:RunningExample}, we expect $n$ particles to disentangle already after order $\log n$ generations.

\begin{proof}[Proof of Proposition \ref{lem:Early-sep}] Consider the $j^{\text{th}}$ particle for some $j  \in [n]:=\{1,\dots,n\}$ which enters the tree. We show that the probability of particle $j$ to exit from $x \in \cZ_{m}$ satisfies
\be \label{eq:HittingTrajectory}
P_{\bT}\left( z_j^m=x \right) \leq \Big(\frac{1}{1+\e}\Big)^{F_n(m)}
\ee
for all $j \in [n]$. Note that if  particle $j$ exits through $x$, it must follow the unique path $[o,x]$; see also Figure \ref{fig:AprioriDisentanglement}. Our goal is to find a generation $m$ large enough that guarantees that on any ray the particle will have enough opportunities to escape this ray. \\

For $d_{\bT} \geq 3$, we argue that any particle will encounter at least $F_n(m)$ many locations on $[o,x]$ which have at least $2$ holes in front when the particle arrives. To see this, suppose that particle $j$ encounters at least $n (d_{\bT}-1)^{-1}$ generations among the first $n$ generations with no two empty sites in front of it when arriving at that generation. In other words, particle $j$ sees at least $d_{\bT} - 1$ particles directly in front of its current position when reaching such a generation. Since particle $j$ may follow the trajectory of at most one of these particles, this implies that particle $j$ encounters at least $(d_{\bT}-1)\cdot \frac{n}{d_{\bT}-1} = n$ different particles in total until reaching level $n$. This is a contradiction as $j \leq  n$ and the tree was originally empty.  \\

\begin{figure}[t]
\centering
\begin{tikzpicture}[scale=1]

   \node[shape=circle,scale=1.5,draw] (A) at (0,0){} ;
 	\node[shape=circle,scale=1.5,draw] (B1) at (2,1){} ;
 	\node[shape=circle,scale=1.5,draw] (B2) at (2,-1){} ;
	\node[shape=circle,scale=1.5,draw] (C1) at (4,2) {};
	\node[shape=circle,scale=1.5,draw] (C2) at (4,1) {};
	\node[shape=circle,scale=1.5,draw] (C3) at (4,-1) {};
	
	\node[shape=circle,scale=1.5,draw, thick] (D1) at (6,2) {};
	\node[shape=circle,scale=1.5,draw, thick] (D2) at (6,1) {};
	\node[shape=circle,scale=1.5,draw, thick] (D3) at (6,0) {};	
	\node[shape=circle,scale=1.5,draw, thick] (D4) at (6,-1) {};

   \node[shape=circle,scale=1.2,fill,red] (AX1) at (A){} ;
         \node[shape=circle,scale=1.2,fill,red] (AX3) at (C2){} ; 
            \node[shape=circle,scale=1.2,fill,red] (AX4) at (D4){} ;  

			\draw[thick] (A) to (B1);	
			\draw[thick, RoyalBlue] (A) to (B2);	
			\draw[thick] (B1) to (C1);		
			\draw[thick] (B1) to (C2);			
			\draw[thick, RoyalBlue] (B2) to (C3);			
			\draw[thick] (C1) to (D1);	
			\draw[thick] (C2) to (D2);			
			\draw[thick] (C3) to (D3);				
			\draw[thick, RoyalBlue] (C3) to (D4);	
						
			\draw[thick,dashed] (D1) to (7,2.3);		
			\draw[thick,dashed] (D1) to (7,1.7);	

			\draw[thick,dashed] (D2) to (7,1);	
				
			\draw[thick,dashed] (D3) to (7,0.3);	
			\draw[thick,dashed] (D3) to (7,-0.3);	
			
			\draw[thick,dashed] (D4) to (7,-0.7);		
			\draw[thick,dashed] (D4) to (7,-1.3);	
				
		   \node (AX) at (0,-0.55){$i$};
			\node (AY) at (6,-1.6){$j$};		

    		\draw[->,thick] (-0.8,0) to (A);	

			\draw[->,thick] (A) to (B1);	
			\draw[->,line width=2pt, RoyalBlue] (A) to (B2);	
			\draw[->,thick] (B1) to (C1);		
			\draw[->,thick] (B1) to (C2);			
			\draw[->,line width=2pt, RoyalBlue] (B2) to (C3);			
			\draw[->,thick] (C1) to (D1);	
			\draw[->,thick] (C2) to (D2);			
			\draw[->,thick] (C3) to (D3);				
			\draw[->,line width=2pt, RoyalBlue] (C3) to (D4);

\end{tikzpicture}
\caption[Key idea for the proof of the a priori bound]{\label{fig:AprioriDisentanglement} Visualization of the key idea for the proof of the a priori bound on the disentanglement. When \hyperref[def:UniformElliptic]{\normalfont(UE)} holds, the probability that particle $i$ follows the blue trajectory of particle $j$ is at most $\big( \frac{1}{1+\varepsilon} \big)^2$.}
\end{figure}
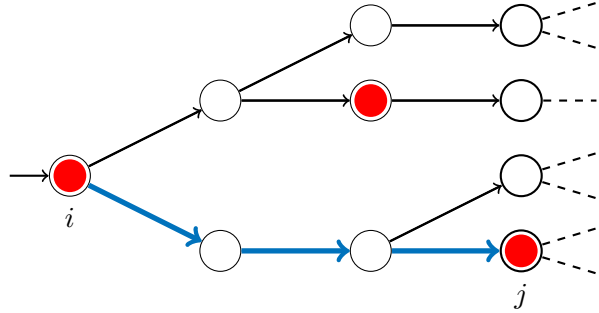

For $ d_{\bT} \in \{1,2\}$ we apply a similar argument. We need to find $m$ large enough so that every possible trajectory has $\min_{x \in \cZ_m}F(o, x) \ge n$ locations where, when a particle arrives there are at least two children, and there is no particle ahead.   
By definition, every possible trajectory has at least $F(o, x) \ge F_n(m)+n$ sites with at least two children. 
Observe that in order to follow the trajectory $[o,x]$ for some $x \in \cZ_m$, the first accepted transition at every stage must be along $[o,x]$. 
But there can be at most $n$ sites at which the first attempt was not to follow $[o,x]$ and this attempt was suppressed. 
This is because in order to block an attempt of leaving $[o,x]$, the blocking particle cannot be on $[o,x]$ and thus block only a single attempt of particle $j$ to jump. 
Hence, there must be at least $F_n(m)$ sites of degree at least $3$ accepting the first attempted transition. \\

Now we prove \eqref{eq:HittingTrajectory}. Suppose that particle $j$ is at one of the $F_n(m)$ many locations, say $y \in \cZ_\ell$, on $[o,x]$ where two different children  $z_1, z_2$ of $y$ are vacant. At most one of them belongs to $[o,x]$, say $z_1$. Using \hyperref[def:UniformElliptic]{\normalfont(UE)}, the probability of selecting $z_1$ is bounded from above by $(1+\e)^{-1}$. To stay on $[o,x]$, we must pick the unique site in $[o,x]$ at least $F_n(m)$ many times, independently of the past trajectory. This shows \eqref{eq:HittingTrajectory}.
Since particle $i$ is not influenced by the motion of particle $j$ for all $j> i$, we conclude
\begin{equation*}
P_{\bT}\left( \exists i,j \in [n], i \neq j, \colon z_i^m = z_j^m  \right)  \leq \sum_{1\leq i<j\leq n}P_{\bT}\left( z_i^m=z_j^m \right) \leq n^2 \Big(\frac{1}{1+\e}\Big)^{F_n(m)} , 
\end{equation*} applying 
\eqref{eq:HittingTrajectory} for the last inequality.
\end{proof}

\subsection{Geometric properties of the Galton--Watson tree}

Next, we give an estimate on the number $F(o,x)$, defined in \eqref{def:NumberOfTwoOffsprings}, which will be essential in the proof of Theorem \ref{thm:DisentanglementGWT} when there is a positive probability to have exactly one offspring.  \\

We define the \textbf{core} of a Galton--Watson tree  to be the Galton--Watson tree, which we obtain by conditioning in the offspring distribution with respect to $(p_k)_{k \in \N}$ on producing at least $2$ sites. Intuitively, we obtain the core from a given tree by collapsing all linear segments to single edges.  On the other hand, given a core $\tilde{\bT}$ according to the conditioned offspring distribution, we can reobtain a Galton--Watson tree with the original offspring distribution according to $(p_k)_{k \in \N}$, by extending every edge $\tilde e$ to a line segment of size $G_{\tilde e}$ where $(G_{\tilde e})_{\tilde e \in E(\tilde \bT)}$ are i.i.d.\ Geometric-$(1-p_1)$-distributed random variables supported on 
$\bN_0$. 
Moreover, we have to attach a line segment $[o,\tilde{o}]$ of Geometric-$(1-p_1)$-size to the root $\tilde{o}$ of $\tilde{\bT}$ and declare $o$ to be the new root of the tree. \\

 An illustration of this procedure is given in Figure \ref{fig:Core}.  We now give an estimate on how much the tree is stretched when extending the core with the conditioned  offspring distribution to a Galton--Watson tree with an offspring distribution with respect to $(p_k)_{k \in \N}$.

\begin{figure}
\centering
\begin{tikzpicture}[scale=0.8]

	\node[shape=circle,scale=1,draw,fill=red] (A2) at (0,0){} ;
 	\node[shape=circle,scale=1,draw,fill=red] (B2) at (2,1.5){} ;
	\node[shape=circle,scale=1,draw,fill=red] (C2) at (2,-1.5) {};
	\node[shape=circle,scale=1,draw,fill=red] (D2) at (4,0.5){} ;
	\node[shape=circle,scale=1,draw,fill=red] (E2) at (4,2.5){} ;
	\node[shape=circle,scale=1,draw,fill=red] (F2) at (4,-1){} ;
	\node[shape=circle,scale=1,draw,fill=red] (G2) at (4,-2){} ;
	\node[shape=circle,scale=1,draw,fill=red] (H2) at (4,1.5){} ;

%	fill=darkblue!70
	
	\draw[thick] (A2) to (B2);		
	\draw[thick] (A2) to (C2);		
	\draw[thick] (B2) to (D2);		
	\draw[thick] (B2) to (E2);	
	\draw[thick] (C2) to (F2);
	\draw[thick] (C2) to (G2);		
	\draw[thick] (B2) to (H2);

	\draw[thick, dashed] (D2) to (5,0.8);			
	\draw[thick, dashed] (D2) to (5,0.2);	
				
	\draw[thick, dashed] (E2) to (5,2.8);		
	\draw[thick, dashed] (E2) to (5,2.2);	
			
	\draw[thick, dashed] (F2) to (5,-1.3);		
	\draw[thick, dashed] (F2) to (5,-1);		
	\draw[thick, dashed] (F2) to (5,-0.7);	

	\draw[thick, dashed] (G2) to (5,-2);		
	\draw[thick, dashed] (G2) to (5,-2.3);		
	\draw[thick, dashed] (G2) to (5,-1.7);	

	\draw[thick, dashed] (H2) to (5,1.8);		
	\draw[thick, dashed] (H2) to (5,1.2);	
	\draw[thick, dashed] (H2) to (5,1.5);	
	
		\node[scale=0.9]  at (0.2,-0.5){$\tilde{o}$};

% fill=sussexg!80

\def\x{9}		
		
		\node[shape=circle,scale=0.8,draw,fill=gray!20] (Z2) at (0+\x,0){} ;
			\node[shape=circle,scale=1,draw,fill=red] (A2) at (1+\x,0){} ;
 	\node[shape=circle,scale=1,draw,fill=red] (B2) at (4+\x,1.5){} ;
	\node[shape=circle,scale=1,draw,fill=red] (C2) at (4+\x,-1.5) {};
	\node[shape=circle,scale=1,draw,fill=red] (D2) at (8+\x,0.5){} ;
	\node[shape=circle,scale=1,draw,fill=red] (E2) at (8+\x,2.5){} ;
	\node[shape=circle,scale=1,draw,fill=red] (F2) at (8+\x,-1){} ;
	\node[shape=circle,scale=1,draw,fill=red] (G2) at (8+\x,-2){} ;
	\node[shape=circle,scale=1,draw,fill=red] (H2) at (8+\x,1.5){} ;

	\draw[thick] (A2) to (Z2);		
	\draw[thick] (A2) to (B2);		
	\draw[thick] (A2) to (C2);		
	\draw[thick] (B2) to (D2);		
	\draw[thick] (B2) to (E2);	
	\draw[thick] (C2) to (F2);
	\draw[thick] (C2) to (G2);		
	\draw[thick] (B2) to (H2);

	\node[shape=circle,scale=0.8,draw,fill=gray!20] (E21) at (6+\x,2){} ;
	
%	\node[shape=circle,scale=0.8,draw,fill=white] (H21) at (5+\x,1.5){} ;
%	\node[shape=circle,scale=0.8,draw,fill=white] (H22) at (6+\x,1.5){} ;
%	\node[shape=circle,scale=0.8,draw,fill=white] (H23) at (7+\x,1.5){} ;		
	
	\node[shape=circle,scale=0.8,draw,fill=gray!20] (G21) at (6+\x,-1.75){} ;
	\node[shape=circle,scale=0.8,draw,fill=gray!20] (F21) at (6.66666+\x,-1.16666){} ;		
	\node[shape=circle,scale=0.8,draw,fill=gray!20] (F22) at (5.33333+\x,-1.33333){} ;	
	
	\node[shape=circle,scale=0.8,draw,fill=gray!20] (D21) at (6.66666+\x,0.833333){} ;		
	\node[shape=circle,scale=0.8,draw,fill=gray!20] (D22) at (5.33333+\x,1.166666){} ;		
	
	\node[shape=circle,scale=0.8,draw,fill=gray!20] (B21) at (2+\x,0.5){} ;		
	\node[shape=circle,scale=0.8,draw,fill=gray!20] (B22) at (3+\x,1){} ;

	\draw[thick, dashed] (D2) to (9+\x,0.8);			
	\draw[thick, dashed] (D2) to (9+\x,0.2);	
				
	\draw[thick, dashed] (E2) to (9+\x,2.8);		
	\draw[thick, dashed] (E2) to (9+\x,2.2);	
			
	\draw[thick, dashed] (F2) to (9+\x,-1.3);		
	\draw[thick, dashed] (F2) to (9+\x,-1);		
	\draw[thick, dashed] (F2) to (9+\x,-0.7);	

	\draw[thick, dashed] (G2) to (9+\x,-2);		
	\draw[thick, dashed] (G2) to (9+\x,-2.3);		
	\draw[thick, dashed] (G2) to (9+\x,-1.7);	

	\draw[thick, dashed] (H2) to (9+\x,1.8);		
	\draw[thick, dashed] (H2) to (9+\x,1.2);	
	\draw[thick, dashed] (H2) to (9+\x,1.5);	
	
		\node[scale=0.9]  at (0.2+\x,-0.5){$o$};
		
		\node[scale=0.9]  at (4.2+\x,-3){$\bT_o$};
		\node[scale=0.9]  at (3.2,-3){$\tilde \bT_{\tilde o}$};

\end{tikzpicture}
\caption[A Galton--Watson tree and its core]{\label{fig:Core}  A core at the left-hand side and one of its corresponding Galton--Watson trees on the right-hand side. We obtain the Galton--Watson tree from the core (the core from the Galton--Watson tree) by adding  (removing) the smaller vertices depicted in gray.}
\end{figure}
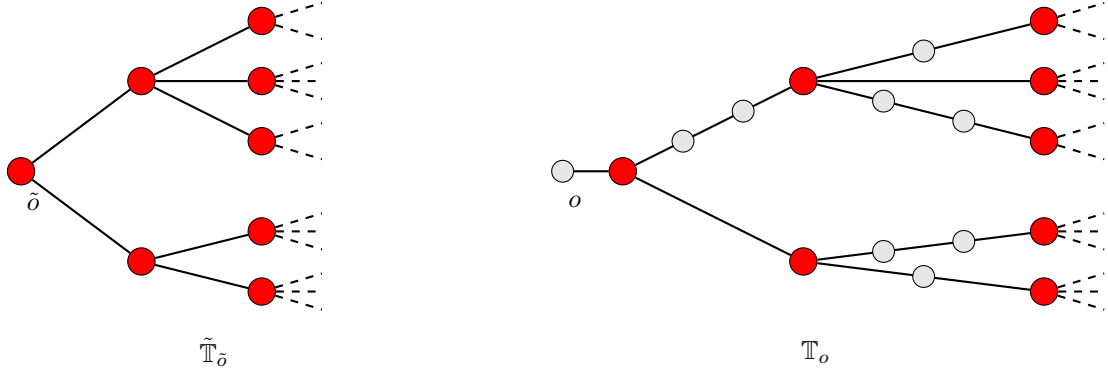

\begin{lemma} \label{lem:coreLD}  Let $(H_n)_{n \in \N}$ be an increasing sequence that goes to infinity 
%at least at order $\log n$ 
and assume that $p_{1} \in (0,1)$. Recall $\tilde{ \mathfrak m}$ from \eqref{def:MinimalDegreeConditionedOffspring}. Set $M_n := \ce{ \alpha H_n}$ for all $n \in \N$, where
\begin{equation}
\alpha :=  \frac{5+ \log_2 \tilde{ \mathfrak m}}{\log_2(1+p_1) -\log_2({2p_1})}\, .
\end{equation}
Then 
\be\label{eq:de-coreL2}
\GW\Big( \inf_{x \in  \cZ_{M_n}} \sum_{v\in [o, x)} \mathds{1}\{ \deg(v) \geq 3 \} \ge \ce{ H_n}  \Big) \ge 1 - 2^{-2H_n+1} \, .
\ee 
\end{lemma}

\begin{proof} Note that all sites in the core $\tilde \bT$ other than the root have at least degree $3$. Hence, it suffices to bound the probability that all sites at generation $H_n$ of $\tilde \bT$ are mapped to a generation  less or equal than $M_n$ in the corresponding Galton--Watson tree. Using Markov's inequality, we see that
\begin{align}\label{eq:vertexMarkovbound}
\GW(  |x \in V(\tilde{\bT}) \colon \abs{x}=H_n|  \ge  \tilde{\mathfrak m}^{H_n}2^{2H_n}  )
\le 2^{-2H_n}\, .
\end{align}
Note that each site $x$ at level $H_n$ in $\tilde{\bT}$ is mapped to a generation given as the sum of $H_n$-many independent Geometric-$(1-p_1)$-distributed random variables $(G_i)_{i \in [H_n]}$. Using Chebyshev's inequality, we see that
\begin{align} \label{eq:geoLDPbound}
 P\Big( \sum_{i=1}^{H_n} G_i \ge M_n\Big) \leq \eup^{-tM_n}\left(\frac{1-p_1}{1-p_1\eup^t} \right)^{H_n} =  \left(\frac{1+p_1}{2p_1}\right)^{-M_n}2^{H_n}
\end{align} when we set $t=\log(\frac{1+p_1}{2p_1})$. Fix some site $\tilde{x} \in  \cZ_{M_n}$. Now condition on the number of sites at level $H_n$ in $\tilde{\bT}$ and apply \eqref{eq:vertexMarkovbound} together with a union bound to see that
\begin{align*}
\GW\Big(  \exists x \in  \cZ_{M_n} &\colon \sum_{v\in [o, x)} \mathds{1}\{ \deg(v) \geq 3 \} \le \ce{ H_n}  \Big) \\
&\leq   \tilde{\mathfrak m}^{H_n}2^{2H_n}  \GW\Big(  \sum_{v\in [o, \tilde{x})} \mathds{1}\{ \deg(v) \geq 3 \} \le \ce{ H_n}   \Big) + 2^{-2H_n} \\
&\leq \tilde{\mathfrak m}^{H_n}2^{2H_n}  P\Big( \sum_{i=1}^{H_n} G_i \ge M_n\Big) + 2^{-2H_n} \leq 2^{-2H_n+1}
\end{align*}
using \eqref{eq:geoLDPbound} and the definition of $M_n$ for the last two steps.
\end{proof}

\subsection{Entering times of the particles in the tree}

We now define an inverse for the current. For any $n \in \N$, $m \in \N_0$, we set 
\be\label{def:TauUp}
\tau_{m}^{n} := \inf \{ t \geq 0 : \cJ_{m}(t) \ge n \}\,  .
\ee
In words, $\tau_{m}^{n}$ gives the time that the aggregated current across generation $m$ becomes $n$, or equivalently, precisely $n$ particles reached $\cZ_m$. 
Hence, the following two events are equal:
\begin{equation*}
\{ \tau_{m}^{n} \le t \} = \{ \cJ_{m}(t) \ge n \} \, .
\end{equation*}

The main goal of this section is to give a bound on the first time $\tau_{0}^{n}$ at which $n$ particles have entered the tree. Note that this random time $\tau_{0}^{n}$ depends on the underlying tree as well as on the evolution of the exclusion process. 

\begin{proposition}\label{pro:EnteringParticles} Fix a number of particles $n$. Consider a supercritical Galton--Watson tree without extinction and assume that \hyperref[def:ExponentialScaling]{\normalfont(ED)} holds for some constant $c_{\textup{low}}$. Recall $c_{o}$ from \eqref{def:cmu}. There exists a constant $c>0$ such that 
\begin{equation}
\GW\left(  P_{\bT}\big (\tau_{0}^{n} <  c n^{c_{\textup{low}}c_{o}+1}\log n \big) \geq 1- \frac{2}{n^2} \right) \geq 1- \frac{2}{n^2}
\end{equation} 
for all $n$ sufficiently large.
\end{proposition}

In order to show Proposition \ref{pro:EnteringParticles}, we require a bit of setup.  Let $\cZ_m^{(x)}$  be the $m^{\text{th}}$ generation of the subtree $\bT_x$ rooted at $x$. For a tree $(\bT,o) \in \mathcal{T}$ and a site $x$, we say that the exclusion process on $\bT$ has \textbf{depth of traffic} $D_x(t) \in \N_0$ with
\be
D_x(t) = \inf\{ m \ge 0: \eta_t(z) = 0, \textrm{ for some } z \in \cZ_m^{(x)}\}\, ,
\ee
at site $x$ at time $t$. In words, $D_x(t)$ is the distance to the first generation ahead of $x$ which contains an empty site. Note that for any fixed $x$, the process $(D_x(t))_{t \geq 0}$ is a non-negative integer process. It takes the value $0$ when $\eta_t(x) = 0$ and it is positive when $\eta_t(x) = 1$. Note that $(D_x(t))_{t \geq 0}$ can go down only in steps of one, unless at $0$ where it jumps to some positive integer. The following lemma gives a bound on the depth of traffic at the root in Galton--Watson trees. 

\begin{lemma}\label{lem:depthboundRoot} Let $H_n=\log_2 n$ and recall $M_n$ and $\tilde{\mathfrak m}$ from Lemma \ref{lem:coreLD}. Then 
\begin{equation}
\GW\Big( P_{\bT}\big( D_o(t) \leq M_n+1 \text{ for all } t \leq \tau_{0}^{n} \big) = 1 \Big) \ge 1 - \frac{2}{n^{2}}\, .
\end{equation}
%holds where $o$ denotes the root of the underlying Galton--Watson tree. 
\end{lemma}
In words, this means that with probability at least $1-2n^{-2}$, the depth at the root is smaller than $M_n$ whenever no more than $n$ particles have entered the tree.
\begin{proof}[Proof of Lemma \ref{lem:depthboundRoot}] Observe that the root can only have depth $\ell$ when all vertices until level $\ell$ are occupied and that there are at most $n$ particles until time $\tau^{n}_{0}$. Note that Lemma~\ref{lem:coreLD} guarantees, with our choice of $H_n$, that with probability at least $1-2n^{-2}$, the tree up to generation $M_n$ contains more than $n$ sites. Hence, there is at least one empty site until generation $M_n$ by the definition of $\tau^{n}_{0}$. 
\end{proof}

Next, we give a bound on the renewal times of the process $(D_o(t))_{t \geq 0}$. For $t \geq 0$ and $x\in V$,  we define the first \textbf{availability time} $\psi_x(t)$ after time $t$ to be
\[
\psi_x(t)= \inf\{ s > t: D_x(s)=0 \} -t \geq 0 \, .
\]
This is the time it takes until $x$ is empty, observing the process from time $t$ onward. 

\begin{lemma} \label{lem:depth-ren-root} Fix a tree $(\bT,o)\in \mathcal{T}$ with root $o$, and assume that  \hyperref[def:ExponentialScaling]{\normalfont(ED)} holds for some $c_{\textup{low}}$, $\kappa >0$. Moreover, let $t=t(\ell)\geq 0$ satisfy $0 \le D_o(t) \le \ell$. Then  for all $c>0$
\begin{equation}\label{eq:RenewalBound}
P_{\bT}\Big( \psi_o(t) > (1+c) (\ell+1) \kappa^{-1}\eup^{c_{\textup{low}} (\ell+1)} \Big) \leq  \exp\Big(-(c-\log(1+c)) \ell \Big) \, .
\end{equation} 
\end{lemma}

\begin{proof}[Proof of Lemma \ref{lem:depth-ren-root}] 
Since $D_o(t) \le \ell$, there exists a site $y$ with $|y| \leq \ell+1$ and $\eta_t(y)=0$, such that the ray connecting $y$ to $o$ is fully occupied by particles. 
Thus,  $\psi_o(t)$ is bounded by the time a hole at level $\ell+1$ needs to travels to $o$. By \hyperref[def:ExponentialScaling]{\normalfont(ED)},
\[
 \psi_o(t)   \leq  \kappa^{-1}\exp(c_{\textup{low}} (\ell+1)) \sum_{i = 1}^{\ell+1} \om_{i}
\] holds for independent Exponential-$1$-distributed random variables $(\om_i)_{i\in [l+1]}$. Now 
\begin{align*}
 P\left( \sum_{i = 1}^{\ell+1} \om_i > (1+c) (\ell+1)  \right) \le \exp\left({- (c - \log(1+c))\ell}\right) 
\end{align*} by using Cram\'er's theorem. This yields an upper bound on the left-hand side in  \eqref{eq:RenewalBound}; see Theorem 2.2.3 \cite{DZ:LargeDeviations}.
\end{proof}
\begin{proof}[Proof of Proposition \ref{pro:EnteringParticles}] Recall that a particle can enter the tree if and only if the root is empty, and that particles are created at the root at rate $\lambda$. Thus
\begin{equation}\label{eq:TelescopicSum}
\tau_0^{i} - \tau_0^{i-1} \leq \psi_o(\tau_0^{i-1}) + \lambda^{-1} \om_i
\end{equation} 
holds for all $i \in [n]$ for some sequence $(\om_i)_{i \in [n]}$ of i.i.d.\ Exponential-$1$-distributed random variables. Recall  \eqref{def:cmu} where for $d_{\min}>1$, we take $c_o$ such that $M_n = c_{o}\log n$, and set $c_o=1/\log d_{\min}$ otherwise. Rewriting $\tau_0^{n}$ as a telescopic sum yields
\begin{align*}
P_{\bT}&( \tau_0^{n} > c n^{c_{\textup{low}}c_{o}+1} \log n  ) \leq P_{\bT}( \exists i\in [n] \colon \tau_0^{i} - \tau_0^{i-1}  > c n^{c_{\textup{low}}c_{o}} \log n  ) \notag\\
& \leq n \max_{i \in [n]}  P_{\bT}\left( \psi_o(\tau_0^{i-1}) > (c-3\lambda^{-1}) n^{c_{\textup{low}}c_{o}} \log n  \right) + n  P_{\bT}(  \om_1 > 3\log n )  \, .
\end{align*}
Together with Lemma \ref{lem:depthboundRoot} and Lemma \ref{lem:depth-ren-root} for $\ell=c_{o}\log n$, we obtain that 
\begin{align*}
 n \max_{i \in [n]}  P_{\bT}\left( \psi_o(\tau_0^{i-1}) > (c-3\lambda^{-1}) n^{c_{\textup{low}}c_{o}} \log n  \right) 
 &+ n  P_{\bT}(  \om_1 > 3\log n)  \leq  \frac{1}{n^2} + \frac{1}{n^2}
\end{align*}
holds for some $c>0$  with $\GW$-probability at least $1-2n^2$ for all $n$ sufficiently large.
\end{proof}

\subsection{Proof of the disentanglement theorem}

For the proof of Theorem \ref{thm:DisentanglementGWT} we have the following strategy. We wait until all $n$ particles have entered the tree. We then consider a level in the tree which was reached by no particle yet. For every vertex at that level as a starting point, we use the a priori bound on the disentanglement from Proposition \ref{lem:Early-sep}; see also Figure \ref{fig:TASEPDnMn}.  \\

\begin{figure}
\centering
\begin{tikzpicture}[scale=0.85, >=latex]

\draw[draw=none,fill=gray!20] (-1,-2) rectangle (1,2);

\draw[draw=none,fill=gray!20] (5,-2) rectangle (7,2);

	\node[shape=circle,scale=1.2,draw,fill=white] (Y2) at (-6,0){} ;
	\node[shape=circle,scale=1.2,draw,fill=white] (Z2) at (-3,0){} ;
	\node[shape=circle,scale=1.2,draw,fill=white] (A2) at (0,0){} ;
 	\node[shape=circle,scale=1.2,draw,fill=white] (B2) at (3,1){} ;
	\node[shape=circle,scale=1.2,draw,fill=white] (C2) at (3,-1){};
	\node[shape=circle,scale=1.2,draw,fill=white] (D2) at (6,0.5){} ;
	\node[shape=circle,scale=1.2,draw,fill=white] (E2) at (6,1.5){} ;
	\node[shape=circle,scale=1.2,draw,fill=white] (F2) at (6,-0.5){} ;
	\node[shape=circle,scale=1.2,draw,fill=white] (I2) at (6,-1.5){} ;
	\node[shape=circle,scale=1.2,draw,fill=white] (G2) at (0,-1.2){} ;	
	\node[shape=circle,scale=1.2,draw,fill=white] (H2) at (0,1.2){} ;	
	
	%\node[shape=circle,scale=0.9,draw,fill=red] (Y3) at (-6,0){} ;
	\node[shape=circle,scale=0.9,draw,fill=red] (Z3) at (0,-1.2){} ;
	\node[shape=circle,scale=0.9,draw,fill=red] (A3) at (0,1.2){} ;
 	\node[shape=circle,scale=0.9,draw,fill=red] (B3) at (6,1.5){} ;
	\node[shape=circle,scale=0.9,draw,fill=red] (C3) at (3,1){};
	
	\node (d) at (6, 0.5) {};
	\draw[->, thick] (C3)  to [out=-90,in=-120] (d);	

	\draw[thick] (Y2) to (Z2);
	\draw[thick] (A2) to (Z2);			
	\draw[thick] (A2) to (B2);		
	\draw[thick] (A2) to (C2);		
	\draw[thick] (B2) to (D2);		
	\draw[thick] (B2) to (E2);	
	\draw[thick] (C2) to (F2);	
	\draw[thick] (Z2) to (G2);	
	\draw[thick] (Z2) to (H2);	
	
	\draw[thick, dashed] (E2) to (7,1.7);			
	\draw[thick, dashed] (E2) to (7,1.3);	

	\draw[thick] (C2) to (I2);	
	
	\draw[thick, dashed] (D2) to (7,0.5);	
	
	\draw[thick, dashed] (F2) to (7,-0.7);		
	\draw[thick, dashed] (F2) to (7,-0.3);		
	
	\draw[thick, dashed] (I2) to (7,-1.5);		
	
	\draw[thick, dashed] (G2) to (1,-1.2);		
	\draw[thick, dashed] (H2) to (1,1.5);	
	\draw[thick, dashed] (H2) to (1,0.9);	
	
	\draw[thick, dashed] (Y2) to (-4.5,0.4);	
	\draw[thick, dashed] (Y2) to (-4.5,-0.4);

		\node[scale=0.9]  at (0,-2.4){$\sD_n=2$};
		\node[scale=0.9]  at (3,-2.4){$3$};
		\node[scale=0.9]  at (-3,-2.4){$1$};
		\node[scale=0.9]  at (6,-2.4){$\sM_n=4$};
		\node[scale=0.9]  at (-6,-2.4){$0$};	

	%	\node[scale=0.9]  at (0.3,-0.5){$o$};	
%		\node[scale=1,red]  at (-4.1,0.4){$\xi$};	
	%			\node[scale=0.9]  at (3.05,1.55){$y$};	
		%				\node[scale=0.9]  at (0.05,1.7){$x$};	
								\node[scale=1]  at (-5.9,-0.6){$o$};
\node[scale=1]  at (-6.85,0.7){$\lambda$};								
								
		\draw [->,line width=1] (-7.5,0) to [bend right,in=135,out=45] (Y2);								

\end{tikzpicture}
\caption[TASEP on trees and the generations $\sD_n$ and $\sM_n$]{\label{fig:TASEPDnMn}Visualization of the TASEP on trees and the different generations $\sD_n$ and $\sM_n$ involved in the proof for $n=4$. The particles are drawn in red. Note that it depends on the next successful jump of the particle at generation $3$, if the first $4$ particles are disentangled at generation $\sM_n=4$, i.e.\ they will disentangle if the particle jumps at the location indicated by the arrow.
}
\end{figure}
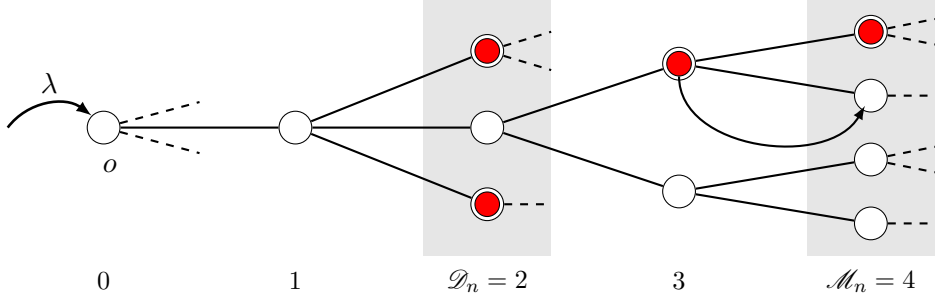

Starting from the empty initial configuration, we study the maximal generation which is reached until time $\tau_0^{n}$. The next lemma gives an estimate on the degrees of the vertices along the possible trajectories of the particles.

\begin{lemma} \label{lem:FurthestGenerationBound}
Let $(L_n)_{n \in \N}$ be an integer sequence such that $L_n \ge \tilde{c} \log n$ holds for some $\tilde{c} > 0$ and $n\in \N$. Then we can find a sequence $(\delta_n)_{n \in \N}$ with $\delta_n$ tending to $0$ with $n$ such that the following statement holds with $\GW$-probability at least $1-n^{-2}$ for all $n$ large enough:  for every site $x\in \cZ_{\lceil L_n(1+\delta_n) \rceil}$, there exists a site $y \leq x$, i.e.\ $y$ is on a directed path from the root to $x$, with $|y|\geq L_n$ and $\deg(y) \leq \log\log n$.
\end{lemma}

\begin{proof} 
It suffices to consider the case where the offspring distribution has infinite support. Using Markov's inequality, we see that with $\GW$-probability at least $1-(2n)^{-2}$, the Galton--Watson tree contains at most $(2n)^2 \mathfrak{m}^{L_n}$ sites at generation $L_n$. We denote by $(\bT_i)_{i \in [\abs{\cZ_{L_n}}]}$ the trees with roots $o_i$ attached to these sites. 
We claim that 
%by generation $\ce{ \log G_n}$ every tree $\bT_i$, 
with $\GW$-probability at least $1-(2n)^{-4} \mathfrak{m}^{-L_n}$, every ray $[o_{i},x]$ for $x$ at level $\ce{ \delta_n L_n}$ of $\bT_i$ contains at least one vertex which has at most $\log\log n$ neighbors. To see this, we use a comparison to a different offspring distribution.
Recall that the mean of the offspring distribution is $\mathfrak{m} < \infty$, and that $p_i$ is the probability of having precisely $i$ offspring. We define another offspring distribution for weights $(\bar{p}_i)_{i \in \{ 0,1,\dots\}}$, where
\begin{equation*}
\bar{p}_i := \begin{cases} p_i & \text{ for } i > \log\log n \\
1 - \sum\limits_{i=1}^{\fl{\log\log n}} p_i & \text{ for } i = 0 \\
0, & \text{ else. }
\end{cases}
\end{equation*} Let $\bar{\mathfrak{m}}_n$ denote the mean of the distribution given by $(\bar{p}_i)_{i \in \{ 0,1,\dots\}}$, and note that $\bar{\mathfrak{m}}_n \rightarrow 0$ holds when $n \rightarrow \infty$. Observe that the probability that all rays up to generation $\ce{ \delta_n L_n}$ contain at least one vertex of degree at most $\log\log n$ is equal to the probability that the tree with offspring distribution drawn according to $(\bar{p}_i)_{i \in \{ 0,1,\dots\}}$ dies out until generation $\ce{ \delta_n L_n}$. Using a standard estimate for Galton--Watson trees, this probability is at least $1-\bar{\mathfrak{m}}_n^{\ce{ \delta_n L_n}}$. 
Set 
\begin{equation}\label{eq:deltaNSequence}
\delta_n = - \frac{2L_n + 4 \log_\mathfrak m(2n)}{ L_n \log_{\mathfrak m} \bar{\mathfrak m}_n }
\end{equation} 
and note that $\delta_n \rightarrow 0$ holds when $n \rightarrow \infty$. From this, and $L_n \ge \tilde{c} \log n$ for some $\tilde{c} >0$,  for all $n$ large enough
%as a consequence of assumption \hyperref[def:ExponentialScaling]{\normalfont(ED)}, we can choose $(m(n))_{n \in \N}$ such that \eqref{eq:GenerationBound} as well as
\begin{equation*}
\bar{\mathfrak{m}}_n^{\ce{ \delta_n L_n}} \leq (2n)^{-4} \mathfrak{m}^{-L_n}
\end{equation*}
follows. We conclude with a union bound over all trees $\bT_i$ at level $L_n$.
\end{proof}
Next, for all $t \ge 0$, we let $\sS (t)$ denote the  generation 
\[
\sS (t) = %\max\{ \ell \geq 0 : \cJ_{\ell}(t) = 1 \} 
\max \{ \ell \geq 0 : \cJ_{\ell}(t) \geq 1 \}
\]    
when starting from the configuration where all sites are empty. 
\begin{lemma} \label{lem:HowManyEnter} Recall $(\sD_n)$ from \eqref{def:FurthestGeneration} and $(\delta_n)$ from \eqref{eq:deltaNSequence}.  Then $\P$-almost surely
\begin{equation}\label{eq:SDestimate}
\sS (\tau_0^{n}) \leq  (1+\delta_n)\sD_n  
\end{equation}  for all $n$ sufficiently large.
\end{lemma}

\begin{proof}

By Lemma \ref{lem:FurthestGenerationBound}, with $\GW$-probability at least $1-n^{-2}$, there exists some generation $\ell \geq \sD_n $ such that for every $i \in [n]$, the $i^{\text{th}}$ particle has at most $\log\log n$ neighbors. Let $\zeta_{i}$ be the holding time at this generation for particle $i$ and note that $\zeta_i$ satisfies
the stochastic domination 
\[
\zeta_{i} \succeq \om_{i}\sim \textup{Exp}(r^{\max}_{\sD_n} \log\log n ) \, .
\]  Set $t= cn^{c_{\textup{low}}c_{o}+1}\log n$ for $c>0$ sufficiently large  such that for all $n$ large enough
\begin{equation*} 
\GW\Big( P_\bT\big(  \sS (t) \ge  \sS (\tau^n_0\big)  \big) \ge 1 - \frac{2}{n^2} \Big) \ge \GW\Big( P_\bT\big(  \tau^n_0 \leq  t\big) \ge 1 - \frac{2}{n^2}  \Big) \geq 1- 2n^{-2}
\end{equation*}  using that $\sS(\cdot)$ is monotone increasing for the first inequality, and Proposition \ref{pro:EnteringParticles} for the second step. For the same choice of $t$ and using the definitions of $\sD_n$ and $\sS(t)$
\begin{align*}
P_\bT( \sS (t) &> \sD_n (1+\delta_n) ) \le P_\bT\Big(\min_{1 \le i \le n} \zeta_i <  t \Big) 
 \le P_\bT\Big(\min_{1 \le i \le n} \om_i <  t \Big) \leq \frac{c_1\log\log n}{ n \log^2 n} 
\end{align*} 
holds for some constant $c_1>0$ and all $n$ sufficiently large, with $\GW$-probability at least $1-n^{-2}$. An integral test shows that all error terms in  the above estimates are summable with respect to $n$, and we obtain \eqref{eq:SDestimate} by the Borel--Cantelli lemma.
\end{proof}

\begin{proof}[Proof of Theorem \ref{thm:DisentanglementGWT}]

Note that when the event in Lemma \ref{lem:HowManyEnter} occurs, $\P$-almost surely no ray contains more than $ \sD_n (1+\delta_n)$ particles out of the first $n$ particles  for all $n$ sufficiently large. We will use this observation to apply the a priori bound from Proposition \ref{lem:Early-sep} for all trees $(\bT^i)$ rooted at generation $\sD_n (1+\delta_n)$ which eventually contain at least one of the first $n$ particles. In the following, we assume that $\sD_n < n$. For $\sD_n \ge n$, we directly apply Proposition \ref{lem:Early-sep} for the original tree $\bT$ with $n$ particles. \\

We start with the case where $d_{\min} \geq 2$ holds. Let $\delta \in (0,1)$ be fixed and set
\begin{equation}\label{eq:FinalDisentangle1}
 \tilde{\sM}_n = \frac{1}{d_{\min} -1}(\sD_n(1+\delta_n)) + (2+\delta) \log_{1+\e} (n\sD_n) \,  .
\end{equation}
Moreover, we fix a tree $\bT^i$ rooted at generation $\sD_n (1+\delta_n)$ which eventually contains a particle. We claim that by Proposition \ref{lem:Early-sep}, all of the at most $\sD_n (1+\delta_n)$ particles entering $\bT^i$ are disentangled after $\tilde{\sM}_n$ generations in $\bT^i$ with $P_\bT$-probability at least $1 - cn^{-2-\delta}$ for some constant $c>0$. To see this, recall \eqref{def:DisentangleGeneration} and observe that
\begin{equation*}
F_{\sD_n (1+\delta_n)}(\tilde{\sM}_n) \leq (2+\delta)\log_{1+\e}( n \sD_n) \, .
\end{equation*}
We then apply \eqref{eq:DisentanglementStatementLemma} to obtain the claim. Note that this holds for {$\GW$-almost} every tree $(\bT,o) \in \mathcal{T}$. Moreover, the events that the particles disentangle on the trees $(\bT^i)$ are mutually independent, and we conclude using a union bound for the trees $(\bT^i)$. \\

Now suppose that $d_{\min} = 1$ holds. 
%We apply a similar argument as in the proof for $d_{\min}\geq 2$.
Recall $c_{o}$ from \eqref{def:cmu} and that $\delta \in (0,1)$ is fixed. Note that $\delta_n \leq \delta$ holds for all $n$ sufficiently large and set
\begin{equation}\label{eq:FinalDisentangle2} 
\tilde{\sM_n} = c_{o}(\sD_n (1+\delta)) + (2+\delta)c_{o} \log_{1+\e} (n\sD_n) \, .
\end{equation} Observe that $(2+\delta)\log_{1+\e} n\geq \log_2 n$  for all $n$ using the definition of~$\varepsilon$~in~\hyperref[def:UniformElliptic]{\normalfont(UE)}. Let $H_n=\sD_n (1+\delta)+(2+\delta)\log_{1+\e} (n\sD_n)$. Similar to the case $d_{\min} \geq 2$, we now combine  Proposition \ref{lem:Early-sep} and Lemma \ref{lem:coreLD} to see that $\P$-almost surely, all of the at most $\sD_n (1+\delta)$ particles entering $\bT^i$ are disentangled after $\tilde{\sM}_n$ generations in $\bT^i$ for all $i \in [n]$ and $n$ large enough. Compare \eqref{eq:FinalDisentangle1} and \eqref{eq:FinalDisentangle2} with $\sM_n$ in \eqref{def:DecouplingPoint} and \eqref{def:DecouplingPoint2} to conclude.
\end{proof}

%\begin{corollary}[Quenched estimates]
%\end{corollary}

%%%%%%%%%%%%%%%%%%%%%%%%%%%%%%%%%%%%%%%%%%%%%%%%%%%%%%
\section{Couplings} \label{sec:Couplings}
%For any finite $A \subseteq V(\bT)$ we define  
%\[
%r_{A}^{\min} =   \min_{x \in A, y \in \cC(x)} r_{x, y}, \qquad{\rm and } \qquad  r_{A}^{\max} =   \max_{x \in A, y \in \cC(x)} r_{x, y}. 
%\] 
%We are interested in minimum and maximum rates in two particular classes of $A$, namely subtrees rooted at a node $x \in V(\bT)$, truncated after $\ell$ generations
%\[
%\bT^{(x)}_\ell = \bigcup_{ y \succeq_\bT x: \| x-y \| \le \ell } y,
%\]
%and the strip of nodes between generations $k, \ell$ of the subtree $\bT^{(x)}$.
%\[
%\bT_{(k, \ell)}^{(x)} %= \bigcup_{i = k}^\ell \bigcup_{y \in \cZ_i^{(x)}} y 
%= \bT^{(x)}_\ell \setminus \bT^{(x)}_k.
%\]
%When $x = o$, we omit it from the notation, unless any confusion may arise.
%Also recall that for any $x \in V(\bT)$, 
%\[
%r_x = \sum_{y \in \cC(x)} r_{x,y},
%\]
%and we define the maximum and minimum rates for which a particle jumps from $\cZ_\ell$ to $ \cZ_{\ell+1}$ when no particles are blocking to be 
%\[
%\lambda^{\max}_{\ell} = \max_{x \in \cZ_\ell } r_x \quad \text{ and  } \quad \lambda^{\min}_{\ell} = \min_{x \in \cZ_\ell } r_x.
%\]

In this section, we discuss three methods of comparing the TASEP on trees to related processes via couplings. We start with the canonical coupling which allows us to compare the TASEP on trees for different initial configurations. Next, we introduce a comparison to independent random walks. This coupling is used to prove a lower bound on the time window in Theorem~\ref{thm:LinearCurrentL} and an upper bound on the window of generations in Theorem~\ref{thm:CurrentfixedT}. Our third model is a slowed down TASEP which is studied using an inhomogeneous LPP model. It is used to give an upper bound on the time window in Theorem~\ref{thm:LinearCurrentL} and a lower bound on the window of generations in Theorem~\ref{thm:CurrentfixedT}. In all cases, we fix a tree $\bT=(V,E,o) \in \mathcal{T}$ and a family of rates $(r_{x,y})_{x,y \in \Er}$ such that the TASEP is a Feller process.

\subsection{The canonical coupling}\label{sec:CanCoupling}

Let $(\eta^{1}_t)_{t \geq 0}$ and $(\eta^{2}_t)_{t \geq 0}$ denote two totally asymmetric simple exclusion processes on $\bT=(V,\Er)$ with transition rates $(r_{x,y})$, where particles are generated at the root at rates $\lambda_1$ and $\lambda_2$, respectively. Assume $\lambda_1 \leq \lambda_2$. The canonical coupling is the following joint evolution $(\eta^{1}_t,\eta^{2}_t)_{t \geq 0}$ of the two TASEPs. \\

For every edge $e=(x,y) \in \Er$, consider independent rate $r_{x,y}$ Poisson clocks. Whenever 
 a clock rings at time $t$ for an edge $(x,y)$, we try in both processes to move a particle from $x$ to $y$, provided that $\eta^{1}_t(x)=1-\eta^{1}_t(y)=1$ or $\eta^{2}_t(x)=1-\eta^{2}_t(y)=1$ holds.  We place a rate $\lambda_1$ Poisson clock at the root. Whenever the clock rings, we try to place a particle at the root in both processes. Furthermore, if $\lambda_1 \neq \lambda_2$, we place an additional independent rate $(\lambda_2-\lambda_1)$ Poisson clock at the root. Whenever this clock  rings, we try to place a particle at the root in $(\eta^{2}_t)_{t \geq 0}$.  \\
 
Let $\succeq$ denote the component-wise partial order on $\{ 0,1 \}^{V}$ and denote by $\mathbf{P}$ the law of the canonical coupling.
\begin{lemma}\label{lem:CanonicalCoupling} Let $(\eta^{1}_t)_{t \geq 0}$ and $(\eta^{2}_t)_{t \geq 0}$ be two TASEPs on trees within the above canonical coupling. Suppose that  $\lambda_1 \leq \lambda_2$ holds, then 
\begin{equation}
\mathbf{P}\left( \eta_t^{1} \preceq_{} \eta_t^{2} \text{ for all } t \geq 0 \mid  \eta_0^{1} \preceq \eta_0^{2} \right) = 1  .
\end{equation}
\end{lemma}

\begin{remark}\label{rem:AdditionalSources} Similarly, we can define the  canonical coupling for the TASEP on trees when we allow reservoirs of intensities $\lambda^{v}_1$ and $\lambda^{v}_2$ at all sites $v \in V$, respectively. The canonical coupling preserves the partial order $\succeq$ provided that $\lambda^{v}_1 \leq \lambda^{v}_2$ holds for all sites $v \in V$.
\end{remark}

\subsection{A comparison with independent random walks} \label{sec:IndependentRWs} We start by comparing the TASEP $(\eta_t)_{t \geq 0}$ on $\bT$ to independent biased random walks on $\bT$. Assume that the TASEP is started from some state $\eta$, which is --- in contrast to our previous assumptions --- not necessarily the configuration with only empty sites. We enumerate the particles according to an arbitrary rule and denote by $z^{i}_t$ the position of the $i^{\text{th}}$ particle at time $t \geq 0$. We define the waiting time $\sigma_{\ell}^{(i)} $ in level $\ell$ for all $i \in \Z$ and $\ell \in \N$  to be the time particle $i$ spends on generation $\ell$ once it  sees at least one empty site.
Recall $R^{\max}_\ell$ from \eqref{def:R_xdInverse} and, with a slight abuse of notation, let $\succeq$ denote the stochastic domination for random variables.  Then
\be\label{eq:WaitingTimeInequality}
\sigma_{\ell}^{(i)}  \succeq	 R^{\max}_\ell \om_{\ell}^{(i)} 
\ee  holds for all $i \in [n]$ and $\ell \geq 0$, where $\om_{\ell}^{(i)}$ are independent Exponential-$1$-distributed random variables.  We now define the independent random walks $(\tilde{\eta}_t)_{t \geq 0}$ started from $\eta$. \\

Each particle at level $\ell$ waits according to independent rate $(R^{\max}_\ell)^{-1}$ Poisson clocks, and jumps to a neighbor at generation $\ell+1$ chosen uniformly at random when the clock rings. When a particle is created in $(\eta_t)_{t \geq 0}$, create a particle in $(\tilde{\eta}_t)_{t \geq 0}$ as well. \\

 Note that in these dynamics, a site can be occupied by multiple particles at a time.  Let $\tilde{z}^{i}_t$ denote the position of the $i^{\text{th}}$ particle in $(\tilde{\eta}_t)_{t \geq 0}$ at time $t \geq 0$ and denote by $(\tilde{\cJ}_{\ell}(t))_{t \geq 0}$ the aggregated current of $(\tilde{\eta}_t)_{t \geq 0}$ at generation $\ell \in \N_0$. The following lemma is  immediate from  \eqref{eq:WaitingTimeInequality} and the construction of the  random walks $(\tilde{\eta}_t)_{t \geq 0}$.
\begin{lemma}\label{lem:CouplingToIRW} There exists a coupling $\tilde{\P}$ between the TASEP $(\eta_{t})_{t \geq 0}$ on $\bT$ and the corresponding independent random walks $(\tilde \eta_{t})_{t \geq 0}$ such that
\begin{equation}
\tilde{\P}\left(\abs{z^{i}_t} \leq \abs{\tilde{z}^{i}_t}  \text{ for all } i\in \N \mid \tilde{\eta}_0 =\eta_0\right) = 1 \ .
\end{equation} In particular, $\cJ_{\ell}(t) \leq \tilde{\cJ}_{\ell}(t)$ holds  for all $\ell \in \N_0$ and $t \geq 0$.
\end{lemma}

Using the comparison to independent random walks, we can give bounds on the current using estimates on weighted sums of Exponential random variables. We will frequently use the following estimates.

\begin{lemma} \label{lem:RWcomparision} For $\ell \in \N$ and 
$c_0, c_1, c_2, \ldots, c_\ell, t \geq 0$, set $S := \sum_{i=0}^\ell {c_i}^{-1}$ as well as $c:=\min_{i \in \{0,1, \ldots, \ell\}} c_i$.  Let $(\om_{i})_{i \in \{0,1, \ldots, \ell\}}$ be independent Exponential-$1$-distributed random variables.
Then for any $\delta \in (0,1)$, 
\begin{align*}
1- \frac{\eup^{- \delta c t}}{(1 - \delta) ^{ c S}} 
\le P\Big( \sum_{i = 0}^{\ell} \frac{1}{c_i} \om_{i} \leq t\Big) \leq \min\left(\frac{\exp(\delta  c t)}{(1+\delta)^{c S}}, 
\exp\Big( (\ell + 1) ( 1 + \log( t\ell^{-1} ) ) + \sum_{i=0}^{\ell} \log(c_i) \Big)
 \right) \, .
\end{align*}
\end{lemma}
\begin{proof} By Chebyshev's inequality, we see that 
\begin{align*}
P\Big( \sum_{i = 0}^{\ell} \frac{1}{c_i} \om_{i} \leq t\Big) \leq \eup^{\ell} \prod_{i=0}^\ell {E}\left[ \exp\left(-\frac{\ell}{t c_i} \om_{i} \right)\right] = \eup^{\ell}\exp\left(-\sum_{i=0}^{\ell} \log\left(1 + \frac{\ell}{t c_i}\right) \right)  
\end{align*}
holds. Since the logarithm is increasing, we can rearrange the sums to get the second upper bound. For the first upper bound, again apply Chebyshev's inequality for
\begin{align}\label{eq:SecondEquationLemmaRW}
P\Big( \sum_{i = 0}^{\ell} \frac{1}{c_i} \om_{i} \leq t\Big) \le \eup^{\delta  c t}\exp\left(-\sum_{i=0}^{\ell} \log\left(1 + \frac{\delta c}{c_i}\right) \right) \, . 
%\le  \eup^{\delta c t} \eup^{- \log(1+\delta) c S} 
\end{align}
Using concavity of the logarithm, we obtain for all $i \in \{0,1, \ldots, \ell\}$ and all $x>-1$ that
\begin{equation}\label{eq:logEstimate}
 \log\left( 1+ \frac{x c}{c_i} \right) \geq  \log\left( 1+ x\right) \frac{c}{c_i} \, .
\end{equation}
For $x=\delta$ in \eqref{eq:logEstimate}, together with \eqref{eq:SecondEquationLemmaRW}, this yields the first upper bound. For the lower bound, we use again Chebyshev's inequality and \eqref{eq:logEstimate} with $x=-\delta$ to get that
\begin{align*}
P\Big( \sum_{i = 0}^{\ell} \frac{1}{c_i} \om_{i} \geq t \Big) \le \eup^{-\delta c t}\exp\left(-\sum_{i=0}^{\ell} \log\left(1 - \frac{c \delta}{c_i}\right) \right)  \le  \frac{\eup^{- \delta  c t}}{(1 - \delta) ^{ c S}} \, .
\end{align*}
This finishes the proof of the lemma.
\end{proof} 
%\begin{remark} Note that the bounds in Lemma \ref{lem:RWcomparision} are in general not sharp, and can be refined, for example when for some $k>1$ the weights satisfy $c_i=k^{i}$  for all~$i \geq 1$; see Theorem 2.3  in \cite{DGJPS:Fragments}.
%\end{remark}

\subsection{A comparison with an inhomogeneous LPP model} \label{sec:LPP}  In this section, we compare the TASEP on $\bT$ to a slowed down exclusion process, which we study using last passage percolation (LPP) in an inhomogeneous environment. To describe this model, we will now give a brief introduction to last passage percolation, and refer the reader to \cite{S:SurveyPlaneModels,S:ReviewCornerGrowth} for a more  comprehensive discussion. 

Consider the lattice $\N\times \N$, and let $(\om_{i,j})_{i,j\in \N}$ be independent Exponential-$1$-distributed random variables. Let $\pi_{m,n}$ be an up-right \textbf{lattice path} from $(1,1)$ to $(m,n)$, i.e.\ %an ordered sequence of sites on $\N_0\times \N$ with
\[
\pi_{m,n} = \{ u_1 = (1,1), u_2, \ldots , u_{m+n} = (m,n): \, u_{i+1} - u_{i} \in \{(1,0), (0,1) \} \text{ for all } i\}\, .
\] The set of all up-right lattice paths from $(1,1)$ to $(m,n)$ is denoted by $\Pi_{m,n}$.  The \textbf{last passage time} in an environment $\omega$ is defined as 
\be
G^{\om}_{m,n} = \max_{\pi_{m,n} \in \Pi_{m,n}} \sum_{u \in \pi_{m,n}} \om_{u} \, ,
\ee
%i.e.\ it is the maximal sum of weights that one can collect on sites of an up-right path. 
for all $m,n\in \N$. Equivalently, the last passage times are defined recursively as  
\be\label{eq:LPP}
G^{\om}_{m,n} = \max\{G^{\om}_{m-1, n}, G^{\om}_{m, n-1}\} + \om_{m,n} \, ,
\ee with boundary conditions for all $k, \ell \in \N$ given by
\be\label{eq:LPP2}
G^{\om}_{1, \ell} = \sum_{j=1}^\ell \om_{1, j}\, , \qquad  G^{\om}_{k, 1} = \sum_{i=1}^k \om_{i, 1} \, .
\ee 
In the following, we will restrict the space of lattice paths, i.e.\ we consider the set of paths $A_m := \{ u = (u^1, u^2):  u^2 \ge u^1 - m  \} \cap \N\times \N$. For any $(i, j)$ in $\N\times \N$, we define 
\begin{equation*}
G^{\om}_{i,j}(A_m) = \max_{ \pi \in \Pi_{i,j}(A_m)} \sum_{ u \in \pi } \om_u \, ,
\end{equation*}
where $\Pi_{i,j}(A_m)$  contains all up-right paths from $(1,1)$ to $(i,j)$ that do not exit $A_m$, i.e.\
\begin{equation*}
\Pi_{i,j}(A_m) = \Big\{ \pi = \{ (1,1)=  u_1, \ldots u_{i+j} = (i,j) \} : u_{i+1} - u_i  \in \{(1,0), (0,1)\}, u_i \in A_m  \Big\} \, .
\end{equation*}
Based on the environment $\omega$, we define an environment $\tilde \omega=\{\om_{i,j}\}_{i \in \N, j \in \N}$ by
\be \label{eq:envLPP}
\tilde\om_{i,j} :=
\begin{cases} \vspace{0.2cm}
\frac{1}{r_{i-j-1}^{\min}} \om_{i,j} & \text{ if } j < i\, , \\
 \lambda^{-1} \om_{i,j} & \text{ if } j = i \, , \\ 
0, & \text{ else};
\end{cases}
\ee 
see Figure \ref{fig:LPPEnvironmentTASEP} for a visualization. The next lemma shows that the last passage times in~$\tilde{\omega}$ can be used to study the entering time of the $n^{\text{th}}$ particle in the TASEP on trees.
\begin{lemma} \label{lem:ExpoLPP} Let $m, n \in \N$ be such that $m \le \sM_n$ holds, where $\sM_n$ is defined in Theorem~\ref{thm:DisentanglementGWT}. Then there exists a coupling between $G^{\tilde \om}_{n, n + m}$ and the time $\tau_m^n$ of the TASEP on trees, defined in \eqref{def:TauUp}, such that $\P$-almost surely,  for all $n$ large enough 
 \be\label{eq:LPPvsEnteringTime}
G_{n + m,n}^{\tilde \om}(A_{\sM_n}) \ge \tau_m^n \, .
 \ee
\end{lemma}\begin{figure}
\centering
\begin{tikzpicture}[scale=0.9]

\foreach \x in{1,...,6}{
	\draw[gray!50,thin](\x,0) to (\x,5); }
\foreach \y in{1,...,4}{
	\draw[gray!50,thin](0,\y) to (7,\y); }
	
	\draw[thick] (0,0) to (7,0);
	\draw[gray!50,thin] (7,0) to (7,5);
	\draw[gray!50,thin] (7,5) to (0,5);
	\draw[thick] (0,5) to (0,0);

	%\node[shape=circle,scale=1.2,draw] (A2) at (0.5,-1){} ;
 %	\node[shape=circle,scale=1.2,draw] (B2) at (1.5,-1){} ;
	%\node[shape=circle,scale=1.2,draw] (C2) at (2.5,-1) {};
%	\node[shape=circle,scale=1.2,draw] (D2) at (3.5,-1){} ;
 %	\node[shape=circle,scale=1.2,draw] (E2) at (4.5,-1){} ;

%	\node[shape=circle,scale=0.9,draw,fill=red] (A3) at (A2){} ;
	%\node[shape=circle,scale=0.9,draw,fill=red] (C3) at (C2){} ;
	%\node[shape=circle,scale=0.9,draw,fill=red] (D3) at (D2){} ;

%	\draw (A2) --(B2);
%	\draw (C2) --(B2);
%	\draw (D2) --(E2);
%	\draw (C2) --(D2);
%	\draw[dashed] (E2) --(5.2,-1);

	%\draw [->,line width=1] (-0.5,-1) to [bend right,in=135,out=45] (A2);	

	%\draw [red,line width=1.2] (0,0) --  (0,2) -- (1,2) -- (2,2) -- (2,5) -- (5,5) -- (5,6) ;	

	\node (x1) at (0.5,0.5){$\lambda$} ;
	\node (x2) at (0.5,1.5){$0$} ;
	\node (x3) at (0.5,2.5){$0$} ;
	\node (x4) at (0.5,3.5){$0$} ;
	\node (x5) at (0.5,4.5){$0$} ;

	\node (x1) at (1.5,0.5){$r^{\min}_{0}$} ;
	\node (x2) at (1.5,1.5){$\lambda$} ;
	\node (x3) at (1.5,2.5){$0$} ;
	\node (x4) at (1.5,3.5){$0$} ;
	\node (x5) at (1.5,4.5){$0$} ;

	\node (x1) at (2.5,0.5){$r^{\min}_{1}$} ;
	\node (x2) at (2.5,1.5){$r^{\min}_{0}$} ;
	\node (x3) at (2.5,2.5){$\lambda$} ;
	\node (x4) at (2.5,3.5){$0$} ;
	\node (x5) at (2.5,4.5){$0$} ;

	\node (x1) at (3.5,0.5){$r^{\min}_{2}$} ;
	\node (x2) at (3.5,1.5){$r^{\min}_{1}$} ;
	\node (x3) at (3.5,2.5){$r^{\min}_{0}$} ;
	\node (x4) at (3.5,3.5){$\lambda$} ;
	\node (x5) at (3.5,4.5){$0$} ;

	\node (x1) at (4.5,0.5){$r^{\min}_{3}$} ;
	\node (x2) at (4.5,1.5){$r^{\min}_{2}$} ;
	\node (x3) at (4.5,2.5){$r^{\min}_{1}$} ;
	\node (x4) at (4.5,3.5){$r^{\min}_{0}$} ;
	\node (x5) at (4.5,4.5){$\lambda$} ;

	\node (x1) at (5.5,0.5){$r^{\min}_{4}$} ;
	\node (x2) at (5.5,1.5){$r^{\min}_{3}$} ;
	\node (x3) at (5.5,2.5){$r^{\min}_{2}$} ;
	\node (x4) at (5.5,3.5){$r^{\min}_{1}$} ;
	\node (x5) at (5.5,4.5){$r^{\min}_{0}$} ;
	
	\node (x1) at (6.5,0.5){$r^{\min}_{5}$} ;
	\node (x2) at (6.5,1.5){$r^{\min}_{4}$} ;
	\node (x3) at (6.5,2.5){$r^{\min}_{3}$} ;
	\node (x4) at (6.5,3.5){$r^{\min}_{2}$} ;
	\node (x5) at (6.5,4.5){$r^{\min}_{1}$} ;

\end{tikzpicture}
\caption[Environment to describe the slowed down TASEP]{\label{fig:LPPEnvironmentTASEP}Visualization of the environment which is used to describe the slowed down TASEP as a last passage percolation model. The numbers 
in the cells are the parameters of the respective Exponential-distributed random variables. The square at the bottom left of the grid corresponds to the cell $(1,1)$.}
\end{figure}

In order to show Lemma \ref{lem:ExpoLPP}, we require a bit of setup. Consider the event 
\begin{equation}\label{def:EventDn}
D_n := \{\text{the first $n$ particles disentangle by generation $\sM_n$}\} 
\end{equation}
which holds for all $n$ large enough by Theorem \ref{thm:DisentanglementGWT}. In particular, note that if $D_n$ holds, whenever one of the first $n$ particles reaches generation $\sM_n$, it no longer blocks any of the first $n$ particles. Moreover, observe that when it is possible to jump for particle $i$ from generation $\ell$, the time $\sigma_{\ell}^{(i)}$ until this jump is performed is stochastically dominated by an Exponential-distributed random variable with the smallest possible rate out from generation $\ell$. In other words, the inequality
\[
\sigma_{\ell}^{(i)} \preceq\ \frac{1}{r^{\min}_{\ell}} \om_{\ell+i+1,i}
\] holds for all $i,\ell \in \N$. \\

We construct now a slowed down TASEP $(\tilde{\eta}_t)_{t \geq 0}$ where the $i^{\text{th}}$ particle waits a time of $(r^{\min}_{\ell})^{-1} \om_{\ell+i+1,i}$ to jump from generation $\ell$ to $\ell+1$, but only after particle $i-1$ left generation $\ell+1$. Moreover, we assume without loss of generality that all particles follow the trajectories of the original dynamics $({\eta}_t)_{t \geq 0}$.  As before, let $z^{i}_t$ and $\tilde{z}^{i}_t$ denote the position of the $i^{\text{th}}$ particle in $({\eta}_t)_{t \geq 0}$ and $(\tilde{\eta}_t)_{t \geq 0}$, respectively. The following lemma is immediate from the construction of the two processes.

\begin{lemma} \label{lem:LPPComparison}There exists a coupling $\tilde{\P}$ between the TASEP $(\eta_{t})_{t \geq 0}$ on $\bT$ and the corresponding slowed down dynamics $(\tilde\eta_{t})_{t \geq 0}$ such that for any common initial configuration
\begin{equation}
\tilde{\P}\left( \abs{\tilde z^{i}_t} \leq \abs{z^{i}_t}  \text{ for all } i\in [n] \right) = 1 \ .
\end{equation} 
\end{lemma}
\begin{proof}[Proof of Lemma \ref{lem:ExpoLPP}]
It suffices to show that the time in which the $n^{\text{th}}$ particle reaches generation $m$ in the slowed down dynamics has the same law as $G_{n + m,n}^{\tilde \om}(A_{\sM_n})$. 
Let $\widetilde G_{m,n}$ be the time the $n^{\text{th}}$ particle jumped $m-n$ times in the slowed down process and note that for all $m,n$
\begin{align*}
\widetilde G_{m,n} =\max(\widetilde G_{m-1,n} ,\widetilde G_{m,n-1})+  \tilde\om_{m, n} \, .
\end{align*} Moreover, 
\[
\widetilde G_{0,m} =\sum_{\ell=1}^{m} \tilde \om_{0, \ell} , \qquad \widetilde G_{\ell,1}  = \sum_{k=1}^{\ell} \tilde\om_{k,1} \, .
\]
The right-hand side of the last three stochastic equalities are the recursive equations and initial conditions for the one-dimensional TASEP, in which particle $i$ waits on site $\ell$  for $(r^{\min}_{\ell})^{-1} \om_{\ell, \ell+1}$ amount of time, after $\ell+1$ becomes vacant.  Note that any maximal path from $(0,1)$ up to $(n, n+\sM_n)$ will never touch the sites for which the environment is $0$, so the passage times in environment \eqref{eq:LPP} and \eqref{eq:LPP2}  coincide with those in environment \eqref{eq:envLPP}, as long as we restrict the set of paths to not cross the line $\ell - i =  \sM_n$. 
For any time $t \geq 0$, on the event $D_n$, this yields 
\begin{equation*} 
P_\bT(  \cJ_{m}(t) \le n, D_n )  \le P_\bT ( \widetilde \cJ_{m}(t) \le n, D_n ) \leq P_\bT (  G_ {n + m,n}^{\tilde \om}(A_{\sM_n}) \ge t ) \, .
\end{equation*} 
 We set $t = \tau_{m}^{n}$ and conclude as $D_n$ holds $\P$-almost surely for all $n$ large enough.
\end{proof}

We use this comparison to an inhomogeneous LPP model to give a rough estimate on the time $\tau_m^n$ for general transition rates. Note that this bound can be refined when we have more detailed knowledge about the structure of the rates.

\begin{lemma} \label{lem:ExpoEstimate}
Recall $\sM_n$ from Theorem \ref{thm:DisentanglementGWT} and fix $\alpha>0$. Then
\be\label{eq:LPPUB}
P_{\bT}\Big(G_{n + \sM_n,n}^{\tilde \om}(A_{\sM_n}) \le \frac{ 4(1+ \alpha) }{ \min\{ \lambda, \min_{ |x| \leq \mathcal{M}_n} r_{x,y} }(n+ \sM_n)  \Big) \ge 1- \eup^{-c n}
\ee holds for some constant $c=c(\alpha)>0$ with $\lim_{\alpha \rightarrow \infty}c(\alpha)=\infty$.
\end{lemma}

\begin{proof} Let $G_{m,n}^{(1)}$ be the passage time up to $(m, n)$ in an i.i.d.\ environment with Exponential-$1$-distributed  weights. Observe that we have the stochastic domination
\begin{equation}
\label{eq:tagged1}
  G_{n + \sM_n,n}^{\tilde \om}(A_{\sM_n}) \preceq G_{n+\sM_n, n+\sM_n}^{\tilde \om}(A_{\sM_n})\preceq \big(\min\limits_{\abs{x}\leq \sM_n  }r_{x} \big)^{-1} G_{n+\sM_n, n+\sM_n}^{(1)}\, .
\end{equation}
For all $\alpha>0$, we obtain from Theorem 4.1 in  \cite{S:CouplingInterfaces} that
\be\label{eq:tagged2}
P_{\bT}\Big(  G_{M,M}^{(1)} \le 4(1+\alpha)M  \Big) \ge 1- \eup^{-c M}
\ee holds for some $c=c(\alpha)>0$ with $\lim_{\alpha \rightarrow \infty}c(\alpha)=\infty$ and all $M \in \N$, where the constant $c(\alpha)$ is an explicitly known rate function.  Combine \eqref{eq:tagged1} and  \eqref{eq:tagged2} to conclude.
\end{proof}

\section{Proof of the current theorems} \label{sec:CurrentTheorems}

We  have now all tools to prove Theorem \ref{thm:LinearCurrentLSPECIAL} and Theorem \ref{thm:CurrentfixedTSPECIAL}. In fact, we will prove more general theorems which allow for any transition rates $(r_{x,y})$ satisfying the assumptions  \hyperref[def:UniformElliptic]{\normalfont(UE)} and \hyperref[def:ExponentialScaling]{\normalfont(ED)}. We start with a  generalization of Theorem~\ref{thm:LinearCurrentLSPECIAL} on the current in a time window $[t_{\textup{low}}, t_{\textup{up}}]$. Recall the notation from Section \ref{sec:ModelResults}. In particular, recall \eqref{def:R_xdInverse}, and set
\begin{equation}
 \rho_{\ell} := \min_{i \leq \ell} \max_{x \in \cZ_i}  r_{x} \, .
\end{equation}
For the lower bound of the time window,  we define $$t_{\textup{low}} := \max\big( t_1^{\textup{low}}, t_2^{\textup{low}} \big)$$ with
\begin{align}  \label{def:LowerBoundTime1}
t_1^{\textup{low}} &:= R^{\max}_{0,\ell_n}\big(1- 2\big(R^{\max}_{0,\ell_n}\rho_{\ell_n}\big)^{-\frac{1}{3}}\log R^{\max}_{0,\ell_n} \big) \\ \label{def:LowerBoundTime2} t_2^{\textup{low}} &:=\frac{\ell_n}{2}  \exp\Big(\frac{1}{\ell_n+1}\sum_{i=0}^{\ell_n} \log R^{\max}_{i} \Big) \,  .
\end{align}
%R^{\max}_{\ell_n} \frac{\log(1+\delta)}{2\delta}+ 
Note that both terms in the maximum can give the main contribution in the definition of $t_{\textup{low}}$, depending on the rates. For the upper bound, we define
\be\label{def:thetaRoot}
			\theta := \liminf_{n\to \infty} (\min_{\sM_n < i \leq \ell_n} r^{\min}_{i}) R^{\min}_{\sM_n,\ell_n} \in [0,\infty]
			\ee
and fix some $\delta \in (0,1)$. We let $t_{\textup{up}}=t_{\textup{up}}(\delta)$ be
%\begin{equation} \label{def:UpperBoundTime}
%t_{\textup{up}} = \begin{cases} 5(n+\sM_n)\big(\min\limits_{\abs{x}\leq \sM_n  }r_{x}\big)^{-1} + \left(2-\frac{4}{\theta}\right) R^{\min}_{\sM_n,\ell_n} & \text{ if } \theta< \infty \\
%5(n+\sM_n)\big(\min\limits_{\abs{x}\leq \sM_n  }r_{x}\big)^{-1} + (1+\theta_n) R^{\min}_{\sM_n,\ell_n} & \text{ if } \theta= \infty
%\end{cases} 
%\end{equation}
\begin{equation} \label{def:UpperBoundTime}
t_{\textup{up}} :=  \frac{5(n+\sM_n)}{\min\limits_{\abs{x}\leq \sM_n  }r_{x}} + \left[\mathds{1}_{\{\theta<\infty\}}\left(1+\delta-\frac{2\log \delta}{\theta\delta}\right) + \mathds{1}_{\{\theta=\infty\}}(1+\theta_n)\right]R^{\min}_{\sM_n,\ell_n} 
\end{equation}
with some sequence $(\theta_n)_{n \in \N}$ tending to $0$ satisfying \begin{equation}\label{def:ThetaN}
\liminf_{n \rightarrow \infty} \frac{1}{\theta_n} (\min_{\sM_n < i \leq \ell_n} r^{\min}_{i}) R^{\min}_{\sM_n,\ell_n} = \infty 
\end{equation} when $\theta=\infty$. Consider the first $n$ particles which enter the tree, starting with the configuration which contains only empty sites.
The following theorem states that we see at least an aggregated current in $[t_{\textup{low}},t_{\textup{up}}]$ of order $n$.
\begin{theorem}
\label{thm:LinearCurrentL} Suppose that  \hyperref[def:UniformElliptic]{\normalfont(UE)} and \hyperref[def:ExponentialScaling]{\normalfont(ED)} hold and let $(\ell_n)_{n \in \N}$ be a sequence of generations  with $\ell_n \geq \sM_n$ for all $n \in \N$. Fix $\delta \in(0,1)$ and let $t_{\textup{low}}$ and $t_{\textup{up}}=t_{\textup{up}}(\delta)$ be given in \eqref{def:LowerBoundTime1}, \eqref{def:LowerBoundTime2} and \eqref{def:UpperBoundTime}. Then $\P$-almost surely
\begin{equation}
\lim_{n\to \infty} \cJ_{\ell_n}(t_{\textup{low}}) =0\, , \qquad  \liminf_{n\to \infty}\frac{1}{n}  \cJ_{\ell_n}(t_{\textup{up}}) \ge 1- \delta  \, .
\end{equation}
\end{theorem} 
Note that Theorem \ref{thm:LinearCurrentL} indeed implies Theorem \ref{thm:LinearCurrentLSPECIAL} for rates which satisfy \eqref{def:ExpDecNew}.

\begin{proof} We start with the lower bound involving $t_{\textup{up}}$.
Recall $D_n$ from \eqref{def:EventDn} as the event that the first $n$ particles are disentangled at generation $\sM_n$, and  $\tau_n^{\sM_n}$ from \eqref{def:TauUp} as the first time such that the first $n$ particles have reached generation $\sM_n$. Set
\[
t_1 = 5(n+\sM_n)\big(\min\limits_{\abs{x}\leq \sM_n  }r_{x}\big)^{-1}
\] and define $t_2:=t_{\textup{up}}-t_1$. Combining Theorem \ref{thm:DisentanglementGWT},  Lemma \ref{lem:ExpoLPP} and Lemma \ref{lem:ExpoEstimate}, we see that
\begin{equation}\label{eq:HighProbEvent}
D_n \cap \{\tau_{\sM_n}^n   \leq t_1\}
\end{equation}
holds $\P$-almost surely for all $n$ sufficiently large. In words, this means that all particles have reached generation $\sM_n$ by time $t_1$ and perform independent random walks after level $\sM_n$. We claim that it suffices to show that
\begin{equation}\label{eq:DecouplingT2}
p:= P_{\bT}\left( \sum_{i=\sM_n}^{\ell_n} \frac{\om_i}{r^{\min}_{i}}> t_2\right) < \delta
\end{equation} holds, where $(\om_i)$ are independent Exponential-$1$-distributed random variables. To see this, let $B_i$ be the indicator random variable of the event that the $i^{\normalfont th}$ particle did not reach level $\ell_n$ by time $t_{\textup{up}}$. From \eqref{eq:DecouplingT2}, we obtain that $(B_i)_{i \in [n]}$ are stochastically dominated by independent Bernoulli-$p$-random variables when conditioning on the event in  \eqref{eq:HighProbEvent}. Hence, we obtain that
\begin{align*}\label{eq:DominationIndependents}
P_{\bT}\left(  \cJ_{\ell_n}(t_{\textup{up}}) < (1-\delta) n \ \Big| \  D_n , \tau_{\sM_n}^n \leq t_1 \right) & \leq P_{\bT}\left( \sum_{i=1}^{n} B_i \geq \delta n \ \Big| \ D_n , \tau_{\sM_n}^n \leq t_1 \right) \notag \\ 
&\leq \eup^{-\delta n}(1+\eup^{pn})
\end{align*} holds using Chebyshev's inequality for the second step. Together with a Borel--Cantelli argument and \eqref{eq:HighProbEvent}, this proves the claim. \\

In order to verify \eqref{eq:DecouplingT2}, we distinguish two cases depending on the value of $\theta$ defined in \eqref{def:thetaRoot}. Suppose that $\theta< \infty$ holds. Then by Lemma \ref{lem:RWcomparision} and a calculation, we obtain that
\begin{align*}%\label{eq:FirstEstimateLBCurrent}
P_{\bT}\left( \sum_{i=\sM_n}^{\ell_n} \frac{\om_i}{r^{\min}_{i}}> t_2\right) &< \exp\Big(  (\min_{\sM_n < i \leq \ell_n} r^{\min}_{i}) R^{\min}_{\sM_n,\ell_n} \big(-\delta-\delta^2+\frac{2\log \delta}{\theta} - \log (1- \delta)\big) \Big) \\
&\leq \exp\left(  \theta \big(-\delta \big(1-{\theta}^{-1}\log \delta\big) + \delta^2\big) \right) \leq \delta
\end{align*} holds for all $n$ large enough and $\delta \in (0,1)$, using the Taylor expansion of the logarithm for the second step. Similarly, when $\theta= \infty$, we  apply Lemma \ref{lem:RWcomparision} to see that
\begin{equation}\label{eq:SecondEstimateLBCurrent}
P_{\bT}\left( \sum_{i=\sM_n}^{\ell_n} \frac{\om_i}{r^{\min}_{i}}> t_2\right) \leq \exp\left(  (\min_{\sM_n < i \leq \ell_n} r^{\min}_{i}) R^{\min}_{\sM_n,\ell_n} \big(-\delta \big(\theta_n\log \delta\big) + \delta^2\big) \right) 
\end{equation} holds for all $n$ large enough and some sequence $(\theta_n)_{n \in \N}$ according to \eqref{def:ThetaN}. In this case, we obtain that for any fixed $\delta \in (0,1)$, the right-hand side in \eqref{eq:SecondEstimateLBCurrent} converges to $0$ when $n \rightarrow \infty$. Thus,  we obtain that \eqref{eq:DecouplingT2} holds for both cases depending on $\theta$, which gives the lower bound. \\
%In either case, define auxiliary Bernoulli 0-1 variables $\mathscr B_i$, with probability of success $p_{\ell_n}$. Since we have restricted on the event $D_n$, all $n$ particles perform independent random walk after generation $\sM_n$. Then, particle $i$ did not cross generation $\ell$ by time $T_1 + T_2$ if and only if $\mathscr B_i=1$.   Then 
%\begin{align*}
%P_{\bT}\bigg\{ {\normalfont card}\Big\{ i: &\sum_{k=\sM_n+1}^\ell \frac{1}{\lambda_k^{\min}}\om^{(i)}_k >   T_2 \Big\} > (1-\alpha)n \bigg\} \le
%P_{\bT}\bigg\{ \sum_{i=1}^n \mathscr B_i > (1-\alpha)n \bigg\}\\
%&= P_{\bT}\bigg\{ \sum_{i=1}^n (\mathscr B_i - p_{\ell_n}) > n (1-\alpha - p_{\ell_n}) \bigg\} \le \eup^{-n c(\alpha)},
%\end{align*}
%for some uniform constant $c(\alpha)$. The last exponential inequality follows when $1-\alpha>  p_{\ell_n}$ which we have verified in the calculations. 

Next, for the upper bound, we use a comparison to the independent random walks $(\tilde{\eta}_t)_{t \geq 0}$ defined in Section \ref{sec:IndependentRWs}. By Lemma \ref{lem:CouplingToIRW}, 
\begin{equation*}
P_{\bT}\left( \cJ_{\ell_n}(t_{\textup{low}}) \le  \delta \right) \geq P_{\bT}\left( \tilde\cJ_{\ell_n}(t_{\textup{low}}) \le  \delta \right)
\end{equation*}
holds  for all $\delta>0$, where $(\tilde\cJ_t)_{t \geq 0}$ denotes the current with respect to $(\tilde{\eta}_t)_{t \geq 0}$.
Fix some $\delta>0$ and let $(\om_i)_{i \in \N_0}$ be independent Exponential-$1$-distributed random variables. We claim that the probability for a particle in $(\tilde{\eta}_t)_{t \geq 0}$ to reach level $\ell_n$ is bounded from above by
\begin{equation}\label{eq:CurrentSecondEstimate}
P_{\bT}\bigg(\sum_{i=0}^{\ell_n-1} \frac{\om_i}{r_i^{\max}} \leq t_{\textup{low}} \bigg) \leq \frac{1}{2\lambda t_{\textup{low}} } 
\end{equation} for all $n$ sufficiently large, where we recall that particles enter the tree at rate $\lambda>0$. To see this, we distinguish two cases. Recall the construction of $t_{\textup{low}}$ in \eqref{def:LowerBoundTime1} and \eqref{def:LowerBoundTime2}, and assume that $t_{\textup{low}}=t_1^{\textup{low}}$. By the first upper bound in Lemma \ref{lem:RWcomparision}, 
\begin{equation*}
t_{\textup{low}}P_{\bT}\bigg(\sum_{i=0}^{\ell_n} \frac{\om_i}{r_i^{\max}} \leq t_{\textup{low}} \bigg) \leq  t_1^{\textup{low}}
\exp\left({ \delta  \rho_{\ell_n} t_1^{\textup{low}}-\rho_{\ell_n} R_{0,\ell_n}^{\max} \log(1+\delta)}\right)
\end{equation*} holds for all $\delta \in (0,1)$. For $\delta=(\rho_{\ell_n} R_{0,\ell_n}^{\max})^{-1/2}$ and using the Taylor expansion of the logarithm, we see that the right-hand side in \eqref{eq:CurrentSecondEstimate} converges to $0$ when $n \rightarrow\infty$.  Similarly, for $t_{\textup{low}}=t_2^{\textup{low}}$ the second upper bound in Lemma \ref{lem:RWcomparision} yields 
\begin{equation*}
t_{\textup{low}} P_{\bT}\bigg(\sum_{i=0}^{\ell_n} \frac{\om_i}{r_i^{\max}} \leq t_{\textup{low}} \bigg) \leq  
t_2^{\textup{low}} \exp\left({\ell_n (1+ \log t_2^{\textup{low}}-\log \ell_n ) - \sum_{i=0}^{\ell_n}\log R^{\max}_{i}} \right) \, ,
\end{equation*} where the right-hand side converges to $0$ for $n\rightarrow \infty$ using the definition of $t_2^{\textup{low}}$ and comparing the leading order terms. Since particles enter in both dynamics at the root at rate $\lambda$, note that for all $n$ large enough, at most $\frac{5}{4}\lambda t_{\textup{low}}$ particles have entered by time~$t_{\textup{low}}$. By  Chebyshev's inequality together with \eqref{eq:CurrentSecondEstimate},  $\P$-almost surely no particle has reached generation $\ell_n$ by time~$t_{\textup{low}}$ for all $n$ sufficiently large. 
\end{proof}

\begin{example}[$d$-regular tree, Constant rates (C)] \label{ex:TwindowConstant}  Let us choose $\ell_n$ to be $\ell_n = 1+2\sM_n$ for all $n\in \N$, where we recall from Example \ref{ex:intro} that $\sM_n$ is of order $n$ for constant rates. Then 
\[
R^{\min}_{\sM_n, \ell_n} = \frac{\sM_n}{d-1} = R^{\max}_{\sM_n, \ell_n}, \quad \rho_{\ell_n} = d-1 \, ,
\]
which give $\theta=\infty$. Hence, if we choose $\theta_n = 1/\log n$, we have 
\[
t_{\textup{up}} = \frac{5(n+\sM_n)}{d-1} + \Big(1 + \frac{1}{\log n}\Big) \frac{\sM_n}{d-1} = \frac{1}{d-1}(5n + 6 \sM_n) + o(n) \, .
\]
For $t_{\textup{low}}$ we have 
\[
t_{\textup{low}} = \max\Big\{\frac{2\sM_n+1}{d-1}(1 - o(1)), \frac{2\sM_n+1}{2}\exp\{ - \log(d-1)\} \Big\}=\frac{2\sM_n+1}{d-1}(1 - o(1)) \, .
\]
\end{example}
For more examples when the rates decay polynomially or exponentially we refer to Section~\ref{sec:RegularTreeTASEP}.

Now let $t$ be a fixed time horizon and define an interval $[L_{\textup{low}}, L_{\textup{up}} ]$ of generations. 
 Recall $\sM_{n}$ from Theorem \ref{thm:DisentanglementGWT} and define the generations
\begin{equation}\label{def:UpperLowerGeneration}
L_{\textup{low}} := \sM_{n_{t}}  \quad \text{and} \quad L_{\textup{up}}:=  \min(L^{\textup{up}}_1,L^{\textup{up}}_2) + 1
\end{equation}
for $n_t$ from  \eqref{eq:goodC} and, recalling \eqref{def:minandmax},
\begin{align*}
L_1^{\textup{up}} := \inf \Big\{ \ell : \log \ell - \frac{1}{\ell+2}\sum_{i=1}^\ell \log r_{i}^{\max} \ge  \log t +2  \Big\} \, , \ 
L_2^{\textup{up}} := \inf\Big\{ \ell :  R_{0,\ell}^{\max} \ge t+t^\frac{2}{3} \Big\}  \, .
\end{align*} 
Since $r_{i}^{\max}$ is bounded from above uniformly in $i$, $L^{\textup{up}}_1$ and $L^{\textup{up}}_2$ are both finite.
% More precisely,  if $\liminf_{ \ell \rightarrow \infty} \rho_{\ell} R_{0,\ell}^{\max} = \infty$, we have that $L^{\textup{up}}_1$ is finite and for $\liminf_{ \ell \rightarrow \infty} \rho_{\ell} R_{0,\ell}^{\max} < \infty$ that $L^{\textup{up}}_2$ is finite

The following theorem is the dual result of Theorem \ref{thm:LinearCurrentL}. Recall $n_t$ from \eqref{eq:goodC}. We are interested in a window of generations $[L_{\textup{low}}, L_{\textup{up}} ]$ where we can locate the first $n_t$ particles.
\begin{theorem}
\label{thm:CurrentfixedT} Suppose that  \hyperref[def:UniformElliptic]{\normalfont(UE)} and \hyperref[def:ExponentialScaling]{\normalfont(ED)} hold.  Then the aggregated current through generations $L_{\textup{low}}$ and $L_{\textup{up}}$ satisfies $\P$-almost surely
\begin{equation}\label{eq:CurrentfixedBounds}
\limsup_{t\to \infty} \cJ_{L_{\textup{up}}}(t) =0 \,  , \qquad  \liminf_{t\to \infty}\frac{1}{n_{t}} \cJ_{L_{\textup{low}}}(5 t)  \ge 1 \, .
\end{equation}
\end{theorem}
Note that Theorem \ref{thm:CurrentfixedT} implies Theorem \ref{thm:CurrentfixedTSPECIAL} for rates which satisfy \eqref{def:ExpDecNew}, keeping in mind that in the setup of Theorem \ref{thm:CurrentfixedTSPECIAL}, there exist some $c>0$ such that $n_{5t} \leq c n_t$ for all $t \geq 0$.

\begin{proof} Let us start with the bound involving $L_{\textup{up}}$. Let $(\om_i)_{i \in \N_0}$ be independent Exponential-$1$-distributed random variables. Note that $\P$-almost surely, no more than $2\lambda t$ particles have entered the tree by time $t$ for all $t>0$ large enough. Using a similar argument as after \eqref{eq:CurrentSecondEstimate} in the proof of Theorem \ref{thm:LinearCurrentL}, it suffices to show that 
\begin{align}\label{eq:Llowbound}
\lim_{t \rightarrow \infty} 2\lambda t  P_{\bT}\left( \sum_{i=0}^{L_{\textup{up}}} \frac{\om_i}{r_i^{\max}} \leq t\right) = 0\, .
\end{align} By Lemma \ref{lem:RWcomparision} and using the definition of $L_1^{\textup{up}}$
\begin{align*}%\label{eq:L1lowbound1}
 t P_{\bT}\left( \sum_{i=0}^{L_1^{\textup{up}}} \frac{\om_i}{r_i^{\max}} \leq t \right) \leq  \exp\left({ L_1^{\textup{up}}(1+\log t-\log L_1^{\textup{up}}) +\log t + \sum_{i=0}^{L_1^{\textup{up}}} \log r_{i}^{\max}}\right) \, ,
\end{align*} where the right-hand side converges to $0$ for $t\rightarrow \infty$.
Moreover, by Lemma \ref{lem:RWcomparision}
\begin{align}\label{eq:SecondEstimatePropL}
 t P_{\bT}\left( \sum_{i=0}^{L^{\textup{up}}_{2}} \frac{\om_i}{r_i^{\max}} \leq t \right) \leq   \frac{t \exp(\delta \rho_{L^{\textup{up}}_{2}} t)}{\exp\left({\rho_{L^{\textup{up}}_{2}}R^{\max}_{0,L^{\textup{up}}_{2}}}\log(1+\delta)\right)}
\end{align}
 holds for any $\delta \in (0,1)$ which may also depend on $t$. Note that $\sup_{\ell \in \N} \rho_\ell< \infty$ holds by our assumptions that the transition rates are uniformly bounded from above. Set $\delta= 2 (t^{2/3}\rho_{L^{\textup{up}}_{2}})^{-1} \log t$ for all $t$ large enough. Using the definition of $L^{\textup{up}}_{2}$ and the Taylor expansion of the logarithm, we conclude that the right-hand side in \eqref{eq:SecondEstimatePropL} converges to $0$ for $t\rightarrow \infty$. Since $L_{\textup{up}}=\min(L^{\textup{up}}_{1},L^{\textup{up}}_{2})$, we obtain \eqref{eq:Llowbound}. \\
 
For the remaining bound in Theorem \ref{thm:CurrentfixedT}, recall the slowed down exclusion process from Section \ref{sec:LPP}.  By Lemma \ref{lem:ExpoLPP} and Lemma \ref{lem:ExpoEstimate}, note that  for some $c>0$ 
\begin{align}\label{eq:CurrentReduced}
P_{\bT}( \cJ_{L_{\textup{low}}}( 5 t) < n_{t} ) \leq P_{\bT}\Big(  G^{\tilde{\om}}_{n_t+L_{\textup{low}},n_t} \geq 5 t   \Big)
 \le \eup^{-c n_{t}}
 \end{align} holds $\P$-almost surely when $t$ is sufficiently large. 
 Consider a sequence of times $(t_i)_{i \in \N}$ such that $t_i \rightarrow \infty$ as $i \rightarrow \infty$ and
\begin{equation}
\label{eq:LiminfStep} \lim_{i \rightarrow \infty}\frac{\cJ_{L_{\textup{low}}(5t_i)}(5t_i)}{n_{t_i}} = \liminf_{t \rightarrow \infty}  \frac{\cJ_{L_{\textup{low}}(5t)}(5t)}{n_t}
\end{equation} 
By possibly removing some of the $t_i$'s, we can assume without loss of generality that $n_{t_i} < n_{t_{i+1}}$. This way,  $n_{t_i} \geq i$ for all $i\in \N$. Therefore by \eqref{eq:CurrentReduced} and the Borel--Cantelli lemma, we obtain that
\begin{align*}
\cJ_{L_{\textup{low}}}( 5 t_i) \geq n_{t_i} 
 \end{align*} holds almost surely for all $i$ large enough. Theorem \ref{thm:CurrentfixedT} follows from \eqref{eq:LiminfStep}.
\end{proof}

\begin{remark} \label{rem:Replacement}Note that the bound in Theorem \ref{thm:CurrentfixedT} involving $L_{\textup{low}}$ continues to hold when we replace $n_t$ by some $n$ with $n_{t} \geq  n > c^{\prime} \log t$.
\end{remark}

\section{Current theorems for the TASEP on regular trees}  \label{sec:RegularTreeTASEP} 

In this section, we let the underlying tree be a $d$-regular tree, i.e.\ we assume that the offspring distribution is the Dirac measure on $d-1$ for some $d\geq 3$. Our goal is to show how the results of  Theorems~\ref{thm:LinearCurrentL} and \ref{thm:CurrentfixedT} can be refined when knowing the structure of the tree and the rates. This is illustrated in Section~\ref{sec:RegularTASEPpoly} for polynomially decaying rates, and in Section~\ref{sec:RegularTASEPexpo} for exponentially decaying rates, including the homogeneous rates  from \eqref{eq:RunningExample}.

\subsection{The regular tree with polynomially decaying rates}\label{sec:RegularTASEPpoly}

Consider the $d$-regular tree and homogeneous polynomial rates, i.e.\ we assume that we can find some  $p> 0$ such that the rates satisfy 
\begin{equation}\label{def:Polyweights}
 \frac{1}{j^p} = r_{j}^{\min} =  r_{j}^{\max}
\end{equation} for all $j \in \N$. For this choice of the rates, we want to show how the bounds in Theorem~\ref{thm:LinearCurrentL} on a time window can be improved. In the following, we will write
$a_n \sim b_n$ if $\lim_{n \to \infty} a_n (b_n)^{-1} = 1$.
Note that $\sD_n$ and $\sM_n$ from \eqref{def:FurthestGeneration} and \eqref{def:DecouplingPoint2} satisfy 
\begin{equation*}
\sD_n \sim \left(n^{2 + c_oc_{\textup{low}}}\log^{3} n \right)^{\frac{1}{p}} \quad \text{ and } \quad  \sM_n \sim \frac{d-1+\delta}{d-2}\min(\sD_n,n)
\end{equation*} 
for all $p>0$ and $\delta>0$; see Example \ref{ex:intro}. 
Recall that we are free in the choice of the sequence of generations $(\ell_n)_{n \in \N}$ with $\ell_n \geq \sM_n$ for all $n \in \N$ along which we observe the current created by the first $n$ particles  entering the tree. We assume that $(\ell_n)_{n \in \N}$ satisfies
\begin{equation}\label{eq:abBoundsMn}
\lim_{n \rightarrow \infty} \frac{\sM_n}{\ell_n^p}=a \, , \qquad \lim_{n\to \infty}\frac{n\sM_n^p}{\ell_n^{p+1}} = b 
\end{equation}
for some $a\in [0,1)$ and $b \in [0, \infty)$. 
We apply now Theorem \ref{thm:LinearCurrentL} in this setup.
\begin{proposition}\label{pro:RegularSpecial1} Consider the TASEP on the $d$-regular tree with polynomial weights from \eqref{def:Polyweights} for some $p>0$, and $a$ and $b$ as in \eqref{eq:abBoundsMn} for some sequence of generations $(\ell_n)_{n \in \N}$. Let $t_{\textup{up}}, t_{\textup{low}}$ be taken from  \eqref{def:LowerBoundTime1}, \eqref{def:LowerBoundTime2} and \eqref{def:UpperBoundTime}. For $a \in [0,1)$ and $b=0$,
\begin{equation}
\lim_{n \rightarrow \infty} \frac{t_{\textup{up}}}{t_{\textup{low}}} = \lim_{n \rightarrow \infty} t_{\textup{up}}\frac{(d-1)(1+p)}{(1-a)\ell_n^{p+1}} =1\, .
\end{equation} For $a \in [0,1)$ and $b \in (0,\infty)$, 
\begin{equation}
c \leq \liminf_{n \rightarrow \infty} \frac{t_{\textup{up}}}{t_{\textup{low}}} \leq  \limsup_{n \rightarrow \infty} \frac{t_{\textup{up}}}{t_{\textup{low}}} \leq c^{\prime}
\end{equation} holds for some constants $c,c^{\prime}>0$.
\end{proposition}

\begin{proof} Recall the notation from Section \ref{sec:ModelResults}. For $b\in (0,\infty)$, we observe that  for the above choice of transitions rates
\[
(\min_{\sM_n < \abs{x} \leq \ell_n} r_x) R^{\min}_{\sM_n,\ell_n} = r_{\ell_n} \sum_{k = \sM_n}^{\ell_n} \frac{1}{r_k} = \frac{1}{\ell_n^p} \sum_{k = \sM_n}^{\ell_n} k^{p} \sim \frac{1}{\ell_n^p} \int_{\sM_n}^{\ell_n} x^p \dif x \sim \frac{1-a}{1+p} \ell_n  
\]
holds, and hence $\theta=\infty$ in \eqref{def:thetaRoot}. Thus, we see that
\begin{align}\label{eq:Tupesti}
t_{\textup{up}}\sim 5(n(\sM_n)^p + (\sM_n)^{p+1}) +  \frac{1-a}{(d-1)(1+p)} \ell_n^{p+1} \, .
\end{align}
A similar computation for $b=0$ shows that $t_{\textup{up}}\sim (1-a)((d-1)(1+p))^{-1} \ell_n^{p+1}$ holds.
For the lower bound $t_{\textup{low}}$, we use that $t_{\textup{low}} \geq t_1^{\textup{low}}$ with $ t_1^{\textup{low}}$ in \eqref{def:LowerBoundTime1} to see that 
\begin{equation}\label{eq:Tlowesti}
t_{\textup{low}} \sim R^{\min}_{0,\ell_n} \sim \frac{1-a}{(d-1)(1+p)} \ell_n^{p+1}
\end{equation} holds. Therefore, combining \eqref{eq:Tupesti} and \eqref{eq:Tlowesti}, we obtain a sharp time window where we see a current of order $n$ when $b = 0$. We obtain the correct leading order for the time window to observe a current linear in $n$ in the case of  $0< b < \infty$.
\end{proof}

We conclude this section by discussing some examples of the sequence $(\ell_n)_{n \in \N}$ for the $d$-regular tree with $d \geq 3$, and rates which satisfy \eqref{def:Polyweights}, with the tree-TASEP starting from an all empty initial condition.

%For the case of the $d$-regular tree with polynomial rates from \ eqref{eq:PolynomialRatesExampleReg}, we have the following sharp result on the current and time window.
In the following examples, we take $\delta \to 0$ when estimating $\sM_n$ in Example \ref{ex:intro}. Moreover, because the rates decay polynomially, $c_{\textup{low}}$ can be taken to be arbitrarily close to $0$.    

\begin{example}\label{pro:PolyLinearCurrentFixedSPECIAL}  Fix some $c > \max(2, \frac{3}{1+p})$. Then for every $\delta^{\prime}>0$, we have that $\P$-almost surely \begin{equation*}
\lim_{n\to \infty} \cJ_{n^{c}}\Big(\frac{1-\delta^{\prime}}{(d-1)(1+p)} n^{c(1+p)} \Big) =0\, , \qquad  \liminf_{n\to \infty}\frac{1}{n}  \cJ_{n^{c}} \Big(  \frac{1+\delta^{\prime}}{(d-1)(1+p)} n^{c(1+p)}  \Big) >0.
\end{equation*} 
This is because the condition on $c$ guarantees $b = 0$. 
\end{example} 

\begin{example}
Let $p = 2(2+c_oc_{\textup{low}})$, and note that 
\begin{equation}
\sD_n \sim n^\frac{1}{2}\log^{3/p} n  \quad \text{ and } \quad  \sM_n \sim \frac{d-1}{d-2} n^\frac{1}{2}\log^{3/p} n
\end{equation} holds. Choosing $\ell_n= n^{(2+p)/(2+2p)}\log^{3/(1+p)} n$ for all $n\in \N$ yields that $a=0$ and $b\in (0,\infty)$ in \eqref{eq:abBoundsMn}. Hence, we can choose $t_{\textup{up}}$ and $t_{\textup{low}}$ in Proposition \ref{pro:RegularSpecial1} to be both of order $n^{1+p/2}\log^3n$.
\end{example}
\begin{example}
Let $p=1$, and note that 
\begin{equation}\label{eq:Example2}
\sD_n \sim n  \quad \text{ and } \quad  \sM_n \sim \frac{d-1}{d-2} n
\end{equation} holds. Choosing $\ell_n= n$ for all $n\in \N$ yields that $a\in (0,1)$ and $b \in (0,\infty)$ in \eqref{eq:abBoundsMn}. Hence, we can choose $t_{\textup{up}}$ and $t_{\textup{low}}$ in Proposition \ref{pro:RegularSpecial1} to be both of order $n^2$. 
\end{example}
\begin{example}
Let $p=\frac{1}{2}$ and $d\geq 4$. Then $\sD_n$ and $\sM_n$ satisfy \eqref{eq:Example2}. Choosing $\ell_n= n^2$ for all $n\in \N$ yields that $a=1/(d-2)$ and $b=0$ in \eqref{eq:abBoundsMn}. Hence, we can choose \begin{equation}
t_{\textup{up}} \sim t_{\textup{low}} \sim \frac{2(d-3)}{3(d-1)(d-2)}n^{3/2}
\end{equation} in Proposition \ref{pro:RegularSpecial1}. 
\end{example}
\begin{example}\label{last}
Let $p=\frac{3}{4}$ and $d\geq 3$. Then $\sD_n$ and $\sM_n$ satisfy \eqref{eq:Example2}. Choosing $\ell_n= n^2$ for all $n\in \N$ yields that $a=0$ and $b=0$ in \eqref{eq:abBoundsMn}. Hence, we can choose \begin{equation}
t_{\textup{up}} \sim t_{\textup{low}} \sim \frac{4}{7(d-1)}n^{7/2}
\end{equation} in Proposition \ref{pro:RegularSpecial1}. 
\end{example}

\subsection{The regular tree with exponentially decaying rates}\label{sec:RegularTASEPexpo}

We now study the $d$-regular tree with exponentially decaying rates, i.e.\ the rates satisfy
\begin{equation*}
\kappa \eup^{-c_{\textup{up}}\ell} = r_{\ell}^{\min} =  r_{\ell}^{\max}
\end{equation*} for all $\ell \in \N$ and some constants $\kappa,c_{\textup{up}}>0$.  In this setup, our goal is to improve the bounds on the window of generations in Theorem \ref{thm:CurrentfixedT}. Let  $(N_t)_{t\geq 0}$ be some integer sequence and assume that
\begin{equation}\label{eq:cexpconstruction}
\lim_{t \rightarrow \infty}\frac{\log N_t}{\log t} = c_{\exp}
\end{equation} holds for some $c_{\exp} \in [0,1)$. 
\begin{proposition}\label{pro:ExponentialRegularExample} Consider the TASEP on the $d$-regular tree with exponentially decaying rates, and fix some $\delta\in (0,1)$. We set
\begin{equation}\label{def:LupGen}
\tilde{L}_{\textup{up}}:= \left\lceil\frac{1}{c_{\textup{up}}}\log t\left( 1+ \log^{-\frac{1}{3}} t\right) \right\rceil \quad \text{ and } \quad \tilde{L}_{\textup{low}}:= \frac{1-\delta}{c_{\textup{up}}}\log t \, .
\end{equation} Then there exists some $C=C(\delta,c_{\textup{up}})>0$ such that if $c_{\exp}\leq C$, then
\begin{equation}\label{eq:StatementexpregularTree}
\lim_{t \rightarrow \infty}  J_{\tilde{L}_{\textup{up}}}(t) =0 \quad \text{ and } \quad \lim_{t \rightarrow \infty} \frac{1}{N_{t}} J_{\tilde{L}_{\textup{low}}}(t) = \infty \,  .
\end{equation}  
In particular, for $c_{\exp}=0$, we can choose $\tilde{L}_{\textup{up}}$ and $\tilde{L}_{\textup{low}}$ such that $\tilde{L}_{\textup{up}} \sim \tilde{L}_{\textup{low}}$ holds.
\end{proposition}
\begin{proof} We start with the lower bound $\tilde{L}_{\textup{low}}$. Observe that by Theorem \ref{thm:DisentanglementGWT}, there exists some $C=C(\delta,c_{\textup{up}})\in (0,1)$ such that the first $\lceil t^C\rceil$ particles are $\P$-almost surely disentangled at generation $L_{\textup{low}}$ for all $t$ sufficiently large. Since $J_{m}(t)$ is decreasing in the generation $m$, and $\sM_n$ is increasing in the number of particles $n$, we apply Theorem \ref{thm:CurrentfixedT} and Remark \ref{rem:Replacement} to conclude the second statement in \eqref{eq:StatementexpregularTree}. \\

For the first statement, we follow the proof of Theorem \ref{thm:CurrentfixedT}. It suffices to show that
\begin{equation}\label{eq:TheLastNumberedEquation}
\lim_{t \rightarrow \infty} 2\lambda t P_{\bT}\left( \sum_{i=0}^{\tilde{L}_{\textup{up}}} \frac{\om_i}{r_i^{\max}} < t \right) = 0
\end{equation}  holds, where $(\om_i)$ are independent Exponential-$1$-distributed random variables. Using Chebyshev's inequality, we obtain that
\begin{equation*}
P_{\bT}\left( \sum_{i=0}^{\tilde{L}_{\textup{up}}} \frac{\om_i}{r_i^{\max}} < t \right)   \leq  \exp\left( \tilde{L}_{\textup{up}} -\sum_{i=0}^{\tilde{L}_{\textup{up}}} \log\Big( 1+ \frac{\tilde{L}_{\textup{up}}}{t}\kappa \exp({c_{\textup{up}}i}) \Big) \right) \, .
\end{equation*} Since $\tilde{L}_{\textup{up}}t^{-1} \exp({c_{\textup{up}}i}) \geq 0$ holds for all $i\in \N$, we see that
\begin{equation*}
t P_{\bT}\left( \sum_{i=0}^{\tilde{L}_{\textup{up}}} \frac{\om_i}{r_i^{\max}} < t \right) \leq  \exp\left(\log t+ \tilde{L}_{\textup{up}} - \sum_{i=\lfloor\tilde{L}_{\textup{up}}-\sqrt{\tilde{L}_{\textup{up}}}\rfloor}^{\tilde{L}_{\textup{up}}} \big(c_{\textup{up}}i + \log(\kappa\tilde{L}_{\textup{up}})-\log t \big) \right)  \, .
\end{equation*} Plugging in the definition of $\tilde{L}_{\textup{up}}$ from \eqref{def:LupGen}, a computation shows that the right hand side  converges to $0$ when $t \rightarrow \infty$. This yields \eqref{eq:TheLastNumberedEquation}.
\end{proof}

We conclude this paragraph by revisiting the $d$-regular tree with homogeneous rates from \eqref{eq:RunningExample} in Section \ref{sec:ModelResults}.

\begin{proof}[Proof of Example \ref{pro:RunningExample}] Note that the first bound involving $L_{\textup{up}}$ follows immediately from Proposition \ref{pro:ExponentialRegularExample}. For the second bound involving $L_{\textup{low}}$,  note that we have $c_{\exp}=0$ for $N_t=t^{\alpha}$ to conclude.
\end{proof}

\section{Invariant distributions and blockage} \label{sec:LargeTimes}
In this section, our goal is to show Theorem \ref{thm:ConvergenceFlow}. The different parts of Theorem \ref{thm:ConvergenceFlow} will be shown in Propositions \ref{pro:ConvergenceFlow}, \ref{pro:Superflow}, \ref{pro:SuperflowVanishing} and \ref{pro:BlockageSubflow}, respectively. Let $\bT=(V,E,o) \in \mathcal{T}$ be a locally finite, rooted tree on which the TASEP is a Feller process with respect to a given family of rates $(r_{x,y})$. For a pair of probability measures $\pi, \tilde{\pi}$ on $\{0,1\}^{V}$, we say that $\tilde{\pi}$ \textbf{is stochastically dominated by} $\pi$   (and write $\tilde{\pi}\preceq \pi$), if 
 \begin{equation}
 \int f \dif \tilde{\pi} \leq  \int f \dif \pi
 \end{equation} holds for all functions $f$ which are increasing. Moreover, recall that for $\rho \in [0, 1]$, $\nu_{\rho}$ is the Bernoulli-$\rho$-product measure on $\{0, 1\}^V$ and that we consider the TASEP on $\bT$ with initial distribution $\nu_0$.
\begin{proposition}\label{pro:ConvergenceFlow} Let $(S_t)_{t \geq 0}$ be the semi-group of the TASEP $(\eta_t)_{t \geq 0}$
% with a reservoir, 
where particles are generated at  the root at rate $\lambda$ for some 
$\lambda > 0$. There exists a probability measure $\pi_\lambda$ on $\{ 0,1 \}^{V}$ such that
\begin{equation}\label{eq:ConvergenceFromAllEmpty}
\lim_{t \rightarrow \infty}\nu_{0}S_t = \pi_\lambda\, .
\end{equation}
In particular, $\pi_\lambda$ is a stationary measure for $(\eta_t)_{t \geq 0}$.
\end{proposition}
 In order to show Proposition \ref{pro:ConvergenceFlow}, we adopt a sequence of results from Liggett~\cite{L:ErgodicI}. Let $\bT_n$ denote the tree restricted to level $n$, where particles exit from the tree at $x\in \mathcal{Z}_n$ at rate $r_x$. %; see Figure \ref{fig:TreeTruncated}. 
For every $n$, let $\pi_\lambda^n$ denote the invariant distribution of the dynamics $(\eta_t^{n})_{t \geq 0}$ on $\bT_n$ with semi-group $(S^n_t)_{t \geq 0}$. We extend each measure $\pi_\lambda^n$ to a probability measure on $\{0,1\}^{V(\bT)}$ by taking the Dirac measure on $0$ for all sites $x \in V(\bT) \setminus V(\bT_n)$. 
\begin{lemma}[c.f.\  Proposition 3.7 in \cite{L:ErgodicI}] \label{lem:Couplinggeneralrates} For any initial distribution $\tilde{\pi}$,  the laws of the TASEPs $(\eta_t^{n})_{t \geq 0}$ and $(\eta_t^{n+1})_{t \geq 0}$ on $\bT_n$ and $\bT_{n+1}$, respectively, satisfy
\begin{equation}
\tilde{\pi}S_t^n = \P(\eta^n_t \in  \cdot) \preceq \P(\eta^{n+1}_t \in  \cdot) = \tilde{\pi}S_t^{n+1}
\end{equation} for all $t \geq 0$. In particular, $\pi_\lambda^{n} \preceq\pi_\lambda^{n+1}$ holds for all $n \in \N$.
\end{lemma} 
 \begin{proof} We follow the arguments in the proof of Theorem 2.13 in \cite{L:ErgodicI}. We note that for all $n\in \N$, the generators $\mathcal{L}_n$ and $\mathcal{L}_{n+1}$ of the TASEPs on $\bT_n$ and $\bT_{n+1}$ satisfy
 \begin{equation*}
 \mathcal{L}_{n+1}f(\eta) -  \mathcal{L}_{n}f(\eta) = \sum_{x \in \mathcal{Z}_n, y \in \mathcal{Z}_{n+1}} \left[ f(\eta^x)- f(\eta)\right] r_{x,y} \left( -\eta(x)\eta(y)\right) \geq 0
 \end{equation*} for any increasing function $f$ which does only depend on $V(\bT_n)$, for all $\eta \in \{ 0,1\}^{V(\bT)}$. Using the extension arguments from Theorem 2.3 and Theorem 2.11 in \cite{L:ErgodicI}, we obtain that
 \begin{equation}\label{eq:StochasticDominationLaws}
 \int f \dif \left[ \tilde{\pi}S^{n}_t\right] \leq  \int f \dif \left[ \tilde{\pi}S^{n+1}_t\right]
 \end{equation}  for any increasing function $f$ which only depends on $V(\bT_n)$, for all $t\geq 0$. It suffices now to show that \eqref{eq:StochasticDominationLaws} holds for all increasing functions $f$ which only depend on $V(\bT_{n+1})$. This follows verbatim the proof of Theorem 2.13 in \cite{L:ErgodicI} by decomposing $f$ according to its values on $V(\bT_{n+1}) \setminus V(\bT_{n})$.
 \end{proof}Lemma \ref{lem:Couplinggeneralrates} implies that the probability distribution $\pi_{\lambda}$ given by
\begin{equation}\label{def:LimitMeasure}
\pi_\lambda := \lim_{n \rightarrow \infty}\pi_\lambda^{n} 
\end{equation}
exists; see also Theorem 3.10 (a) in \cite{L:ErgodicI}. More precisely, Lemma \ref{lem:Couplinggeneralrates} guarantees  for every increasing cylinder function $f$ that
\begin{equation*}
\lim_{n \rightarrow \infty} \int f \dif \pi_{\lambda}^n = \int f \dif \pi_{\lambda}\, .
\end{equation*} Since the set of increasing functions is a determining class, \eqref{def:LimitMeasure} follows.
Furthermore, since $S^n_tf$ converges uniformly to $S_tf$ for any cylinder function $f$,  $\pi_{\lambda}$ is invariant for $(\eta_t)_{t \geq 0}$; see Proposition 2.2 and Theorem 4.1 in \cite{L:ErgodicI}. We now have all tools to show Proposition \ref{pro:ConvergenceFlow}.
\begin{proof}[Proof of Proposition \ref{pro:ConvergenceFlow}] Since we know that  $\pi_{\lambda}$ is invariant, we apply the  canonical coupling from  Lemma \ref{lem:CanonicalCoupling} to see that for all $t\geq 0$,
\begin{equation*}
 \nu_0S_t \preceq \pi_\lambda\, .
\end{equation*} Moreover, by Lemma \ref{lem:Couplinggeneralrates},  for all $t\geq 0$ and all $n \in \N$ 
\begin{equation*}
 \nu_0 S_t^n \preceq\nu_0 S_t \, .
\end{equation*} To prove Proposition \ref{pro:ConvergenceFlow}, it suffices  to show that
\begin{equation*}
\lim_{t \rightarrow \infty} \int f \dif\left[ \nu_0S_t \right] =  \int f \dif\pi_{\lambda}
\end{equation*} holds for any increasing cylinder function $f$. Combining the above observations
\begin{equation*}
\int f \dif \pi^n_{\lambda} = \liminf_{t \rightarrow \infty}\int f \dif\left[ \nu_0S^n_t \right] 
\leq \liminf_{t \rightarrow \infty} \int f \dif\left[ \nu_0S_t \right] 
\leq \limsup_{t \rightarrow \infty}\int f \dif\left[ \nu_0S_t \right] 
 \leq \int f \dif\pi_{\lambda} 
\end{equation*} holds for every $n \in \N$ and for any increasing cylinder function $f$. We conclude the proof recalling \eqref{def:LimitMeasure}; see also the proof of Lemma 4.3 in \cite{L:ErgodicI}.
\end{proof} 
Next, we show that if the rates satisfy a flow rule then there exists an invariant Bernoulli-$\rho$-product measure for some $\rho \in (0,1)$ for the TASEP on the tree; see Theorem 2.1 in \cite[Chapter VIII]{L:interacting-particle}.
\begin{lemma}\label{lem:InvarianceFlow} Let $\bT$ be a locally finite, rooted tree with rates satisfying a flow rule for a flow of strength $q$. Assume that  particles are generated at the root at rate $\lambda=\rho q $ for some $\rho \in (0,1)$. Then $\nu_{\rho}$ is an invariant measure for the TASEP $(\eta_t)_{t \geq 0}$ on $\bT$.
\end{lemma}
\begin{proof} We have to show that for all cylinder functions $f$,
\begin{equation*}
\int \mathcal{L} f \dif \nu_{\rho} = 0\, .
\end{equation*} 
Due to the linearity of $\mathcal{L}$, it suffices to consider $f$ of the form
\begin{equation}
f(\eta) = \prod_{x \in A} \eta(x) 
\end{equation}
with $\eta \in \{ 0,1\}^{V(\bT)}$ and $A$ some finite subset of $V(\bT)$. A calculation shows that 
if $o \notin A$,
\begin{equation}\label{eq:GenereatorEstimateProduct}
\int \mathcal{L} f \dif \nu_{\rho} = (1-\rho )\rho ^{\abs{A}}\sum_{x \in A, y \notin A}\left[ r_{y,x}-r_{x,y} \right] \, ;
\end{equation}
see also the proof of Theorem 2.1(a) in \cite[Chapter VIII]{L:interacting-particle}. Since a flow rule holds, the sum in \eqref{eq:GenereatorEstimateProduct} is zero. Similarly,  we obtain in the case $o \in A$
\begin{equation*}
\int \mathcal{L} f \dif \nu_{\rho} = (1-\rho )\rho ^{\abs{A}}\left(\sum_{x \in A, y \notin A}\left[ r_{y,x}-r_{x,y} \right] +\frac{\lambda}{\rho}\right) \, .
\end{equation*} 
We conclude using the flow rule, noting $r_o = q = \frac{\lambda}{\rho}$ and recalling the definition of $r_o$.
\end{proof}

\begin{remark}\label{rem:TASEPConstruction} Note that the measure $\nu_{1}$ is always invariant for the TASEP on trees. Theorem 1 of \cite{BL:ExclusionMeasures} shows that the TASEP on $\bT$ with a half-line attached to the root, where all edges point to the root, has an invariant Bernoulli-$\rho$- product measure with  $ \rho \in (0,1)$ if and only if a flow rule holds. If a flow rule holds, a similar argument as Theorem 1.17 in  \cite[Part III]{L:Book2} shows that $\nu_{\rho}$ is extremal invariant for all $\rho \in [0,1]$.
\end{remark}

%\section{Proof of positive current} \label{sec:LargeTimes}
\begin{figure}
\centering
\begin{tikzpicture}[scale=0.6]

		  \node[shape=circle,scale=1.2,draw,fill=white] (A1) at (0,0){} ;						
		\draw [->,line width=1] (-1.5,0) to [bend right,in=135,out=45] (A1);			
		
	\node[scale=1]  at (-0.85,0.7){$5$};					
	\node[scale=1]  at (0.7,1.15){$3$};							
	\node[scale=1]  at (1.3,-0.45){$2$};

	\node[scale=1]  at (3.3,1.8){$1$};							
	\node[scale=1]  at (3,0.5){$2$};	
	\node[scale=1]  at (3,2.45){$2$};	

	\node[scale=1]  at (3,-0.7){$1$};	
	\node[scale=1]  at (3,-2.3){$1$};	

\def\y{10}	
\def\z{5.5}		
		  \node[shape=circle,scale=1.2,draw,fill=white] (A1) at (\y,0){} ;						
		\draw [->,line width=1] (-1.5+\y,0) to [bend right,in=135,out=45] (A1);			
		
	\node[scale=1]  at (\y-0.85,0.7){$5$};					
	\node[scale=1]  at (\y+0.7,1.15){$3$};							
	\node[scale=1]  at (\y+1.3,-0.45){$2$};

	\node[scale=1]  at (\y+3.3,1.8){$1$};							
	\node[scale=1]  at (\y+3,0.5){$1$};	
	\node[scale=1]  at (\y+3,2.45){$1$};	

	\node[scale=1]  at (\y+3,-0.7){$1$};	
	\node[scale=1]  at (\y+3,-2.3){$1$};

		 \node[shape=circle,scale=1.2,draw,fill=white] (B1) at (\y+4+\z,1.5){} ;						
		\draw [->,line width=1] (\y+\z+2.5,1.5) to [bend right,in=135,out=45] (B1);		

	\node[scale=1]  at (\y+\z+4-0.85,1.5+0.7){$2$};	
	\node[scale=1]  at (\y+5+\z,0.5){$1$};							
	\node[scale=1]  at (\y+5+\z,2.45){$1$};

\foreach\x in{0,10,17.5}{

    \node[shape=circle,scale=1.2,draw,fill=white] (A1) at (\x,0){} ;
	
	\node[shape=circle,scale=1.2,draw,fill=white] (B1) at (2+\x,1.5){} ;
	\node[shape=circle,scale=1.2,draw,fill=white] (B2) at (2+\x,-1.5){} ;

	\node[shape=circle,scale=1.2,draw,fill=white] (C1) at (4+\x,2.5){} ;
	\node[shape=circle,scale=1.2,draw,fill=white] (C2) at (4+\x,1.5){} ;
	\node[shape=circle,scale=1.2,draw,fill=white] (C3) at (4+\x,0.5){} ;

	\node[shape=circle,scale=1.2,draw,fill=white] (C4) at (4+\x,-0.8){} ;
	\node[shape=circle,scale=1.2,draw,fill=white] (C5) at (4+\x,-2.2){} ;

	\draw[thick] (A1) to (B1);
	\draw[thick] (A1) to (B2);

	\draw[thick] (B1) to (C1);
	\draw[thick] (B1) to (C2);
	\draw[thick] (B1) to (C3);

	\draw[thick] (B2) to (C4);
	\draw[thick] (B2) to (C5);

	\draw[thick, dashed] (C1) to (4.8+\x,2.5);	
	\draw[thick, dashed] (C2) to (4.8+\x,1.5);	
	\draw[thick, dashed] (C3) to (4.8+\x,0.5);	
	\draw[thick, dashed] (C4) to (4.8+\x,-0.8);	
	\draw[thick, dashed] (C5) to (4.8+\x,-2.2);	
	
\node[scale=1]  at (\x+0.1,-0.6){$o$};

	\node[shape=circle,scale=0.9,draw,fill=red] (K1) at (A1){} ;
	\node[shape=circle,scale=0.9,draw,fill=red] (K2) at (C1){} ;
	\node[shape=circle,scale=0.9,draw,fill=red] (K3) at (C3){} ;
 %	\node[shape=circle,scale=0.9,draw,fill=red] (K4) at (C4){} ;
	\node[shape=circle,scale=0.9,draw,fill=red] (K5) at (B2){};

	}
%		\node[scale=0.9]  at (0,-2.4){$\sD_n=2$};
%		\node[scale=0.9]  at (3,-2.4){$3$};
%		\node[scale=0.9]  at (-3,-2.4){$1$};
%		\node[scale=0.9]  at (6,-2.4){$\sM_n=4$};
%		\node[scale=0.9]  at (-6,-2.4){$0$};	

	%	\node[scale=0.9]  at (0.3,-0.5){$o$};	
%		\node[scale=1,red]  at (-4.1,0.4){$\xi$};	
	%			\node[scale=0.9]  at (3.05,1.55){$y$};	
		%				\node[scale=0.9]  at (0.05,1.7){$x$};	

\node[scale=2]  at (6.75,0){``=''};	
\node[scale=1.7]  at (16.05,0){+};

\end{tikzpicture}
\caption[Visualization of the superflow decomposition used in Lemma \ref{lem:DominationSuperflow}]{\label{fig:Superflow}Visualization of the superflow decomposition used in Lemma \ref{lem:DominationSuperflow}. The superflow given at the left-hand side is decomposed into two flows of strengths $5$ and $2$, respectively, shown at the right-hand side. }
\end{figure}
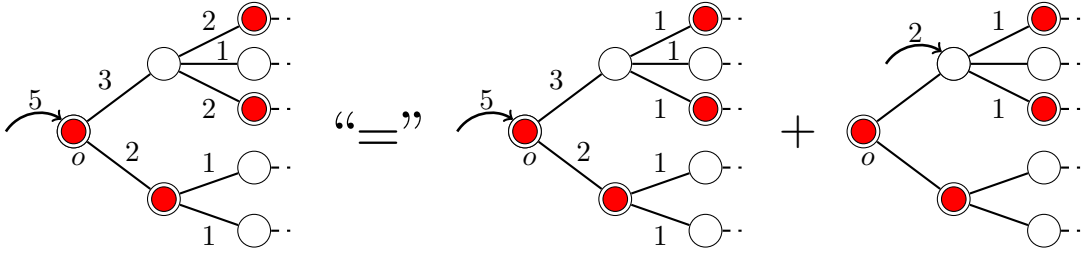
Next, we consider the case where the rates do not necessarily satisfy a flow rule. 
%In this case, we will also need to consider the %global behavior of the rates in order to study the %invariant measures and the current. 
In the following, we will without loss of generality assume that  $\lambda<q(o)$ holds. When $\lambda \geq q(o)$, the canonical coupling in Lemma \ref{lem:Couplinggeneralrates} yields that the current stochastically dominates the current of any TASEP with rate $\lambda^{\prime}$ for some $\lambda^{\prime} < q(o)$. We now characterize the behavior of the TASEP in the superflow case.
\begin{proposition}\label{pro:Superflow} Assume that a superflow rule holds.  Let $(\cJ_o(t))_{t \geq 0}$ be the current at the root for the TASEP on a tree $\bT$ with a reservoir of rate $\lambda=\rho q(o)$ at the root for some $\rho \in (0,1)$, and initial distribution $\nu_0$. Then the current $(\cJ_o(t))_{t \geq 0}$ through the root satisfies
\begin{equation}\label{eq:CurrentEstimateSuperflow}
\lim_{t \rightarrow \infty} \frac{\cJ_o(t)}{t} =  \lambda \pi_{\lambda}(\eta(o)=0) \geq q(o) \rho(1-\rho)
\end{equation} almost surely, where $\pi_{\lambda}$ is given by \eqref{eq:ConvergenceFromAllEmpty}.
% If in addition we have a global superflow rule, see \eqref{def:GlobalSuperflow}, then
%\begin{equation}\label{eq:VanishingDensitySuperflow}
% \lim_{m \rightarrow \infty}\frac{1}{\abs{\mathcal{Z}_m}}\left( \sum_{x \in \mathcal{Z}_m} \pi_{\lambda}(\eta(x)=1) \right)  = 0 \, ,
%\end{equation}  holds, i.e.\ the average density of particles vanishes when looking deep enough into the tree.
\end{proposition}
In order to prove Proposition \ref{pro:Superflow}, we will use the following lemma, which shows that the law of the  TASEP on trees is always dominated by a certain Bernoulli-$\rho$-product measure on the tree.

\begin{lemma}\label{lem:DominationSuperflow} Assume that the rates satisfy a superflow rule and consider the TASEP $(\eta_t)_{t \geq 0}$ with a reservoir of rate $\lambda=\rho q(o)$ for some $\rho \in (0,1)$.  If $\P(\eta_0 \in \cdot) \preceq \nu_\rho$ holds, then
\begin{equation}\label{eq:StochasticDominationBernoulli}
\P(\eta_t \in \cdot ) \preceq \nu_{\rho}
\end{equation} for all $t\geq 0$. In particular, the measure $\pi_{\lambda}$ from \eqref{eq:ConvergenceFromAllEmpty} satisfies $\pi_{\lambda} \preceq \nu_{\rho}$.
\end{lemma}
\begin{proof} In order to show \eqref{eq:StochasticDominationBernoulli}, we decompose the rates satisfying a superflow rule into flows starting at different sites. More precisely, we claim that there exists a family of transition rates $((r^{z}_{x,y})_{(x,y) \in E(\bT)})_{z \in V(\bT)}$ with the following two properties. For every $i \in V(\bT)$ fixed, the rates $(r^{z}_{x,y})_{(x,y) \in E(\bT)}$ satisfy a flow rule for a flow of strength $q(z)$ for the tree rooted in $z$. Moreover, for all $(x,y) \in E(\bT)$,
\begin{equation*}
\sum_{z \in V(\bT)} r^{z}_{x,y} = r_{x,y} \, ;
\end{equation*} see also Figure \ref{fig:Superflow}. We construct such a family of transition rates as follows. We start with the root $o$ and choose a set of rates $(r^{o}_{x,y})_{(x,y) \in E(\bT)}$ according to an arbitrary rule such that the rates satisfy a flow rule for a flow of strength $q(o)$ starting at $o$, and $r^{o}_{x,y} \leq r_{x,y}$ for all $(x,y) \in E(\bT)$. Next, we consider the neighbors of $o$ in the tree. For every $z \in V(\bT)$ with $|z|=1$, we choose a set of rates $(r^{z}_{x,y})_{(x,y) \in E(\bT)}$ according to an arbitrary rule such that the rates satisfy a flow rule for a flow of strength $q(z)$ starting at $z$. Moreover, we require that
\begin{equation*}
r^{z}_{x,y} \leq r_{x,y} - r^{o}_{x,y}
\end{equation*} holds for all $(x,y) \in E(\bT)$. The existence of the flow is guaranteed by the superflow rule. More precisely, we use the following observation. Whenever the rates satisfy a superflow rule, we can treat the rates as maximal capacities and find a flow $(r^{o}_{x,y})$ of strength $q(o)$ which does not exceed these capacities. Note that the  reduced rates $(r_{x,y} - r^{o}_{x,y})$ must again satisfy a superflow rule, but now on the connected components of the graph with vertex set $V(\bT) \setminus \{ o\}$. This is due to the fact that the net flow vanishes on all sites $V(\bT) \setminus \{ o\}$. We then iterate this procedure to obtain the claim.  

Let $(\tilde{\eta}_t)_{t \geq 0}$ be the exclusion process with rates $(r_{x,y})_{(x,y) \in E(\bT)}$, where in addition, we create particles at every site $x \in V(\bT)$ at rate $q(x)\rho$. Due to the above decomposition of the rates and Lemma \ref{lem:InvarianceFlow}, we claim that the measure $\nu_{\rho}$ is invariant for $(\tilde{\eta}_t)_{t \geq 0}$. To see this, we define a family of generators $(\mathcal{L}_z)_{z \in V(\bT)}$ on the state space $\{ 0,1\}^{V(\bT_z)}$. Here, the trees $\bT_z$  are the subtrees of $\bT$ rooted at $z$, consisting of all sites which can be reached from site $z$ using a directed path. For all cylinder functions $f$, we set 
\begin{equation*}
\mathcal{L}_z f(\eta) = \rho q(z)(1-\eta(z))[f(\eta^{z})-f(\eta)]+\sum_{(x,y) \in E(\bT_z)} r^{z}_{x,y} (1-\eta(y))\eta(x)[f(\eta^{x,y})-f(\eta)] 
\end{equation*} and thus by Lemma \ref{lem:InvarianceFlow}
\begin{equation}\label{def:IndividualGenerators}
\int \mathcal{L}_z f(\eta)  \dif \nu_{\rho} = 0
\end{equation} holds. Note that the generator $\tilde{\mathcal{L}}$ of the process $(\tilde{\eta}_t)_{t \geq 0}$ satisfies
\begin{equation}\label{eq:generatorDecomp}
\tilde{\mathcal{L}} f(\eta)  = \sum_{z \in V(\bT)} \mathcal{L}_z f(\eta)  
\end{equation} for all cylinder functions $f$ on $\{ 0,1 \}^{V(\bT)}$, and that at most finitely many terms in the sum in \eqref{eq:generatorDecomp} are non-zero since $f$ is a cylinder function. Hence, we obtain that $\nu_{\rho}$ is an invariant measure of $(\tilde{\eta}_t)_{t \geq 0}$ by combining \eqref{def:IndividualGenerators} and \eqref{eq:generatorDecomp}. Using Remark \ref{rem:AdditionalSources}, we see that the  canonical coupling $\mathbf{P}$ for the TASEP on trees satisfies
\begin{equation*}
\mathbf{P}\left( \eta_t  \preceq \tilde{\eta}_t  \text{ for all } t \geq 0  \mid   \eta_0  \preceq \tilde{\eta}_0\right) = 1\,  .
\end{equation*} Thus, we let $(\tilde{\eta}_t)_{t \geq 0}$ be started from $\nu_\rho$ and conclude using Strassen's theorem \cite{S:StochasticDomination}.
\end{proof}
\begin{proof}[Proof of Proposition \ref{pro:Superflow}] Combining Proposition \ref{pro:ConvergenceFlow}, Remark \ref{rem:TASEPConstruction}, and  Lemma \ref{lem:DominationSuperflow}, we obtain \eqref{eq:CurrentEstimateSuperflow} by applying the ergodic theorem for Markov processes.
%In order to show \eqref{eq:VanishingDensitySuperflow}, note that for any generation $m$, the current $(J_m(t))_{t \geq 0}$ through this generation must satisfy \eqref{eq:CurrentEstimateSuperflow} as well. Moreover, using Lemma \ref{lem:DominationSuperflow} for the sites at generation $m$, we see that
%\begin{equation}
%\pi_\lambda(\eta(x)=1, \eta(y)= 0)  \geq (1-\rho) \pi_\lambda(\eta(x)=1) \
%\end{equation} holds for all $x\in \mathcal{Z}_{m-1}$ and $y\in \mathcal{Z}_{m}$ with $x \sim y$. Together with the definition of the current $(J_m(t))_{t \geq 0}$ at generation $m$, we see that
%\begin{equation}
%r^{\min}_m  \sum_{x \in \mathcal{Z}_m} \pi_{\lambda}(\eta(x)=1) \leq q(o)
%\end{equation} holds and thus, we conclude  \eqref{eq:VanishingDensitySuperflow} using the assumption that $r^{\min}_m \abs{\mathcal{Z}_m} \rightarrow \infty$. 
\end{proof}
%We now apply Lemma \ref{lem:DominationSuperflow} to show a fan behavior of $\pi_{\lambda}$. %when the rates decay too slowly compared to the growth of the underlying tree. 
\begin{proposition} \label{pro:SuperflowVanishing}
Consider the TASEP $(\eta_t)_{t \geq 0}$ on the tree $T=(V,E)$ for some $\lambda=\rho q(o)>0$ with $\rho\in (0,1)$.  Moreover, assume that a superflow rule holds and that \eqref{eq:SuperflowGrowthConditionSingle} is satisfied. Then the measure $\pi_{\lambda}$ from Proposition \ref{pro:ConvergenceFlow} satisfies
\begin{equation}\label{eq:Fanbehaviour}
\lim_{n \rightarrow \infty} \frac{1}{\abs{\mathcal{Z}_n}} \sum_{x \in \mathcal{Z}_n}\pi_\lambda(\eta(x)=1) = 0 \, .
\end{equation}
\end{proposition} 
\begin{proof} Note that \eqref{eq:SuperflowGrowthConditionSingle} is equivalent to assuming 
\begin{equation}\label{eq:SuperflowGrowthCondition}
\lim_{n \rightarrow \infty} \abs{\mathcal{Z}_n} \min_{\substack{(x,y) \in \Er \\ |x| \in [n,n+m]}} r_{x,y} = \infty \, .
\end{equation} for any $m\geq 0$ fixed. Moreover, note that 
\begin{equation}\label{eq:CompareDifferentGenerations}
\cJ_o(t)-\cJ_n(t) \leq \sum_{i \in [n]} \abs{\mathcal{Z}_i}
\end{equation} 
 for any $n\in \N$ and $t \geq 0$. Using Proposition \ref{pro:ConvergenceFlow}, we see that
\begin{align*}
\lambda \geq \lim_{t \rightarrow \infty} \frac{\cJ_{o}(t)}{t} = \lim_{t \rightarrow \infty} \frac{\cJ_{n}(t)}{t} = 
\sum_{x \in \mathcal{Z}_n \colon (x,y) \in \Er }\pi_\lambda(\eta(x)=1, \eta(y)=0) r_{x,y} 
\end{align*} 
holds for all $n\in \N_0$. In particular, for all $n,m \in \N_0$ 
\begin{equation}\label{eq:EstimatePi1}
\sum_{|x| \in [n,n+m]}\sum_{(x,y) \in \Er }\pi_\lambda(\eta(x)=1, \eta(y)=0) \leq m \lambda  \Big(\min_{\substack{(x,y) \in \Er \\ |x| \in [n,n+m]}} r_{x,y} \Big)^{-1}\, .
\end{equation} Let $\delta>0$ be arbitrary and fix some $m\in \N$ such that $\rho^m \leq \frac{\delta}{2}$. Moreover, for all $x\in \mathcal{Z}_n$, fix a sequence of sites $(x=x_1,x_2,\dots,x_m)$ with $(x_i,x_{i+1}) \in \Er$ for all $i \in [m-1]$. Note that the sites $(x_i)_{i \in [m]}$ are disjoint for different $x \in \mathcal{Z}_n$ and that by Lemma \ref{lem:DominationSuperflow}
\begin{equation} \label{eq:EstimatePi2}
\pi_{\lambda}(\eta(x_i)=1 \text{ for all } i \in [m]) \leq \delta/2
\end{equation} 
for all $x \in \mathcal{Z}_n$. For $x \in \mathcal{Z}_n$, we decompose according to the value on $(x_i)_{i \in [m]}$ to get
\begin{equation*} 
\sum_{x \in \mathcal{Z}_n}\pi_\lambda(\eta(x)=1) \leq \sum_{x \in \mathcal{Z}_n}\pi_{\lambda}(\eta(x_i)=1 \ \forall i \in [m]) + \sum_{\substack{(x,y) \in \Er \\ |x| \in [n,n+m]}} \pi_\lambda(\eta(x)=1, \eta(y)=0)\, .
\end{equation*} 
Hence, combining  \eqref{eq:SuperflowGrowthCondition}, \eqref{eq:EstimatePi1} and \eqref{eq:EstimatePi2}, we see that for all $n$ sufficiently large,
\begin{equation*}
\sum_{x \in \mathcal{Z}_n}\pi_\lambda(\eta(x)=1) \leq \frac{\delta}{2}\abs{\mathcal{Z}_n} + m \lambda  \Big(\min_{\substack{(x,y) \in \Er \\ |x| \in [n,n+m]}} r_{x,y} \Big)^{-1} \leq \delta\abs{\mathcal{Z}_n}\, .
\end{equation*} Since $\delta>0$ was arbitrary, we conclude.
\end{proof}

We use a similar argument  to determine when we have a positive averaged density.
\begin{corollary} Suppose that a superflow rule holds. Consider the TASEP $(\eta_t)_{t \geq 0}$ on the tree $T=(V,E)$ for some $\lambda=\rho q(o)>0$ with $\rho\in (0,1)$. Moreover, assume that $T$ has maximum degree $\Delta$, and that
\begin{equation}\label{eq:correctRates}
\limsup_{n \rightarrow \infty} \abs{\mathcal{Z}_n} \min_{(x,y) \in \Er, x \in \mathcal{Z}_n} r_{x,y} \leq c
\end{equation} holds for some constant $c>0$. Then 
\begin{equation}
\liminf_{n \rightarrow \infty} \frac{1}{\abs{\mathcal{Z}_n}} \sum_{x \in \mathcal{Z}_n}\pi_\lambda(\eta(x)=1) > 0 \, .
\end{equation}
\end{corollary}
\begin{proof} Observe that for every $x \in \mathcal{Z}_n$ and $n \in \N$, we can choose a neighbor $y \in \mathcal{Z}_{n+1}$ of $x$ such that 
\begin{equation*}
\frac{1}{\Delta}\limsup_{t \rightarrow \infty} \frac{\cJ_x(t)}{t} \leq  \limsup_{t \rightarrow \infty} \frac{\cJ_y(t)}{t} =  \pi_\lambda(\eta(x)=1,\eta(y)=0) r_{x,y}
\end{equation*} holds. Together with \eqref{eq:correctRates}
\begin{equation*}
\sum_{x \in \mathcal{Z}_n}\pi_\lambda(\eta(x)=1) \geq   \sum_{x \in \mathcal{Z}_n}\frac{1}{\Delta r_{x,y}}\limsup_{t \rightarrow \infty} \frac{\cJ_x(t)}{t} \geq 
\frac{1}{c\Delta}\abs{ \mathcal{Z}_n} \limsup_{t \rightarrow \infty} \frac{\cJ_n(t)}{t} \, .
\end{equation*} Since the rates satisfy a superflow rule, we conclude by applying Proposition \ref{pro:Superflow}.
\end{proof}

%\section{Proof of blockage}\label{sec:BlockageTASEP}

Next, we consider the case where the rates in the tree decay too fast, i.e.\ when a subflow rule holds; see \eqref{def:GlobalSubflow}. We show that the current is sublinear.
\begin{proposition} \label{pro:BlockageSubflow} Suppose that the rates satisfy a subflow rule. Then the current $(\cJ_o(t))_{t \geq 0}$ of the TASEP $(\eta_t)_{t \geq 0}$ on a tree $T=(V,E)$ with a reservoir of rate $\lambda>0$ satisfies
\begin{equation}\label{eq:BlockageSubflow}
\lim_{t \rightarrow \infty} \frac{\cJ_o(t)}{t}=0
\end{equation} almost surely. Moreover,  the limit measure $\pi_{\lambda}$ of Lemma \ref{pro:ConvergenceFlow} is the Dirac measure $\nu_1$. In particular, $(\eta_t)_{t \geq 0}$ has a unique invariant measure.
\end{proposition}
\begin{proof} By \eqref{eq:CompareDifferentGenerations}, it suffices for  \eqref{eq:BlockageSubflow} to prove that for every $\varepsilon>0$, there exists some $m=m(\varepsilon)$ such that the aggregated current $(\cJ_m(t))_{t \geq 0}$ at generation $m$ satisfies
\begin{equation*}
\limsup_{t \rightarrow \infty} \frac{\cJ_m(t)}{t} \leq \varepsilon \,  .
\end{equation*} Recall $r_x$ from \eqref{r_xdef} for all $x\in V$, and let $(X^{x}_t)_{t \geq 0}$ be a rate $r_x$ Poisson clock, indicating how often the clock of an outgoing edge from $x$ rang until time $t$.  In order to bound $(\cJ_m(t))_{t \geq 0}$, recall that we start with all sites being empty, and observe that the current can only increase by one if a clock at an edge connecting level $m-1$ to level $m$ rings. Thus, we see that
\begin{equation*}
0 \leq \limsup_{t \rightarrow \infty} \frac{\cJ_m(t)}{t} \leq \limsup_{t \rightarrow \infty} \frac{1}{t}\sum_{x \in \mathcal{Z}_{m-1}} X^{x}_t = \sum_{x \in \mathcal{Z}_{m-1}} r_x
\end{equation*} holds almost surely. Using the subflow rule, we can choose $m=m(\varepsilon)$ sufficiently large to conclude \eqref{eq:BlockageSubflow}. To prove that $\pi_{\lambda}$ is the Dirac measure on all sites being occupied, use Proposition \ref{pro:Superflow} to see that \eqref{eq:BlockageSubflow} holds if and only if $\pi_{\lambda}(\eta(o)=1) \in \{ 0,1\}$. Since the rate $\lambda$ at which particles are generated is strictly positive and $\pi_{\lambda}$ is an invariant measure, we conclude that $\pi_{\lambda}(\eta(o)=1) =1$. Using the ergodic theorem, we see that  almost surely for all neighbors $z$ of $o$,
\begin{equation*}
\pi_{\lambda}(\eta(o)=1,\eta(z)=0)r_{o,z}=\lim_{t \rightarrow \infty} \frac{\cJ_z(t)}{t} \leq \lim_{t \rightarrow \infty} \frac{\cJ_o(t)}{t} = 0\, .
\end{equation*} Hence, we obtain that $\pi_{\lambda}(\eta(z)=1) =1$  holds for all $z \in V$ with $|z|=1$ as well. We iterate this argument to conclude. 
\end{proof}

\section{Open problems} \label{sec:OpenProbs}

We saw that under certain assumption on the rates, the first $n$ particles in the TASEP will eventually disentangle and will continue to move as independent random walks. Intuitively, one expects for small times that the particles in the exclusion process block each other.
% and hence force the particles not to follow their predecessors. 
This raises the following question.

\begin{question} Consider the TASEP $({\eta}_{t})_{t \geq 0}$ on $\bT$ started from the all empty configuration. Let $(\tilde{\eta}_{t})_{t \geq 0}$ be the dynamics on $\bT$ where we start $n$ independent random walks at the root. Let $p_{n,\ell}$ and $\tilde p_{n,\ell}$ denote the $P_{\bT}$-probability that the first $n$ particles are disentangled at level $\ell$ in $({\eta}_{t})_{t \geq 0}$ and $({\tilde \eta}_{t})_{t \geq 0}$, respectively.
Does
$\tilde p_{n,\ell} \leq p_{n,\ell}$ hold for all $\ell,n \in \N$?
\end{question}
It is not hard to see that this is true for $n=2$. However, already the case $n=3$ is not clear (at least not to us).
The next question is about the behaviour of the current.

\begin{question}
What can we say about the order of the current and its fluctuations in Theorem \ref{thm:LinearCurrentL} and Theorem \ref{thm:CurrentfixedT}?
\end{question}

The last open problem concerns the properties of the equilibrium measure $\pi_{\lambda}$ from Theorem \ref{thm:ConvergenceFlow}. In Lemma \ref{lem:DominationSuperflow}, we saw that $\pi_{\lambda}$ is stochastically dominated by some Bernoulli product measure. In analogy to the TASEP on the half-line; see Lemma 4.3 in \cite{L:ErgodicI}, we expect the following behaviour of $\pi_{\lambda}$.
\begin{conjecture} Consider TASEP with a reservoir of rate $\lambda=\rho q$ for some $\rho \in (0,\infty)$ such that a flow rule holds for some flow of strength $q$. Recall $\pi_{\lambda}$ from \eqref{eq:ConvergenceFromAllEmpty}. Then for $\rho\leq \frac{1}{2}$, we have that $\pi_{\lambda}=\nu_{\rho}$. For $\rho > \frac{1}{2}$, it holds that
\begin{equation}
\lim_{\abs{x} \rightarrow \infty}\pi_{\lambda}(\eta(x)=1) = 
\frac{1}{2}  \, .
\end{equation} 
\end{conjecture}

\subsection*{Acknowledgment} N.\ Georgiou acknowledges partial support from EPSRC First Grant EP/P021409/1: ``The flat edge in last passage percolation" for the initial development of this project and by “The Dr Perry James (Jim) Browne Research Centre on Mathematics and its Applications" individual grant. D.\ Schmid thanks the Studienstiftung des deutschen Volkes and the TopMath program for financial support.
The  research  was  partly  carried  out  during  visits  of  N.\ Gantert and D.\ Schmid to the University of Sussex and
during visits of N.\ Georgiou to the Technical University of Munich.
Grateful acknowledgement is made for hospitality to both universities.  N.\ Georgiou and D.\ Schmid thank the Mathematisches Forschungsinstitut Oberwolfach for hospitality and support in the final stages of the project. Moreover, we thank Lukas Andreson for pointing out several inaccuracies in the published version. Finally, we thank the anonymous referees and associate editor for their comments.

\bibliographystyle{plain}

\bibliography{TASEP}

\vspace{-0.05cm}

\end{document}